\documentclass[12pt]{article}
\usepackage{enumerate}
\usepackage[OT1]{fontenc}
\usepackage{amsmath,amssymb}
\usepackage{natbib}
\usepackage[usenames]{color}

\usepackage{smile}
\usepackage{natbib}
\usepackage{multirow}

\usepackage[colorlinks,
            linkcolor=red,
            anchorcolor=blue,
            citecolor=blue,
            breaklinks=true
            ]{hyperref}

\usepackage{enumitem}

\makeatletter
\newcommand{\vast}{\bBigg@{4}}
\newcommand{\Vast}{\bBigg@{5}}
\makeatother

\def\given{\,|\,}

\def\tr{\mathop{\text{tr}}\kern.2ex}

\long\def\comment#1{}

\def\tr{\mathop{\text{tr}}}

\def\svec{\text{svec}}
\def\smat{\text{smat}}
\def\F{\text{F}}
\def\gap{\text{gap}}
\def\TV{\text{TV}}

\newcommand{\bel}{\begin{eqnarray}\label}
\newcommand{\eel}{\end{eqnarray}}
\newcommand{\bes}{\begin{eqnarray*}}
\newcommand{\ees}{\end{eqnarray*}}

\newcommand{\red}{\color{red}}

\newcommand{\la}{\langle}
\newcommand{\ra}{\rangle}

\newcommand{\normmm}[1]{{\left\vert\kern-0.25ex\left\vert\kern-0.25ex\left\vert #1 
    \right\vert\kern-0.25ex\right\vert\kern-0.25ex\right\vert}}

\usepackage{mathrsfs}

\usepackage{fullpage}

\def \poly {\text{poly}}

\ifx\problem\undefined
\newtheorem{problem}[theorem]{Problem}
\fi

\def\##1\#{\begin{align}#1\end{align}}
\def\$#1\${\begin{align*}#1\end{align*}}

\newcommand*\samethanks[1][\value{footnote}]{\footnotemark[#1]}

\usepackage{authblk}

\date{\vspace{-5ex}}

\usepackage{breakcites}

\begin{document}

\title{\huge Actor-Critic Provably Finds Nash Equilibria of Linear-Quadratic Mean-Field Games}

\author{Zuyue Fu\thanks{Department of Industrial Engineering and Management Sciences, Northwestern University}~~~~ Zhuoran Yang\thanks{Department of Operations Research and Financial Engineering, Princeton University}~~~~ Yongxin Chen\thanks{School of Aerospace Engineering, Georgia Institute of Technology}~~~~ Zhaoran Wang\samethanks[1]}

\maketitle


\begin{abstract}
We study discrete-time mean-field Markov games with infinite numbers of agents where each agent aims to minimize its ergodic cost.   We consider the setting where the agents have identical linear state transitions and quadratic cost functions, while the aggregated effect of the agents is captured by the population mean of their states, namely, the mean-field state. For such a game, based on the Nash certainty equivalence principle, we provide sufficient conditions for the existence and uniqueness of its Nash equilibrium. Moreover, to find the Nash equilibrium, we propose a mean-field actor-critic algorithm with linear function approximation, which does not require knowing the model of dynamics. Specifically, at each iteration of our algorithm, we use the single-agent actor-critic algorithm to approximately obtain the optimal policy of the each agent given the current mean-field state, and then update the mean-field state. In particular,  we prove that our algorithm converges to the Nash equilibrium at a linear rate. To the best of our knowledge, this is the first success of applying model-free reinforcement learning with function approximation to discrete-time mean-field Markov games with provable non-asymptotic global convergence guarantees. 
\end{abstract}

\section{Introduction}
In reinforcement learning (RL) \citep{sutton2018reinforcement}, an agent learns to make decisions that minimize its expected total cost through sequential interactions with the environment.  
Multi-agent reinforcement learning (MARL) \citep{shoham2003multi,shoham2007if,busoniu2008comprehensive} aims to extend 
RL to sequential decision-making problems involving multiple agents.     
In a non-cooperative game, we are interested in the Nash equilibrium \citep{nash1951non}, which is a joint policy of all the agents such that each agent cannot decrease its expected total cost by unilaterally deviating from its Nash policy. 
The Nash equilibrium plays a critical role in understanding the social dynamics of self-interested agents \citep{ash2000social, axtell2002non} and constructing the optimal policy of a particular agent via fictitious self-play \citep{bowling2000analysis, ganzfried2009computing}. 
In the presence of the recent  development in deep learning \citep{lecun2015deep},  MARL with function approximation achieves tremendous empirical successes
in applications, including Go \citep{silver2016mastering,silver2017mastering}, Dota \citep{OpenAI_dota}, Star Craft \citep{vinyals2019alphastar}, Poker  \citep{heinrich2016deep,moravvcik2017deepstack},  multi-robotic systems \citep{yang2004multiagent}, autonomous driving \citep{shalev2016safe}, and solving social dilemmas \citep{de2006learning, leibo2017multi, hughes2018inequity}. 
However, since the capacity of the joint state and action spaces grows exponentially in  the number of agents,  such MARL approaches become computationally intractable when the number of agents is large, which is common in real-world applications \citep{sandholm2010population,calderone2017models,wang2017display}. 

Mean-field game is proposed by \cite{huang2003individual, huang2006large, lasry2006jeux, lasry2006jeux2, lasry2007mean} with the idea of utilizing mean-field approximation to model the strategic interactions within a large population. In a mean-field game, each agent has the same cost function and state transition, which depend on the other agents only through their aggregated effect.
As a result, the optimal policy of each agent depends solely on its own state and the aggregated effect of the population, and such an optimal policy is symmetric across all the agents. 
Moreover, if the aggregated effect of the population corresponds to the Nash equilibrium, then the optimal policy of each agent jointly constitutes a Nash equilibrium. 
Although such a Nash equilibrium corresponds to an infinite number of agents, it well approximates the Nash equilibrium for a sufficiently large number of agents \citep{bensoussan2016linear}.  
Also, as the aggregated effect of the population abstracts away the strategic interactions between individual agents, it circumvents the computational intractability of the MARL approaches that do not exploit symmetry. 

However, most existing work on mean-field games focuses on characterizing the existence and uniqueness of the Nash equilibrium rather than designing provably efficient algorithms. In particular, most existing work considers the continuous-time setting, which requires solving a pair of Hamilton-Jacobi-Bellman (HJB) and Fokker-Planck (FP) equations, whereas the discrete-time setting  is more common in practice, e.g., in the aforementioned applications. Moreover, most existing approaches, including the ones based on solving the HJB and FP equations, require knowing the model of dynamics \citep{bardi2014linear}, or having the access to a simulator, which generates the next state given any state-action pair and aggregated effect of the population \citep{guo2019learning}, which is often unavailable in practice.  

 

To address these challenges, we develop an efficient model-free RL approach to mean-field game, which provably attains the Nash equilibrium.  In particular, we focus on discrete-time mean-field games with linear state transitions and quadratic cost functions, where the aggregated effect of the population is quantified by the mean-field state. Such games capture the fundamental difficulties of general mean-field games and well approximates a variety of real-world systems such as power grids \citep{minciardi2011optimal}, swarm robots \citep{fang2014lqr, araki2017multi, doerr2018random}, and financial systems \citep{zhou2000continuous,huang2018linear}.  
In detail, based on the Nash certainty equivalence (NCE) principle \citep{huang2006large, huang2007large}, we propose a mean-field actor-critic algorithm which, at each iteration, given the mean-field state $\mu$, approximately attains the optimal policy $\pi_\mu^*$ of each agent, and then updates the mean-field state $\mu$ assuming that all the agents follow $\pi_\mu^*$.        We parametrize the actor and critic by linear and quadratic functions, respectively, and prove that such a parameterization encompasses the optimal policy of each agent. Specifically, we update the actor parameter using policy gradient \citep{sutton2000policy} and natural policy gradient \citep{kakade2002natural, peters2008natural, bhatnagar2009natural} and update the critic parameter using primal-dual gradient temporal difference \citep{sutton2009fast, sutton2009convergent}. 
In particular, we prove that given the mean-field state $\mu$, the sequence of policies generated by the actor converges linearly to the optimal policy $\pi_{\mu}^*$. 
Moreover, when alternatingly update the policy and mean-field state, we prove that the sequence of policies and its corresponding sequence of mean-field states converge to the unique Nash equilibrium at a linear rate. 
Our approach can be interpreted from both ``passive'' and ``active'' perspectives: (i) Assuming that each self-interested agent employs the single-agent actor-critic algorithm, the policy of each agent converges to the unique Nash policy, which characterizes the social dynamics of a large population of model-free RL agents. (ii) For a particular agent, our approach serves as a fictitious self-play method for it to find its Nash policy, assuming the other agents give their best responses. To the best of our knowledge, our work establishes the first efficient model-free RL approach with function approximation that provably attains the Nash equilibrium of a discrete-time mean-field game. 
As a byproduct, we also show that the sequence of policies generated by the single-agent actor-critic algorithm converges at a linear rate to the optimal policy of a linear-quadratic regulator (LQR) problem in the presence of drift, which may be of independent interest.

\vspace{5pt} 

{\noindent \bf Related Work.}
Mean-field game is first introduced in \cite{huang2003individual, huang2006large, lasry2006jeux, lasry2006jeux2, lasry2007mean}.  In the last decade, there is growing interest in understanding continuous-time mean-field games. See, e.g., \cite{gueant2011mean, bensoussan2013mean, gomes2014mean,  carmona2013probabilistic, carmona2018probabilistic} and the references therein. 
Due to their simple structures,  continuous-time  linear-quadratic mean-field games are extensively studied under various model assumptions. See  
\cite{li2008asymptotically, bardi2011explicit,wang2012mean, bardi2014linear, huang2016mean, huang2016backward, bensoussan2016linear, bensoussan2017linear, caines2017epsilon,  huang2017robust, moon2018linear, huang2019linear} for examples of this line of work. Meanwhile, the literature on discrete-time linear-quadratic mean-field games remains relatively scarce. Most of this line of work focuses on characterizing the existence of a Nash equilibrium and the behavior of such a Nash equilibrium  when the number of agents goes to infinity 
 \citep{gomes2010discrete,tembine2011mean, moon2014discrete, biswas2015mean, saldi2018discrete, saldi2018markov, saldi2019approximate}. See also \cite{yang2018deep}, which applies  maximum entropy  inverse RL \citep{ziebart2008maximum} to infer the cost function and social  dynamics of discrete-time mean-field games with finite state and action spaces. 
 Our work is most related to \cite{guo2019learning}, where they propose a mean-field Q-learning algorithm \citep{watkins1992q} for discrete-time mean-field games with finite state and action spaces. Such an algorithm requires the access to a simulator, which, given any state-action pair and mean-field state, outputs the next state. In contrast,  both our state and action spaces are infinite, and we do not require such a simulator but only observations of trajectories. 
 Correspondingly, we study the mean-field actor-critic algorithm with linear function approximation, whereas their algorithm is tailored to the tabular setting.  Also, our work is closely related to \cite{mguni2018decentralised}, which focuses on a more restrictive setting where the state transition does not involve the mean-field state. In such a setting, mean-field games are potential games, which is, however, not true in more general settings \citep{li2017connections, briani2018stable}. In comparison, we allow the state transition to depend on the mean-field state. Meanwhile, they propose a fictitious self-play method based on the single-agent actor-critic algorithm and establishes its asymptotic convergence. However, their proof of convergence relies on the assumption that the single-agent actor-critic algorithm converges to the optimal policy, which is unverified therein.  Meanwhile, a model-based algorithm is proposed in \cite{kaiqing2019approximate} for the discounted linear-quadratic mean-field games, where they only show that the algorithm converges asymptotically to the Nash equilibrium.   In addition, our work is related to \cite{Subramanian2019Reinforce}, where the proposed algorithm is only shown to converge asymptotically to a stationary point of the mean-field game.  
 
 Our work also extends the line of work on finding the Nash equilibria of Markov games using MARL. Due to the computational intractability introduced by the large number of agents, such a line of work focuses on finite-agent Markov games  \citep{littman1994markov, littman2001friend,hu1998multiagent, bowling2001rational, lagoudakis2002value,hu2003nash, conitzer2007awesome,  perolat2015approximate, perolat2016use, perolat2016softened, perolat2018actor,  wei2017online, zhang2018finite, zou2019finite, casgrain2019deep}. See also  \cite{shoham2003multi,shoham2007if, busoniu2008comprehensive,li2018deep} for detailed surveys.  Our work is related to \cite{yang2018mean}, where they combine the mean-field approximation of actions (rather than states) and Nash Q-learning \citep{hu2003nash} to study general-sum Markov games with a large number of agents. However, the Nash Q-learning algorithm is only applicable to finite state and action spaces, and its convergence is established under rather strong assumptions. 
 Also, when the number of agents goes to infinity, their approach yields a variant of tabular Q-learning, which is different from our mean-field actor-critic algorithm.

For policy optimization, based on the policy gradient theorem,  
\cite{sutton2000policy,konda2000actor} propose the actor-critic algorithm, which is later generalized to the natural actor-critic algorithm \citep{peters2008natural, bhatnagar2009natural}. Most existing results on the convergence of actor-critic algorithms are based on stochastic approximation using ordinary differential equations \citep{bhatnagar2009natural, castro2010convergent,konda2000actor, maei2018convergent}, which are asymptotic in nature. For policy evaluation, the convergence of primal-dual gradient temporal difference is studied in \cite{liu2015finite, du2017stochastic, wang2017finite,yu2017convergence,wai2018multi}. However, this line of work assumes that the feature mapping is bounded, which is not the case in our setting. Thus, the existing convergence results are not applicable to analyzing the critic update in our setting. To handle the unbounded feature mapping, 
we utilize a truncation argument, which requires more delicate analysis. 

Finally, our work extends the line of work that studies model-free RL for LQR. For example,  \cite{bradtke1993reinforcement, bradtke1994adaptive} show that  policy iteration converges to the optimal policy,  \cite{tu2017least, dean2017sample} study the sample complexity of least-squares temporal-difference for policy evaluation. More recently,  \cite{fazel2018global,malik2018derivative,  tu2018gap} show that the policy gradient algorithm converges at a linear rate to the optimal policy. See as also \cite{hardt2016gradient, dean2018regret} for more in this line of work.    Our work is also closely related to \cite{yyy2019aclqr}, where they show that the sequence of policies generated by the natural actor-critic algorithm enjoys a linear rate of convergence to the optimal policy. Compared with this work, when fixing the mean-field state, we use the actor-critic algorithm to study LQR in the presence of drift, which introduces significant difficulties in the analysis. As we show in \S \ref{sec:algo}, the drift causes the optimal policy to have an additional intercept, which makes the state- and action-value functions more complicated. 




\vskip5pt

{\noindent \bf  Notations.}
We denote by $\|M\|_2$ the spectral norm,  $\rho(M)$ the spectral radius,  $\sigma_{\min}(M)$ the minimum singular value, and $\sigma_{\max}(M)$ the maximum singular value of a matrix $M$.   We use $\|\alpha\|_2$ to represent the $\ell_2$-norm of a vector $\alpha$, and $(\alpha)_i^j$ to denote the sub-vector $(\alpha_i, \alpha_{i+1}, \ldots, \alpha_j)^\top$, where $\alpha_k$ is the $k$-th entry of the vector $\alpha$.       For scalars $a_1, \ldots, a_n$, we denote by $\poly(a_1, \ldots, a_n)$ the polynomial of $a_1, \ldots, a_n$, and this polynomial may vary from line to line.   We use $[n]$ to denote the set $\{1, 2, \ldots, n\}$ for any $n\in\NN$.


\setitemize{label=\textbullet, leftmargin=20pt, nolistsep}

\section{Linear-Quadratic Mean-Field Game}\label{sec:algo_lqmfg}
A linear-quadratic mean-field game involves $N_{\text{a}}\in\NN$ agents. Their state transitions are given by
\$
x_{t+1}^i = A x_t^i + B u_t^i + \overline A\cdot \frac{1}{N_{\text{a}}}\sum_{j = 1}^{N_{\text{a}}} x_t^j + d^i +  \omega_t^i,\qquad \forall t\geq 0, ~i\in[N_{\text{a}}],
\$
where $x_t^i\in\RR^m$ and $u_t^i\in\RR^k$ are the state and action vectors of agent $i$, respectively, the vector $d^i\in\RR^m$ is a drift term, and $\omega_t^i\in\RR^m$ is an independent random noise term following the Gaussian distribution $\mathcal N(0,\Psi_\omega)$. The agents are coupled through the mean-field state $1/N_{\text{a}}\cdot \sum_{j = 1}^{N_{\text{a}}} x_t^j$. In the linear-quadratic mean-field game, the cost of agent $i\in[N_{\text{a}}]$ at time $t\geq 0$ is given by
\$
c^i_t = (x_t^i)^\top Q x_t^i + (u_t^i)^\top R u_t^i  + \Biggl(\frac{1}{N_{\text{a}}}\sum_{j = 1}^{N_{\text{a}}} x_t^j\Biggr)^\top \overline Q\Biggl(\frac{1}{N_{\text{a}}}\sum_{j = 1}^{N_{\text{a}}} x_t^j\Biggr),
\$
where $u^i_t$ is generated by $\pi^i$, i.e., the policy of agent $i$.
To measure the performance of agent $i$ following its policy $\pi^i$ under the influence of the other agents, we define the expected total cost of agent $i$ as
\$
J^i(\pi^1, \pi^2, \ldots, \pi^{N_{\text{a}}}) = \lim_{T\to\infty} \EE\Biggl(\frac{1}{T}\sum_{t = 0}^T c_t^i\Biggr). 
\$ 
We are interested in finding a Nash equilibrium $(\pi^1, \pi^2, \ldots, \pi^{N_{\text{a}}})$, which is defined by
\$
J^i(\pi^1, \ldots, \pi^{i-1}, \pi^i, \pi^{i+1},\ldots, \pi^{N_{\text{a}}})  \leq J^i(\pi^1, \ldots, \pi^{i-1}, \tilde \pi^i, \pi^{i+1},\ldots, \pi^{N_{\text{a}}}), \qquad \forall \tilde \pi^i, ~ i\in[N_{\text{a}}].
\$
That is, agent $i$ cannot further decrease its expected total cost by unilaterally deviating from its Nash policy.  

For the simplicity of discussion, we assume that the drift term $d^i$ is identical for each agent. By the symmetry of the agents in terms of their state transitions and cost functions, we focus on a fixed agent and drop the superscript $i$ hereafter.    Further taking the infinite-population limit $N_{\text{a}}\to\infty$ leads to the following formulation of linear-quadratic mean-field game (LQ-MFG). 

\begin{problem}[LQ-MFG]\label{prob:mflqr}
We consider the following formulation,
\$
& x_{t+1} = A x_t + B u_t + \overline A  \EE x_t^* + d + \omega_t, \notag\\
& c(x_t, u_t) = x_t^\top Q x_t + u_t^\top R u_t + (\EE x_t^*)^\top \overline Q(\EE x_t^*), \notag\\
& J(\pi) = \lim_{T\to\infty} \EE\Biggl[\frac{1}{T}\sum_{t = 0}^T c(x_t, u_t)\Biggr], 
\$
where $x_t\in\RR^m$ is the state vector, $u_t\in\RR^k$ is the action vector generated by the policy $\pi$, $\{x_t^*\}_{t\geq0}$ is the trajectory generated by a Nash policy $\pi^*$ (assuming it exists), $\omega_t\in\RR^m$ is an independent random noise term following the Gaussian distribution $\mathcal N(0,\Psi_\omega)$, and $d\in\RR^m$ is a drift term. Here the expectation $\EE x_t^*$ is taken across all the agents. We aim to find $\pi^*$ such that $J(\pi^*) = \inf_{\pi\in\Pi}J(\pi)$. 
\end{problem}

The formulation in Problem \ref{prob:mflqr} is studied by \cite{lasry2007mean, bensoussan2016linear, saldi2018discrete, saldi2018markov}. 
We propose a more general formulation in Problem \ref{prob:mflqr2} (see \S\ref{sec:gen_form} of the appendix for details), where an additional interaction term between the state vector $x_t$ and the mean-field state $\EE x_t^*$ is incorporated into the cost function. According to our analysis in  \S\ref{sec:gen_form}, up to minor modification, the results in the following sections also carry over to Problem \ref{prob:mflqr2}. Therefore, for the sake of simplicity, we focus on Problem \ref{prob:mflqr} in the sequel.   

Note that the mean-field state $\EE x_t^*$ converges to a constant vector $\mu^*$ as $t\to\infty$, which serves as a fixed mean-field state, since the Markov chain of states generated by the Nash policy $\pi^*$ admits a stationary distribution.  
As we consider the ergodic setting, it suffices to study Problem \ref{prob:mflqr} with $t$ sufficiently large, which motivates the following drifted LQR (D-LQR) problem, where the mean-field state acts as another drift term.

\begin{problem}[D-LQR]\label{prob:slqr}
Given a  mean-field state $\mu\in\RR^m$, we consider the following formulation,
\$
& x_{t+1} = A x_t + B u_t + \overline A \mu + d + \omega_t,\notag\\
& c_\mu(x_t, u_t) = x_t^\top Q x_t + u_t^\top R u_t + \mu^\top \overline Q\mu, \notag\\
& J_\mu(\pi) = \lim_{T\to\infty} \EE\Biggl[\frac{1}{T}\sum_{t = 0}^T c_\mu(x_t, u_t)\Biggr],  
\$
where $x_t\in\RR^m$ is the state vector, $u_t\in\RR^k$ is the action vector generated by the policy $\pi$,  $\omega_t\in\RR^m$ is an independent random noise term following the Gaussian distribution $\mathcal N(0,\Psi_\omega)$, and $d\in\RR^m$ is a drift term. We aim to find an optimal policy $\pi^*_\mu$ such that $J_\mu(\pi^*_\mu) = \inf_{\pi\in\Pi}J_\mu(\pi)$. 
\end{problem}

For the mean-field state $\mu = \mu^*$, which corresponds to the Nash equilibrium, solving Problem \ref{prob:slqr} gives $\pi^*_{\mu^*}$, which coincides with the Nash policy $\pi^*$ defined in Problem \ref{prob:mflqr}. Compared with the most studied LQR problem \citep{lewis2012optimal}, both the state transition and the cost function in Problem \ref{prob:slqr} have drift terms, which act as the mean-field ``force'' that drives the states away from zero.  Such a mean-field ``force'' introduces additional challenges when solving Problem \ref{prob:slqr} in the model-free setting (see  \S\ref{sec:slqr} for  details). On the other hand, the unique optimal policy $\pi^*_\mu$ of Problem \ref{prob:slqr} admits a linear form $\pi^*_\mu(x_t) = -K_{\pi^*_\mu}x_t + b_{\pi^*_\mu}$ \citep{anderson2007optimal}, where the matrix $K_{\pi^*_\mu} \in\RR^{k\times m}$ and the vector $b_{\pi^*_\mu} \in\RR^{k}$ are the parameters of $\pi_\mu^*$.  Motivated by such a linear form of the optimal policy, we define the class of linear-Gaussian policies as 
\#\label{eq:def-Pi}
\Pi = \{ \pi(x) = -Kx + b + \sigma\cdot \eta\colon K\in\RR^{k\times m}, b\in\RR^k \},
\# 
where the standard Gaussian noise term $\eta\in\RR^k$ is included to encourage exploration.  To solve Problem \ref{prob:slqr}, it suffices to find the optimal policy $\pi^*_\mu$ within $\Pi$.


Now, we introduce the definition of the Nash equilibrium pair \citep{saldi2018discrete, saldi2018markov}. 
The Nash equilibrium pair is characterized by the NCE principle, which states that it suffices to find a pair of $\pi^*$ and $\mu^*$, such that the policy $\pi^*$  is optimal for each agent when the mean-field state is $\mu^*$, while all the agents following the policy $\pi^*$ generate the mean-field state $\mu^*$ as $t\to\infty$. 
To present its formal definition, we define $\Lambda_1(\mu)$ as the optimal policy in $\Pi$ given the mean-field state $\mu$, and define $\Lambda_2(\mu, \pi)$ as the mean-field state generated by the policy $\pi$ given the current mean-field state $\mu$ as $t\to\infty$.

\begin{definition}[Nash Equilibrium Pair]\label{def:eq_pair}
The pair $(\mu^*, \pi^*)\in\RR^m\times \Pi$ constitutes a Nash equilibrium pair of Problem \ref{prob:mflqr} if it satisfies $\pi^*=\Lambda_1(\mu^*)$ and $\mu^* = \Lambda_2(\mu^*, \pi^*)$. Here $\mu^*$ is called the Nash mean-field state and $\pi^*$ is called the Nash policy.
\end{definition}

\section{Mean-Field Actor-Critic} \label{sec:algo}

We first characterize the existence and uniqueness of the Nash equilibrium pair of Problem \ref{prob:mflqr} under mild regularity conditions, and then propose a mean-field actor-critic algorithm to obtain such a Nash equilibrium. As a building block of the mean-field actor-critic, we propose the natural actor-critic to solve Problem \ref{prob:slqr}.

\subsection{Existence and Uniqueness of Nash Equilibrium Pair}

We now establish the existence and uniqueness of the Nash equilibrium pair defined in Definition \ref{def:eq_pair}.  We impose the following regularity conditions. 

\begin{assumption}\label{assum:contraction}
We assume that the following statements hold:
\begin{itemize}
\item[(i)] The algebraic Riccati equation $X = A^\top X A + Q - A^\top  X   B(B^\top  X  B + R)^{-1}B^\top X  A$ admits a unique symmetric positive definite solution $X^*$; 
\item[(ii)] It holds for $L_0 = L_1L_3 + L_2$ that $L_0 < 1$, where 
\$
& L_1 = \bigl\| \bigl[ (I-A)Q^{-1}(I-A)^\top + BR^{-1}B^\top \bigr]^{-1} \overline A \bigr\|_2\cdot  \bigl\|\bigl[ K^*Q^{-1} (I-A)^\top - R^{-1}B^\top \bigr] \bigr\|_2,\notag\\
& L_2 = \bigl[1-\rho(A-BK^*)\bigr]^{-1}\cdot \|\overline A\|_2, \qquad L_3 = \bigl[1-\rho(A-BK^*)\bigr]^{-1}\cdot \|B\|_2.
\$
Here $K^* = -(B^\top X^* B + R)^{-1}B^\top X^* A$. 
\end{itemize}
\end{assumption}

The first assumption is implied by mild regularity conditions on the matrices $A$, $B$, $Q$, and $R$. See Theorem 3.2 in \cite{de1986riccati} for details. 
The second assumption is standard in the literature \citep{bensoussan2016linear, saldi2018markov}, which ensures the stability of the LQ-MFG. 
In the following proposition, we show that Problem \ref{prob:mflqr} admits a unique Nash equilibrium pair. 

\begin{proposition}[Existence and Uniqueness of Nash Equilibrium Pair]\label{prop:uniq_eq}
Under Assumption \ref{assum:contraction},  
the operator $\Lambda(\cdot) = \Lambda_2(\cdot, \Lambda_1(\cdot))$ is $L_0$-Lipschitz, where $L_0$ is given in Assumption \ref{assum:contraction}. Moreover, there exists a unique Nash equilibrium pair $(\mu^*, \pi^*)$ of Problem \ref{prob:mflqr}. 
\end{proposition}

\begin{proof}
See \S\ref{proof:prop:uniq_eq} for a detailed proof. 
\end{proof}

\subsection{Mean-Field Actor-Critic for LQ-MFG}

The NCE principle motivates a fixed-point approach to solve Problem \ref{prob:mflqr}, which generates a sequence of policies  $\{ \pi_s \}_{s\geq 0}$ and mean-field states $\{ \mu_s \}_{s\geq 0}$ satisfying the following two properties: (i) Given the mean-field state $\mu_s$, the policy $\pi_s$ is optimal. (ii) The mean-field state becomes $\mu_{s+1}$ as $t\to\infty$, if all the agents follow $\pi_s$ under the current mean-field state $\mu_s$.  Here (i) requires solving Problem \ref{prob:slqr} given the mean-field state $\mu_s$, while (ii) requires simulating the agents following the policy $\pi_s$ given the current mean-field $\mu_s$.   Based on such properties, we propose the mean-field actor-critic in Algorithm \ref{algo:mflqr}.

\begin{algorithm}[h]
    \caption{Mean-Field Actor-Critic for solving LQ-MFG.}\label{algo:mflqr}
    \begin{algorithmic}[1]
    \STATE{{\textbf{Input:}} 
    \begin{itemize}
    \setlength\itemsep{0em}
    \item Initial mean-field state $\mu_0$ and Initial policy $\pi_0$ with parameters $K_0$ and $b_0$.
    \item Numbers of iterations $S$,   $\{N_s\}_{s\in[S]}$,  $\{H_{s}\}_{s\in[S]}$,  $\{\tilde T_{s,n}, T_{s,n}\}_{s\in[S], n\in[N_s]}$,  $\{\tilde T^b_{s,h}, T^b_{s,h}\}_{s\in[S], h\in[H_{s}]}$.
    \item  Stepsizes $\{\gamma_s\}_{s\in[S]}$,  $\{\gamma^b_{s}\}_{s\in[S]}$,  $\{\gamma_{s, n, t}\}_{s\in[S], n\in[N_s], t\in[T_{s,n}]}$,  $\{\gamma^b_{s, h, t}\}_{s\in[S], h\in [H_{s}], t\in[T^b_{s,h}]}$.
    \end{itemize}}
    \FOR{$s = 0, 1, 2, \ldots, S-1$}
        \STATE \textbf{Policy Update:} Solve for the optimal policy $\pi_{s+1}$ with parameters $K_{s+1}$ and $b_{s+1}$ of Problem \ref{prob:slqr} via Algorithm \ref{algo:ac_lqr} with $\mu_s$,  $\pi_{s}$,  $N_s$,   $H_{s}$, $\{\tilde T_{s, n}, T_{s, n}\}_{n\in[N_s]}$, $\{\tilde T^b_{s, h}, T^b_{s, h}\}_{h\in[H_{s}]}$,  $\gamma_s$,  $\gamma^b_{s}$,  $\{\gamma_{s, n, t}\}_{n\in[N_s], t\in[T_{s,n}]}$, and $\{\gamma^b_{s, h, t}\}_{h\in [H_{s}], t\in[T^b_{s,h}]}$, which gives the estimated mean-field state $\hat \mu_{K_{s+1}, b_{s+1}}$.  \label{line:policy-update}
        \STATE \textbf{Mean-Field State Update:} Update the mean-field state via $\mu_{s+1} \gets \hat \mu_{K_{s+1}, b_{s+1}}$. 
    \ENDFOR
    \STATE{\textbf{Output:} Pair $(\pi_S, \mu_S)$.}
    \end{algorithmic}
\end{algorithm}

Algorithm \ref{algo:mflqr} requires solving Problem \ref{prob:slqr} at each iteration to obtain $\pi_s = \Lambda_1(\mu_s)$ and $\mu_{s+1} = \Lambda_2(\mu_s, \pi_s)$. To this end, we introduce the natural actor-critic in \S\ref{sec:slqr} that solves Problem \ref{prob:slqr}.

\subsection{Natural Actor-Critic for D-LQR}\label{sec:slqr}
Now we focus on solving Problem \ref{prob:slqr} for a fixed mean-field state $\mu$, we thus drop the subscript $\mu$ hereafter. 
We write $\pi_{K,b}(x)=-Kx + b + \sigma\eta$ to emphasize the dependence on $K$ and $b$, and $J(K,b) = J(\pi_{K,b})$ consequently. Now, we propose the natural actor-critic to solve Problem \ref{prob:slqr}. 

For any policy $\pi_{K,b}\in\Pi$, by the state transition in Problem \ref{prob:slqr}, we have
\#\label{eq:f1}
x_{t+1} = (A-BK)x_t + (Bb + \overline A\mu + d) + \epsilon_t, \qquad \epsilon_t\sim\mathcal N(0,\Psi_\epsilon),
\#
where $\Psi_\epsilon = \sigma BB^\top + \Psi_\omega$.  It is known that if $\rho(A-BK) < 1$, then the Markov chain $\{x_t\}_{t\geq 0}$ induced by \eqref{eq:f1} has a unique stationary distribution $\mathcal N(\mu_{K,b}, \Phi_K)$ \citep{anderson2007optimal}, where the mean-field state $\mu_{K,b}$ and the covariance $\Phi_K$ satisfy that
\#
& \mu_{K,b} = (I - A+BK)^{-1} (Bb + \overline A\mu + d),\label{eq:f2q}\\
& \Phi_K = (A-BK) \Phi_K (A-BK)^\top + \Psi_\epsilon.  \label{eq:f2p}
\#
Meanwhile, the Bellman equation for Problem \ref{prob:slqr} takes the following form
\#\label{eq:bellman}
P_K = (Q+K^\top RK) + (A-BK)^\top P_K(A-BK). 
\#
Then by calculation (see Proposition \ref{prop:cost_form} in  \S\ref{sec:res_lqr} of the appendix for details), it holds that the expected total cost $J(K, b)$  is decomposed as
\#\label{eq:a1}
J(K,b) = J_1(K) + J_2(K,b) + \sigma^2 \cdot \tr(R) + \mu^\top \overline Q \mu,
\#
where $J_1(K)$ and $J_2(K,b)$ are defined as
\#\label{eq:a2}
& J_1(K) =\tr\bigl[(Q + K^\top R K)\Phi_K\bigr]  = \tr(P_K \Psi_\epsilon),\notag\\ 
& J_2(K,b) = \begin{pmatrix}
\mu_{K,b}\\
b
\end{pmatrix}^\top
\begin{pmatrix}
Q + K^\top RK & -K^\top R\\
-R K & R
\end{pmatrix}
 \begin{pmatrix}
\mu_{K,b}\\
b
\end{pmatrix}.
\#
Here $J_1(K)$ is the expected total cost in the most studied LQR problems \citep{yyy2019aclqr, fazel2018global}, where the state transition does not have drift terms. Meanwhile, $J_2(K,b)$ corresponds to the expected cost induced by the drift terms. The following two propositions characterize the properties of $J_2(K, b)$. 

First, we show that $J_2(K,b)$ is strongly convex in $b$. 
\begin{proposition}\label{prop:convex_J2}
Given any $K$, the function $J_2(K,b)$ is $\nu_K$-strongly convex in $b$. Here $\nu_K = \sigma_{\min}(Y_{1, K}^\top Y_{1, K} + Y_{2, K}^\top Y_{2, K})$, where $Y_{1, K} = {R^{1/2}}K(I-A+BK)^{-1}B - {R^{1/2}}$ and $Y_{2, K} = {Q^{1/2}} (I-A+BK)^{-1} B$. Also, $J_2(K, b)$ has $\iota_K$-Lipschitz continuous gradient in $b$, where $\iota_K$ is upper bounded as  $\iota_K \leq [1-\rho(A-BK)]^{-2} \cdot (  \|B\|_2^2\cdot \|K\|_2^2\cdot \|R\|_2 + \|B\|_2^2\cdot \|Q\|_2 )$. 
\end{proposition}
\begin{proof}
See \S\ref{proof:prop:convex_J2} for a detailed proof. 
\end{proof}

Second, we show that $\min_b J_2(K, b)$ is independent of $K$. 
 
\begin{proposition}\label{prop:J2}
We define $b^K = \argmin_{b} J_2(K,b)$, where $J_2(K,b)$ is defined in \eqref{eq:a2}.  It holds that
\$
b^K =  \bigl[ KQ^{-1} (I-A)^\top - R^{-1}B^\top \bigr] \cdot \bigl[ (I-A)Q^{-1}(I-A)^\top + BR^{-1}B^\top \bigr]^{-1} \cdot  (\overline A \mu + d). 
\$
Moreover, $J_2(K, b^K)$ takes the form of
\$
J_2(K, b^K) = (\overline A \mu + d)^\top \bigl[ (I-A)Q^{-1}(I-A)^\top + BR^{-1}B^\top \bigr]^{-1} \cdot (\overline A \mu + d),
\$
which is independent of $K$.
\end{proposition}
\begin{proof}
See \S\ref{proof:prop:J2} for a detailed proof. 
\end{proof}

Since $\min_b J_2(K, b)$ is independent of $K$ by Proposition \ref{prop:J2}, it holds that the optimal $K^*$ is the same as $\argmin_K J_1(K)$. This motivates us to minimize $J(K,b)$ by first updating $K$ following the gradient direction $\nabla_K J_1(K)$ to the optimal $K^*$, then updating $b$ following the gradient direction $\nabla_b J_2(K^*, b)$.  We now design our algorithm based on this idea.  

We define  $\Upsilon_K$, $p_{K,b}$, and $q_{K,b}$ as
\#\label{eq:def_upsilon}
& \Upsilon_K = \begin{pmatrix}
Q + A^\top P_K A & A^\top P_K B\\
B^\top P_K A & R + B^\top P_K B
\end{pmatrix} = \begin{pmatrix}
\Upsilon_K^{11} & \Upsilon_K^{12}\\
\Upsilon_K^{21} & \Upsilon_K^{22}
\end{pmatrix},\notag\\
&   p_{K,b} = A^\top\bigl[P_K\cdot (\overline A \mu+d) + f_{K,b} \bigr],\qquad q_{K,b} = B^\top\bigl[P_K\cdot (\overline A \mu+d) + f_{K,b} \bigr],
\#
where $f_{K,b} = (I-A+BK)^{-\top} [  (A-BK)^\top P_K (Bb+\overline A\mu + d) - K^\top Rb ]$.  
By calculation (see Proposition \ref{prop:pg} in \S\ref{sec:res_lqr} of the appendix for details), the gradients of $J_1(K)$ and $J_2(K,b)$ take the forms of
\$ 
&\nabla_K J_1(K) = 2(\Upsilon_K^{22} K  - \Upsilon_K^{21})\cdot \Phi_K,\qquad \nabla_b J_2(K,b) = \Upsilon_K^{22}(-K\mu_{K, b} + b) + \Upsilon_{K}^{21}\mu_{K,b} + q_{K,b}.
\$

Our algorithm follows the natural actor-critic method \citep{bhatnagar2009natural} and actor-critic method \citep{konda2000actor}.  
Specifically, (i) To obtain the optimal $K^*$, in the critic update step, we estimate the matrix $\Upsilon_K$ by $\hat \Upsilon_K$ via a policy evaluation algorithm, e.g., Algorithm \ref{algo:pg_eval} or Algorithm \ref{algo:td_pe} (see \S\ref{sec:pe_algo} and \S\ref{sec:td_algo} of the appendix for details);  in the actor update step, we update $K$ via $K\gets K - \gamma\cdot (\hat \Upsilon_K^{22} K - \hat \Upsilon_K^{21})$, where the term $\hat \Upsilon_K^{22} K - \hat \Upsilon_K^{21}$ is the estimated natural gradient.  
(ii) To obtain the optimal $b^*$ given $K^*$, in the critic update step, we estimate $\Upsilon_{K^*}$, $q_{K^*,b}$, and $\mu_{K^*,b}$ by $\hat \Upsilon_{K^*}$, $\hat q_{K^*,b}$, and $\hat\mu_{K^*,b}$ via a policy evaluation algorithm;  In the actor update step, we update $b$ via $b\gets b - \gamma\cdot \hat \nabla_b J_2(K^*, b)$, where $\hat \nabla_b J_2(K^*, b) = \hat\Upsilon_{K^*}^{22}(-K^*\hat\mu_{K^*, b} + b) + \hat \Upsilon_{K^*}^{21}\hat\mu_{K^*,b} + \hat q_{K^*,b}$ is the estimated gradient.
Combining the above procedure, we obtain the natural actor-critic for Problem \ref{prob:slqr}, which is stated in Algorithm \ref{algo:ac_lqr}.



\begin{algorithm}[h]
    \caption{Natural Actor-Critic Algorithm for D-LQR.}\label{algo:ac_lqr}
    \begin{algorithmic}[1]
    \STATE{{\textbf{Input:}} 
    \begin{itemize}
    \item Mean-field state $\mu$ and initial policy $\pi_{K_0, b_0}$.
    \item Numbers of iterations $N$,   $H$,  $\{\tilde T_{n}, T_{n}\}_{n\in[N]}$,  $\{\tilde T^b_{h}, T^b_{h}\}_{h\in[H]}$.
    \item Stepsizes $\gamma$,  $\gamma^b$,  $\{\gamma_{n, t}\}_{n\in[N], t\in[T_n]}$,  $\{\gamma^b_{h, t}\}_{h\in [H], t\in[T^b_{h}]}$.
    \end{itemize}}
    \FOR{ $n = 0, 1, 2, \ldots, N-1$}
        \STATE\textbf{Critic Update:} Compute $\hat\Upsilon_{K_n}$ via  Algorithm \ref{algo:pg_eval} with $\pi_{K_n, b_0}$,  $\mu$,  $\tilde T_n$,$T_n$,  $\{\gamma_{n, t}\}_{t\in[T_n]}$, $K_0$, and $b_0$ as inputs.\label{line:critic-update-ac_lqr}
        \STATE\textbf{Actor Update:} Update the parameter via 
        \$
        K_{n+1} \gets K_n -  \gamma\cdot (\hat \Upsilon_{K_n}^{22} K_n  - \hat \Upsilon_{K_n}^{21}).
        \$\vskip-5pt
    \ENDFOR
    \FOR{ $h = 0, 1, 2, \ldots, H-1$}
        \STATE\textbf{Critic Update:} Compute $\hat \mu_{K_N, b_h}$, $\hat\Upsilon_{K_N}$, $\hat q_{K_N,b_h}$ via Algorithm \ref{algo:pg_eval} with  $\pi_{K_N, b_h}$,  $\mu$,  $\tilde T^b_h$, $T^b_h$,  $\{\gamma^b_{h, t}\}_{t\in[T^b_h]}$,  $K_0$, and $b_0$.  \label{line:critic-update}
        \STATE\textbf{Actor Update:} Update the parameter via
        \$
        b_{h+1} \gets b_h -  \gamma^b\cdot \bigl[\hat \Upsilon_{K_N}^{22}(-K_N\hat \mu_{K, b_h} + b_h) + \hat \Upsilon_{K_N}^{21}\hat \mu_{K_N,b_h} +\hat  q_{K_N,b_h}\bigr].
        \$\vskip-5pt
    \ENDFOR
    \STATE{\textbf{Output:} Policy $\pi_{K,b} = \pi_{K_N, b_H}$, estimated mean-field state $\hat \mu_{K, b} = \hat \mu_{K_N, b_H}$ .} 
    \end{algorithmic}
\end{algorithm}


\section{Global Convergence Results}
The following theorem establishes the rate of convergence of Algorithm \ref{algo:mflqr} to the Nash equilibrium pair $(\mu^*, \pi^*)$ of Problem \ref{prob:mflqr}. 

\begin{theorem}[Convergence of Algorithm \ref{algo:mflqr}]\label{thm:conv_mfg}
For a sufficiently small tolerance $\varepsilon > 0$, we set the number of iterations $S$ in Algorithm \ref{algo:mflqr} such that 
\#\label{eq:choice_S}
S > \frac{\log \bigl(\|\mu_0 - \mu^*\|_2  \cdot \varepsilon^{-1}\bigr)}{\log (1/L_0)}. 
\#
For any $s\in[S]$, we define 
\#\label{eq:def_eps_s}
& \varepsilon_s =  \min\Bigl\{     \bigl[ 1 - \rho(A-BK^*) \bigr]^4 \bigl( \|B\|_2 + \|\overline A\|_2 \bigr)^{-4} \bigl( \|\mu_s\|_2^{-2} + \|d\|_2^{-2}\bigr)  \cdot\sigma_{\min}(\Psi_\epsilon)\cdot \sigma_{\min}(R)\cdot \varepsilon^2,  \notag\\
&\qquad\qquad\qquad \nu_{K^*} \cdot \bigl[ 1 - \rho(A-BK^*) \bigr]^4\cdot \|B\|_2^{-2}  \cdot M_b(\mu_s)\cdot \varepsilon^2, ~ \varepsilon \Bigr\}\cdot 2^{-s-10},
\#
where $\nu_{K^*}$ is defined in Proposition \ref{prop:convex_J2} and
\#\label{eq:def-mb}
M_b(\mu_s) & = 4 \Bigl\| Q^{-1} (I-A)^\top   \cdot \bigl[ (I-A)Q^{-1}(I-A)^\top + BR^{-1}B^\top \bigr]^{-1} \cdot  (\overline A \mu_s + d) \Bigr\|_2 \notag\\
&\qquad \cdot  \bigl[{{\nu_{K^*}^{-1} + \sigma_{\min}^{-1}(\Psi_\epsilon)\cdot \sigma_{\min}^{-1}(R) }}\bigr]^{1/2}.
\# 
In the $s$-th policy update step in Line \ref{line:policy-update} of Algorithm \ref{algo:mflqr}, we set the inputs via Theorem \ref{thm:ac} such that $J_{\mu_s}(\pi_{s+1}) - J_{\mu_s}(\pi_{\mu_s}^*) < \varepsilon_s$, where the expected total cost $J_{\mu_s}(\cdot)$ is defined in Problem \ref{prob:slqr}, and $\pi_{\mu_s}^* = \Lambda_1(\mu_s)$ is the optimal policy under the mean-field state $\mu_s$. 
Then it holds with probability at least $1-\varepsilon^5$ that 
\$
\|\mu_S - \mu^*\|_2 \leq \varepsilon,\qquad \|K_S - K^*\|_\F \leq \varepsilon, \qquad \|b_S - b^* \|_2 \leq (1 + L_1)\cdot \varepsilon.
\$
Here $\mu^*$ is the Nash mean-field state,  $K_S$ and $b_S$ are parameters of the policy $\pi_S$, and $K^*$ and $b^*$ are parameters of the Nash policy $\pi^*$. 
\end{theorem}
\begin{proof}
See \S\ref{proof:thm:conv_mfg} for a detailed proof. 
\end{proof}

We highlight that if the inputs of Algorithm \ref{algo:mflqr} satisfy the conditions stated in Theorem \ref{thm:ac}, it holds that $J_{\mu_s}(\pi_{s+1}) - J_{\mu_s}(\pi_{\mu_s}^*) < \varepsilon_s$ for any $s\in[S]$.  See Theorem \ref{thm:ac} in \S\ref{sec:res_lqr} of the appendix for details. 
By Theorem \ref{thm:conv_mfg},  Algorithm \ref{algo:mflqr} converges linearly to the unique Nash equilibrium pair $(\mu^*, \pi^*)$ of Problem \ref{prob:mflqr}. 
To the best of our knowledge,  this theorem is the first successful attempt to establish that reinforcement learning with function approximation finds the Nash equilibrium pairs in mean-field games with theoretical guarantee, which lays the theoretical foundations for applying modern reinforcement learning techniques to general mean-field games.

\section{Conclusion}
For the discrete-time linear-quadratic mean-field games, we provide sufficient conditions for the existence and uniqueness of the Nash equilibrium pair. Moreover, we propose the mean-field actor-critic algorithm with linear function approximation that is shown converges to the Nash equilibrium pair with linear rate of convergence.  
Our algorithm can be modified to use other parametrized function classes, including deep neural networks, for solving mean-field games. 
For future research, we aim to extend our algorithm to 
other variations of mean-field games  including  risk-sensitive mean-field games \citep{saldi2018discrete,tembine2014risk}, 
robust mean-field games \citep{bauso2016robust}, and 
partially observed mean-field games \citep{saldi2019approximate}.

\bibliographystyle{ims}
\bibliography{rl_ref}

\begin{thebibliography}{124}
\expandafter\ifx\csname natexlab\endcsname\relax\def\natexlab#1{#1}\fi
\expandafter\ifx\csname url\endcsname\relax
  \def\url#1{\texttt{#1}}\fi
\expandafter\ifx\csname urlprefix\endcsname\relax\def\urlprefix{}\fi

\bibitem[{Alizadeh et~al.(1998)Alizadeh, Haeberly and
  Overton}]{alizadeh1998primal}
\text{Alizadeh, F.}, \text{Haeberly, J.-P.~A.} and \text{Overton, M.~L.}
  (1998).
\newblock Primal-dual interior-point methods for semidefinite programming:
  convergence rates, stability and numerical results.
\newblock \textit{SIAM Journal on Optimization}, \textbf{8} 746--768.

\bibitem[{Anderson and Moore(2007)}]{anderson2007optimal}
\text{Anderson, B.~D.} and \text{Moore, J.~B.} (2007).
\newblock \textit{Optimal control: linear quadratic methods}.
\newblock Courier Corporation.

\bibitem[{Araki et~al.(2017)Araki, Strang, Pohorecky, Qiu, Naegeli and
  Rus}]{araki2017multi}
\text{Araki, B.}, \text{Strang, J.}, \text{Pohorecky, S.}, \text{Qiu, C.},
  \text{Naegeli, T.} and \text{Rus, D.} (2017).
\newblock Multi-robot path planning for a swarm of robots that can both fly and
  drive.
\newblock In \textit{2017 IEEE International Conference on Robotics and
  Automation (ICRA)}. IEEE.

\bibitem[{Ash(2000)}]{ash2000social}
\text{Ash, C.} (2000).
\newblock Social-self-interest.
\newblock \textit{Annals of public and cooperative economics}, \textbf{71}
  261--284.

\bibitem[{Axtell(2002)}]{axtell2002non}
\text{Axtell, R.~L.} (2002).
\newblock Non-cooperative dynamics of multi-agent teams.
\newblock In \textit{Autonomous Agents and Multiagent Systems}.

\bibitem[{Bardi(2011)}]{bardi2011explicit}
\text{Bardi, M.} (2011).
\newblock Explicit solutions of some linear-quadratic mean field games.
\newblock \textit{Networks and heterogeneous media}, \textbf{7} 243--261.

\bibitem[{Bardi and Priuli(2014)}]{bardi2014linear}
\text{Bardi, M.} and \text{Priuli, F.~S.} (2014).
\newblock Linear-quadratic $n$-person and mean-field games with ergodic cost.
\newblock \textit{SIAM Journal on Control and Optimization}, \textbf{52}
  3022--3052.

\bibitem[{Bauso et~al.(2016)Bauso, Tembine and Ba{\c{s}}ar}]{bauso2016robust}
\text{Bauso, D.}, \text{Tembine, H.} and \text{Ba{\c{s}}ar, T.} (2016).
\newblock Robust mean field games.
\newblock \textit{Dynamic games and applications}, \textbf{6} 277--303.

\bibitem[{Bensoussan et~al.(2017)Bensoussan, Chau, Lai and
  Yam}]{bensoussan2017linear}
\text{Bensoussan, A.}, \text{Chau, M.}, \text{Lai, Y.} and \text{Yam, S. C.~P.}
  (2017).
\newblock Linear-quadratic mean field {S}tackelberg games with state and
  control delays.
\newblock \textit{SIAM Journal on Control and Optimization}, \textbf{55}
  2748--2781.

\bibitem[{Bensoussan et~al.(2013)Bensoussan, Frehse and
  Yam}]{bensoussan2013mean}
\text{Bensoussan, A.}, \text{Frehse, J.} and \text{Yam, P.} (2013).
\newblock \textit{Mean field games and mean field type control theory}.
\newblock Springer.

\bibitem[{Bensoussan et~al.(2016)Bensoussan, Sung, Yam and
  Yung}]{bensoussan2016linear}
\text{Bensoussan, A.}, \text{Sung, K.}, \text{Yam, S. C.~P.} and \text{Yung,
  S.-P.} (2016).
\newblock Linear-quadratic mean field games.
\newblock \textit{Journal of Optimization Theory and Applications},
  \textbf{169} 496--529.

\bibitem[{Bhandari et~al.(2018)Bhandari, Russo and Singal}]{bhandari2018finite}
\text{Bhandari, J.}, \text{Russo, D.} and \text{Singal, R.} (2018).
\newblock A finite time analysis of temporal difference learning with linear
  function approximation.
\newblock \textit{arXiv preprint arXiv:1806.02450}.

\bibitem[{Bhatnagar et~al.(2009)Bhatnagar, Sutton, Ghavamzadeh and
  Lee}]{bhatnagar2009natural}
\text{Bhatnagar, S.}, \text{Sutton, R.~S.}, \text{Ghavamzadeh, M.} and
  \text{Lee, M.} (2009).
\newblock Natural actor--critic algorithms.
\newblock \textit{Automatica}, \textbf{45} 2471--2482.

\bibitem[{Biswas(2015)}]{biswas2015mean}
\text{Biswas, A.} (2015).
\newblock Mean field games with ergodic cost for discrete time markov
  processes.
\newblock \textit{arXiv preprint arXiv:1510.08968}.

\bibitem[{Borkar and Meyn(2000)}]{borkar2000ode}
\text{Borkar, V.~S.} and \text{Meyn, S.~P.} (2000).
\newblock The {ODE} method for convergence of stochastic approximation and
  reinforcement learning.
\newblock \textit{SIAM Journal on Control and Optimization}, \textbf{38}
  447--469.

\bibitem[{Bowling(2001)}]{bowling2001rational}
\text{Bowling, M.} (2001).
\newblock Rational and convergent learning in stochastic games.
\newblock In \textit{International Conference on Artificial Intelligence}.

\bibitem[{Bowling and Veloso(2000)}]{bowling2000analysis}
\text{Bowling, M.} and \text{Veloso, M.} (2000).
\newblock An analysis of stochastic game theory for multiagent reinforcement
  learning.
\newblock Tech. rep., Carnegie Mellon University.

\bibitem[{Bradtke(1993)}]{bradtke1993reinforcement}
\text{Bradtke, S.~J.} (1993).
\newblock Reinforcement learning applied to linear quadratic regulation.
\newblock In \textit{Advances in Neural Information Processing Systems}.

\bibitem[{Bradtke et~al.(1994)Bradtke, Ydstie and Barto}]{bradtke1994adaptive}
\text{Bradtke, S.~J.}, \text{Ydstie, B.~E.} and \text{Barto, A.~G.} (1994).
\newblock Adaptive linear quadratic control using policy iteration.
\newblock In \textit{American Control Conference}, vol.~3. IEEE.

\bibitem[{Briani and Cardaliaguet(2018)}]{briani2018stable}
\text{Briani, A.} and \text{Cardaliaguet, P.} (2018).
\newblock Stable solutions in potential mean field game systems.
\newblock \textit{Nonlinear Differential Equations and Applications},
  \textbf{25} 1.

\bibitem[{Busoniu et~al.(2008)Busoniu, Babuska and
  De~Schutter}]{busoniu2008comprehensive}
\text{Busoniu, L.}, \text{Babuska, R.} and \text{De~Schutter, B.} (2008).
\newblock A comprehensive survey of multiagent reinforcement learning.
\newblock \textit{IEEE Transactions on Systems, Man, and Cybernetics, Part C:
  Applications and Reviews}, \textbf{38} 156--172.

\bibitem[{Caines and Kizilkale(2017)}]{caines2017epsilon}
\text{Caines, P.~E.} and \text{Kizilkale, A.~C.} (2017).
\newblock $\epsilon $-nash equilibria for partially observed {LQG} mean field
  games with a major player.
\newblock \textit{IEEE Transactions on Automatic Control}, \textbf{62}
  3225--3234.

\bibitem[{Calderone(2017)}]{calderone2017models}
\text{Calderone, D.~J.} (2017).
\newblock \textit{Models of Competition for Intelligent Transportation
  Infrastructure: Parking, Ridesharing, and External Factors in Routing
  Decisions}.
\newblock University of California, Berkeley.

\bibitem[{Carmona and Delarue(2013)}]{carmona2013probabilistic}
\text{Carmona, R.} and \text{Delarue, F.} (2013).
\newblock Probabilistic analysis of mean-field games.
\newblock \textit{SIAM Journal on Control and Optimization}, \textbf{51}
  2705--2734.

\bibitem[{Carmona and Delarue(2018)}]{carmona2018probabilistic}
\text{Carmona, R.} and \text{Delarue, F.} (2018).
\newblock \textit{Probabilistic Theory of Mean Field Games with Applications
  {I}-{II}}.
\newblock Springer.

\bibitem[{Casgrain et~al.(2019)Casgrain, Ning and Jaimungal}]{casgrain2019deep}
\text{Casgrain, P.}, \text{Ning, B.} and \text{Jaimungal, S.} (2019).
\newblock Deep {Q}-learning for {N}ash equilibria: Nash-{DQN}.
\newblock \textit{arXiv preprint arXiv:1904.10554}.

\bibitem[{Castro and Meir(2010)}]{castro2010convergent}
\text{Castro, D.~D.} and \text{Meir, R.} (2010).
\newblock A convergent online single time scale actor critic algorithm.
\newblock \textit{Journal of Machine Learning Research}, \textbf{11} 367--410.

\bibitem[{Conitzer and Sandholm(2007)}]{conitzer2007awesome}
\text{Conitzer, V.} and \text{Sandholm, T.} (2007).
\newblock {AWESOME}: A general multiagent learning algorithm that converges in
  self-play and learns a best response against stationary opponents.
\newblock \textit{Machine Learning}, \textbf{67} 23--43.

\bibitem[{de~Cote et~al.(2006)de~Cote, Lazaric and Restelli}]{de2006learning}
\text{de~Cote, E.~M.}, \text{Lazaric, A.} and \text{Restelli, M.} (2006).
\newblock Learning to cooperate in multi-agent social dilemmas.
\newblock In \textit{International Conference on Autonomous Agents and
  Multiagent Systems}. ACM.

\bibitem[{De~Souza et~al.(1986)De~Souza, Gevers and Goodwin}]{de1986riccati}
\text{De~Souza, C.}, \text{Gevers, M.} and \text{Goodwin, G.} (1986).
\newblock Riccati equations in optimal filtering of nonstabilizable systems
  having singular state transition matrices.
\newblock \textit{IEEE Transactions on Automatic Control}, \textbf{31}
  831--838.

\bibitem[{Dean et~al.(2017)Dean, Mania, Matni, Recht and Tu}]{dean2017sample}
\text{Dean, S.}, \text{Mania, H.}, \text{Matni, N.}, \text{Recht, B.} and
  \text{Tu, S.} (2017).
\newblock On the sample complexity of the linear quadratic regulator.
\newblock \textit{arXiv preprint arXiv:1710.01688}.

\bibitem[{Dean et~al.(2018)Dean, Mania, Matni, Recht and Tu}]{dean2018regret}
\text{Dean, S.}, \text{Mania, H.}, \text{Matni, N.}, \text{Recht, B.} and
  \text{Tu, S.} (2018).
\newblock Regret bounds for robust adaptive control of the linear quadratic
  regulator.
\newblock In \textit{Advances in Neural Information Processing Systems}.

\bibitem[{Doerr et~al.(2018)Doerr, Linares, Zhu and Ferrari}]{doerr2018random}
\text{Doerr, B.}, \text{Linares, R.}, \text{Zhu, P.} and \text{Ferrari, S.}
  (2018).
\newblock Random finite set theory and optimal control for large spacecraft
  swarms.
\newblock \textit{arXiv preprint arXiv:1810.00696}.

\bibitem[{Du et~al.(2017)Du, Chen, Li, Xiao and Zhou}]{du2017stochastic}
\text{Du, S.~S.}, \text{Chen, J.}, \text{Li, L.}, \text{Xiao, L.} and
  \text{Zhou, D.} (2017).
\newblock Stochastic variance reduction methods for policy evaluation.
\newblock In \textit{Proceedings of the 34th International Conference on
  Machine Learning-Volume 70}. JMLR. org.

\bibitem[{Fang(2014)}]{fang2014lqr}
\text{Fang, J.} (2014).
\newblock The {LQR} controller design of two-wheeled self-balancing robot based
  on the particle swarm optimization algorithm.
\newblock \textit{Mathematical Problems in Engineering}, \textbf{2014}.

\bibitem[{Fazel et~al.(2018)Fazel, Ge, Kakade and Mesbahi}]{fazel2018global}
\text{Fazel, M.}, \text{Ge, R.}, \text{Kakade, S.~M.} and \text{Mesbahi, M.}
  (2018).
\newblock Global convergence of policy gradient methods for the linear
  quadratic regulator.
\newblock \textit{arXiv preprint arXiv:1801.05039}.

\bibitem[{Ganzfried and Sandholm(2009)}]{ganzfried2009computing}
\text{Ganzfried, S.} and \text{Sandholm, T.} (2009).
\newblock Computing equilibria in multiplayer stochastic games of imperfect
  information.
\newblock In \textit{Twenty-First International Joint Conference on Artificial
  Intelligence}.

\bibitem[{Gomes et~al.(2010)Gomes, Mohr and Souza}]{gomes2010discrete}
\text{Gomes, D.~A.}, \text{Mohr, J.} and \text{Souza, R.~R.} (2010).
\newblock Discrete time, finite state space mean field games.
\newblock \textit{Journal de math{\'e}matiques pures et appliqu{\'e}es},
  \textbf{93} 308--328.

\bibitem[{Gomes et~al.(2014)}]{gomes2014mean}
\text{Gomes, D.~A.} \text{et~al.} (2014).
\newblock Mean field games models—a brief survey.
\newblock \textit{Dynamic Games and Applications}, \textbf{4} 110--154.

\bibitem[{Gu{\'e}ant et~al.(2011)Gu{\'e}ant, Lasry and Lions}]{gueant2011mean}
\text{Gu{\'e}ant, O.}, \text{Lasry, J.-M.} and \text{Lions, P.-L.} (2011).
\newblock Mean field games and applications.
\newblock In \textit{Paris-Princeton lectures on mathematical finance 2010}.
  Springer, 205--266.

\bibitem[{Guo et~al.(2019)Guo, Hu, Xu and Zhang}]{guo2019learning}
\text{Guo, X.}, \text{Hu, A.}, \text{Xu, R.} and \text{Zhang, J.} (2019).
\newblock Learning mean-field games.
\newblock \textit{arXiv preprint arXiv:1901.09585}.

\bibitem[{Hardt et~al.(2016)Hardt, Ma and Recht}]{hardt2016gradient}
\text{Hardt, M.}, \text{Ma, T.} and \text{Recht, B.} (2016).
\newblock Gradient descent learns linear dynamical systems.
\newblock \textit{arXiv preprint arXiv:1609.05191}.

\bibitem[{Heinrich and Silver(2016)}]{heinrich2016deep}
\text{Heinrich, J.} and \text{Silver, D.} (2016).
\newblock Deep reinforcement learning from self-play in imperfect-information
  games.
\newblock \textit{arXiv preprint arXiv:1603.01121}.

\bibitem[{Hu and Wellman(1998)}]{hu1998multiagent}
\text{Hu, J.} and \text{Wellman, M.~P.} (1998).
\newblock Multiagent reinforcement learning: Theoretical framework and an
  algorithm.
\newblock In \textit{International Conference on Machine Learning}. Morgan
  Kaufmann Publishers Inc.

\bibitem[{Hu and Wellman(2003)}]{hu2003nash}
\text{Hu, J.} and \text{Wellman, M.~P.} (2003).
\newblock Nash {Q}-learning for general-sum stochastic games.
\newblock \textit{Journal of machine learning research}, \textbf{4} 1039--1069.

\bibitem[{Huang and Huang(2017)}]{huang2017robust}
\text{Huang, J.} and \text{Huang, M.} (2017).
\newblock Robust mean field linear-quadratic-gaussian games with unknown
  ${L}^{2}$-disturbance.
\newblock \textit{SIAM Journal on Control and Optimization}, \textbf{55}
  2811--2840.

\bibitem[{Huang and Li(2018)}]{huang2018linear}
\text{Huang, J.} and \text{Li, N.} (2018).
\newblock Linear--quadratic mean-field game for stochastic delayed systems.
\newblock \textit{IEEE Transactions on Automatic Control}, \textbf{63}
  2722--2729.

\bibitem[{Huang et~al.(2016{\natexlab{a}})Huang, Li and Wang}]{huang2016mean}
\text{Huang, J.}, \text{Li, X.} and \text{Wang, T.} (2016{\natexlab{a}}).
\newblock Mean-field linear-quadratic-{G}aussian ({LQG}) games for stochastic
  integral systems.
\newblock \textit{IEEE Transactions on Automatic Control}, \textbf{61}
  2670--2675.

\bibitem[{Huang et~al.(2016{\natexlab{b}})Huang, Wang and
  Wu}]{huang2016backward}
\text{Huang, J.}, \text{Wang, S.} and \text{Wu, Z.} (2016{\natexlab{b}}).
\newblock Backward mean-field linear-quadratic-{G}aussian ({LQG}) games: Full
  and partial information.
\newblock \textit{IEEE Transactions on Automatic Control}, \textbf{61}
  3784--3796.

\bibitem[{Huang et~al.(2003)Huang, Caines and
  Malham{\'e}}]{huang2003individual}
\text{Huang, M.}, \text{Caines, P.~E.} and \text{Malham{\'e}, R.~P.} (2003).
\newblock Individual and mass behaviour in large population stochastic wireless
  power control problems: centralized and nash equilibrium solutions.
\newblock In \textit{Conference on Decision and Control}. IEEE.

\bibitem[{Huang et~al.(2007)Huang, Caines and Malham{\'e}}]{huang2007large}
\text{Huang, M.}, \text{Caines, P.~E.} and \text{Malham{\'e}, R.~P.} (2007).
\newblock Large-population cost-coupled {LQG} problems with nonuniform agents:
  individual-mass behavior and decentralized $\varepsilon$-{N}ash equilibria.
\newblock \textit{IEEE transactions on automatic control}, \textbf{52}
  1560--1571.

\bibitem[{Huang et~al.(2006)Huang, Malham{\'e}, Caines et~al.}]{huang2006large}
\text{Huang, M.}, \text{Malham{\'e}, R.~P.}, \text{Caines, P.~E.} \text{et~al.}
  (2006).
\newblock Large population stochastic dynamic games: closed-loop
  {M}ckean-{V}lasov systems and the {N}ash certainty equivalence principle.
\newblock \textit{Communications in Information \& Systems}, \textbf{6}
  221--252.

\bibitem[{Huang and Zhou(2019)}]{huang2019linear}
\text{Huang, M.} and \text{Zhou, M.} (2019).
\newblock Linear quadratic mean field games: Asymptotic solvability and
  relation to the fixed point approach.
\newblock \textit{arXiv preprint arXiv:1903.08776}.

\bibitem[{Hughes et~al.(2018)Hughes, Leibo, Phillips, Tuyls, Due{\~n}ez-Guzman,
  Casta{\~n}eda, Dunning, Zhu, McKee, Koster et~al.}]{hughes2018inequity}
\text{Hughes, E.}, \text{Leibo, J.~Z.}, \text{Phillips, M.}, \text{Tuyls, K.},
  \text{Due{\~n}ez-Guzman, E.}, \text{Casta{\~n}eda, A.~G.}, \text{Dunning,
  I.}, \text{Zhu, T.}, \text{McKee, K.}, \text{Koster, R.} \text{et~al.}
  (2018).
\newblock Inequity aversion improves cooperation in intertemporal social
  dilemmas.
\newblock In \textit{Advances in Neural Information Processing Systems}.

\bibitem[{Jayakumar and Aditya(2019)}]{Subramanian2019Reinforce}
\text{Jayakumar, S.} and \text{Aditya, M.} (2019).
\newblock Reinforcement learning in stationary mean-field games.
\newblock In \textit{International Conference on Autonomous Agents and
  Multiagent Systems}.

\bibitem[{Kakade(2002)}]{kakade2002natural}
\text{Kakade, S.~M.} (2002).
\newblock A natural policy gradient.
\newblock In \textit{Advances in neural information processing systems}.

\bibitem[{Konda and Tsitsiklis(2000)}]{konda2000actor}
\text{Konda, V.~R.} and \text{Tsitsiklis, J.~N.} (2000).
\newblock Actor-critic algorithms.
\newblock In \textit{Advances in neural information processing systems}.

\bibitem[{Korda and La(2015)}]{korda2015td}
\text{Korda, N.} and \text{La, P.} (2015).
\newblock On {TD}(0) with function approximation: Concentration bounds and a
  centered variant with exponential convergence.
\newblock In \textit{International Conference on Machine Learning}.

\bibitem[{Kushner and Yin(2003)}]{kushner2003stochastic}
\text{Kushner, H.} and \text{Yin, G.~G.} (2003).
\newblock \textit{Stochastic approximation and recursive algorithms and
  applications}.
\newblock Springer Science \& Business Media.

\bibitem[{Lagoudakis and Parr(2002)}]{lagoudakis2002value}
\text{Lagoudakis, M.~G.} and \text{Parr, R.} (2002).
\newblock Value function approximation in zero-sum {M}arkov games.
\newblock In \textit{Uncertainty in Artificial Intelligence}.

\bibitem[{Lasry and Lions(2006{\natexlab{a}})}]{lasry2006jeux}
\text{Lasry, J.-M.} and \text{Lions, P.-L.} (2006{\natexlab{a}}).
\newblock Jeux {\`a} champ moyen. {I}--le cas stationnaire.
\newblock \textit{Comptes Rendus Math{\'e}matique}, \textbf{343} 619--625.

\bibitem[{Lasry and Lions(2006{\natexlab{b}})}]{lasry2006jeux2}
\text{Lasry, J.-M.} and \text{Lions, P.-L.} (2006{\natexlab{b}}).
\newblock Jeux {\`a} champ moyen. {II}--horizon fini et contr{\^o}le optimal.
\newblock \textit{Comptes Rendus Math{\'e}matique}, \textbf{343} 679--684.

\bibitem[{Lasry and Lions(2007)}]{lasry2007mean}
\text{Lasry, J.-M.} and \text{Lions, P.-L.} (2007).
\newblock Mean field games.
\newblock \textit{Japanese journal of mathematics}, \textbf{2} 229--260.

\bibitem[{LeCun et~al.(2015)LeCun, Bengio and Hinton}]{lecun2015deep}
\text{LeCun, Y.}, \text{Bengio, Y.} and \text{Hinton, G.} (2015).
\newblock Deep learning.
\newblock \textit{Nature}, \textbf{521} 436--444.

\bibitem[{Leibo et~al.(2017)Leibo, Zambaldi, Lanctot, Marecki and
  Graepel}]{leibo2017multi}
\text{Leibo, J.~Z.}, \text{Zambaldi, V.}, \text{Lanctot, M.}, \text{Marecki,
  J.} and \text{Graepel, T.} (2017).
\newblock Multi-agent reinforcement learning in sequential social dilemmas.
\newblock In \textit{International Conference on Autonomous Agents and
  Multiagent Systems}. International Foundation for Autonomous Agents and
  Multiagent Systems.

\bibitem[{Lewis et~al.(2012)Lewis, Vrabie and Syrmos}]{lewis2012optimal}
\text{Lewis, F.~L.}, \text{Vrabie, D.} and \text{Syrmos, V.~L.} (2012).
\newblock \textit{Optimal control}.
\newblock John Wiley \& Sons.

\bibitem[{Li et~al.(2017)Li, Zhang and Zhao}]{li2017connections}
\text{Li, S.}, \text{Zhang, W.} and \text{Zhao, L.} (2017).
\newblock Connections between mean-field game and social welfare optimization.
\newblock \textit{arXiv preprint arXiv:1703.10211}.

\bibitem[{Li and Zhang(2008)}]{li2008asymptotically}
\text{Li, T.} and \text{Zhang, J.-F.} (2008).
\newblock Asymptotically optimal decentralized control for large population
  stochastic multiagent systems.
\newblock \textit{IEEE Transactions on Automatic Control}, \textbf{53}
  1643--1660.

\bibitem[{Li(2018)}]{li2018deep}
\text{Li, Y.} (2018).
\newblock Deep reinforcement learning.
\newblock \textit{arXiv preprint arXiv:1810.06339}.

\bibitem[{Littman(1994)}]{littman1994markov}
\text{Littman, M.~L.} (1994).
\newblock Markov games as a framework for multi-agent reinforcement learning.
\newblock In \textit{Machine Learning Proceedings 1994}. Elsevier, 157--163.

\bibitem[{Littman(2001)}]{littman2001friend}
\text{Littman, M.~L.} (2001).
\newblock Friend-or-foe {Q}-learning in general-sum games.
\newblock In \textit{Proceedings of the Eighteenth International Conference on
  Machine Learning}.

\bibitem[{Liu et~al.(2015)Liu, Liu, Ghavamzadeh, Mahadevan and
  Petrik}]{liu2015finite}
\text{Liu, B.}, \text{Liu, J.}, \text{Ghavamzadeh, M.}, \text{Mahadevan, S.}
  and \text{Petrik, M.} (2015).
\newblock Finite-sample analysis of proximal gradient {TD} algorithms.
\newblock In \textit{Conference on Uncertainty in Artificial Intelligence}.

\bibitem[{Maei(2018)}]{maei2018convergent}
\text{Maei, H.~R.} (2018).
\newblock Convergent actor-critic algorithms under off-policy training and
  function approximation.
\newblock \textit{arXiv preprint arXiv:1802.07842}.

\bibitem[{Magnus(1979)}]{magnus1979expectation}
\text{Magnus, J.~R.} (1979).
\newblock The expectation of products of quadratic forms in normal variables:
  the practice.
\newblock \textit{Statistica Neerlandica}, \textbf{33} 131--136.

\bibitem[{Magnus et~al.(1978)}]{magnus1978moments}
\text{Magnus, J.~R.} \text{et~al.} (1978).
\newblock \textit{The moments of products of quadratic forms in normal
  variables}.
\newblock Univ., Instituut voor Actuariaat en Econometrie.

\bibitem[{Malik et~al.(2018)Malik, Pananjady, Bhatia, Khamaru, Bartlett and
  Wainwright}]{malik2018derivative}
\text{Malik, D.}, \text{Pananjady, A.}, \text{Bhatia, K.}, \text{Khamaru, K.},
  \text{Bartlett, P.~L.} and \text{Wainwright, M.~J.} (2018).
\newblock Derivative-free methods for policy optimization: Guarantees for
  linear quadratic systems.
\newblock \textit{arXiv preprint arXiv:1812.08305}.

\bibitem[{Mguni et~al.(2018)Mguni, Jennings and
  de~Cote}]{mguni2018decentralised}
\text{Mguni, D.}, \text{Jennings, J.} and \text{de~Cote, E.~M.} (2018).
\newblock Decentralised learning in systems with many, many strategic agents.
\newblock In \textit{Thirty-Second AAAI Conference on Artificial Intelligence}.

\bibitem[{Minciardi and Sacile(2011)}]{minciardi2011optimal}
\text{Minciardi, R.} and \text{Sacile, R.} (2011).
\newblock Optimal control in a cooperative network of smart power grids.
\newblock \textit{IEEE Systems Journal}, \textbf{6} 126--133.

\bibitem[{Moon and Ba{\c{s}}ar(2014)}]{moon2014discrete}
\text{Moon, J.} and \text{Ba{\c{s}}ar, T.} (2014).
\newblock Discrete-time {LQG} mean field games with unreliable communication.
\newblock In \textit{Conference on Decision and Control}. IEEE.

\bibitem[{Moon and Ba{\c{s}}ar(2018)}]{moon2018linear}
\text{Moon, J.} and \text{Ba{\c{s}}ar, T.} (2018).
\newblock Linear quadratic mean field stackelberg differential games.
\newblock \textit{Automatica}, \textbf{97} 200--213.

\bibitem[{Morav{\v{c}}{\'\i}k et~al.(2017)Morav{\v{c}}{\'\i}k, Schmid, Burch,
  Lis{\`y}, Morrill, Bard, Davis, Waugh, Johanson and
  Bowling}]{moravvcik2017deepstack}
\text{Morav{\v{c}}{\'\i}k, M.}, \text{Schmid, M.}, \text{Burch, N.},
  \text{Lis{\`y}, V.}, \text{Morrill, D.}, \text{Bard, N.}, \text{Davis, T.},
  \text{Waugh, K.}, \text{Johanson, M.} and \text{Bowling, M.} (2017).
\newblock Deepstack: Expert-level artificial intelligence in heads-up no-limit
  poker.
\newblock \textit{Science}, \textbf{356} 508--513.

\bibitem[{Nash(1951)}]{nash1951non}
\text{Nash, J.} (1951).
\newblock Non-cooperative games.
\newblock \textit{Annals of mathematics} 286--295.

\bibitem[{OpenAI(2018)}]{OpenAI_dota}
\text{OpenAI} (2018).
\newblock Openai five.
\newblock \url{https://blog.openai.com/openai-five/}.

\bibitem[{P{\'e}rolat et~al.(2016{\natexlab{a}})P{\'e}rolat, Piot, Geist,
  Scherrer and Pietquin}]{perolat2016softened}
\text{P{\'e}rolat, J.}, \text{Piot, B.}, \text{Geist, M.}, \text{Scherrer, B.}
  and \text{Pietquin, O.} (2016{\natexlab{a}}).
\newblock Softened approximate policy iteration for markov games.
\newblock In \textit{International Conference on Machine Learning}.

\bibitem[{P{\'e}rolat et~al.(2018)P{\'e}rolat, Piot and
  Pietquin}]{perolat2018actor}
\text{P{\'e}rolat, J.}, \text{Piot, B.} and \text{Pietquin, O.} (2018).
\newblock Actor-critic fictitious play in simultaneous move multistage games.
\newblock In \textit{International Conference on Artificial Intelligence and
  Statistics}.

\bibitem[{P{\'e}rolat et~al.(2016{\natexlab{b}})P{\'e}rolat, Piot, Scherrer and
  Pietquin}]{perolat2016use}
\text{P{\'e}rolat, J.}, \text{Piot, B.}, \text{Scherrer, B.} and
  \text{Pietquin, O.} (2016{\natexlab{b}}).
\newblock On the use of non-stationary strategies for solving two-player
  zero-sum {M}arkov games.
\newblock In \textit{International Conference on Artificial Intelligence and
  Statistics}.

\bibitem[{Perolat et~al.(2015)Perolat, Scherrer, Piot and
  Pietquin}]{perolat2015approximate}
\text{Perolat, J.}, \text{Scherrer, B.}, \text{Piot, B.} and \text{Pietquin,
  O.} (2015).
\newblock Approximate dynamic programming for two-player zero-sum {M}arkov
  games.
\newblock In \textit{International Conference on Machine Learning (ICML 2015)}.

\bibitem[{Peters and Schaal(2008)}]{peters2008natural}
\text{Peters, J.} and \text{Schaal, S.} (2008).
\newblock Natural actor-critic.
\newblock \textit{Neurocomputing}, \textbf{71} 1180--1190.

\bibitem[{Rudelson et~al.(2013)Rudelson, Vershynin et~al.}]{rudelson2013hanson}
\text{Rudelson, M.}, \text{Vershynin, R.} \text{et~al.} (2013).
\newblock Hanson-wright inequality and sub-{G}aussian concentration.
\newblock \textit{Electronic Communications in Probability}, \textbf{18}.

\bibitem[{Saldi et~al.(2018{\natexlab{a}})Saldi, Basar and
  Raginsky}]{saldi2018discrete}
\text{Saldi, N.}, \text{Basar, T.} and \text{Raginsky, M.}
  (2018{\natexlab{a}}).
\newblock Discrete-time risk-sensitive mean-field games.
\newblock \textit{arXiv preprint arXiv:1808.03929}.

\bibitem[{Saldi et~al.(2018{\natexlab{b}})Saldi, Basar and
  Raginsky}]{saldi2018markov}
\text{Saldi, N.}, \text{Basar, T.} and \text{Raginsky, M.}
  (2018{\natexlab{b}}).
\newblock Markov--{N}ash equilibria in mean-field games with discounted cost.
\newblock \textit{SIAM Journal on Control and Optimization}, \textbf{56}
  4256--4287.

\bibitem[{Saldi et~al.(2019)Saldi, Basar and Raginsky}]{saldi2019approximate}
\text{Saldi, N.}, \text{Basar, T.} and \text{Raginsky, M.} (2019).
\newblock Approximate {N}ash equilibria in partially observed stochastic games
  with mean-field interactions.
\newblock \textit{Mathematics of Operations Research}.

\bibitem[{Sandholm(2010)}]{sandholm2010population}
\text{Sandholm, W.~H.} (2010).
\newblock \textit{Population Games and Evolutionary Dynamics}.
\newblock MIT Press.

\bibitem[{Shalev-Shwartz et~al.(2016)Shalev-Shwartz, Shammah and
  Shashua}]{shalev2016safe}
\text{Shalev-Shwartz, S.}, \text{Shammah, S.} and \text{Shashua, A.} (2016).
\newblock Safe, multi-agent, reinforcement learning for autonomous driving.
\newblock \textit{arXiv preprint arXiv:1610.03295}.

\bibitem[{Shoham et~al.(2003)Shoham, Powers and Grenager}]{shoham2003multi}
\text{Shoham, Y.}, \text{Powers, R.} and \text{Grenager, T.} (2003).
\newblock Multi-agent reinforcement learning: a critical survey.

\bibitem[{Shoham et~al.(2007)Shoham, Powers and Grenager}]{shoham2007if}
\text{Shoham, Y.}, \text{Powers, R.} and \text{Grenager, T.} (2007).
\newblock If multi-agent learning is the answer, what is the question?
\newblock \textit{Artificial Intelligence}, \textbf{171} 365--377.

\bibitem[{Silver et~al.(2016)Silver, Huang, Maddison, Guez, Sifre, Van
  Den~Driessche, Schrittwieser, Antonoglou, Panneershelvam, Lanctot
  et~al.}]{silver2016mastering}
\text{Silver, D.}, \text{Huang, A.}, \text{Maddison, C.~J.}, \text{Guez, A.},
  \text{Sifre, L.}, \text{Van Den~Driessche, G.}, \text{Schrittwieser, J.},
  \text{Antonoglou, I.}, \text{Panneershelvam, V.}, \text{Lanctot, M.}
  \text{et~al.} (2016).
\newblock Mastering the game of {G}o with deep neural networks and tree search.
\newblock \textit{Nature}, \textbf{529} 484--489.

\bibitem[{Silver et~al.(2017)Silver, Schrittwieser, Simonyan, Antonoglou,
  Huang, Guez, Hubert, Baker, Lai, Bolton et~al.}]{silver2017mastering}
\text{Silver, D.}, \text{Schrittwieser, J.}, \text{Simonyan, K.},
  \text{Antonoglou, I.}, \text{Huang, A.}, \text{Guez, A.}, \text{Hubert, T.},
  \text{Baker, L.}, \text{Lai, M.}, \text{Bolton, A.} \text{et~al.} (2017).
\newblock Mastering the game of {G}o without human knowledge.
\newblock \textit{Nature}, \textbf{550} 354.

\bibitem[{Sutton and Barto(2018)}]{sutton2018reinforcement}
\text{Sutton, R.~S.} and \text{Barto, A.~G.} (2018).
\newblock \textit{Reinforcement learning: An introduction}.
\newblock MIT Press.

\bibitem[{Sutton et~al.(2009{\natexlab{a}})Sutton, Maei, Precup, Bhatnagar,
  Silver, Szepesv{\'a}ri and Wiewiora}]{sutton2009fast}
\text{Sutton, R.~S.}, \text{Maei, H.~R.}, \text{Precup, D.}, \text{Bhatnagar,
  S.}, \text{Silver, D.}, \text{Szepesv{\'a}ri, C.} and \text{Wiewiora, E.}
  (2009{\natexlab{a}}).
\newblock Fast gradient-descent methods for temporal-difference learning with
  linear function approximation.
\newblock In \textit{Proceedings of the 26th Annual International Conference on
  Machine Learning}. ACM.

\bibitem[{Sutton et~al.(2009{\natexlab{b}})Sutton, Maei and
  Szepesv{\'a}ri}]{sutton2009convergent}
\text{Sutton, R.~S.}, \text{Maei, H.~R.} and \text{Szepesv{\'a}ri, C.}
  (2009{\natexlab{b}}).
\newblock A convergent $ o (n) $ temporal-difference algorithm for off-policy
  learning with linear function approximation.
\newblock In \textit{Advances in neural information processing systems}.

\bibitem[{Sutton et~al.(2000)Sutton, McAllester, Singh and
  Mansour}]{sutton2000policy}
\text{Sutton, R.~S.}, \text{McAllester, D.~A.}, \text{Singh, S.~P.} and
  \text{Mansour, Y.} (2000).
\newblock Policy gradient methods for reinforcement learning with function
  approximation.
\newblock In \textit{Advances in neural information processing systems}.

\bibitem[{Sznitman(1991)}]{sznitman1991topics}
\text{Sznitman, A.-S.} (1991).
\newblock Topics in propagation of chaos.
\newblock In \textit{Ecole d'{\'e}t{\'e} de probabilit{\'e}s de Saint-Flour
  XIX—1989}. Springer, 165--251.

\bibitem[{Tembine and Huang(2011)}]{tembine2011mean}
\text{Tembine, H.} and \text{Huang, M.} (2011).
\newblock Mean field difference games: {M}ckean-{V}lasov dynamics.
\newblock In \textit{Conference on Decision and Control and European Control
  Conference}. IEEE.

\bibitem[{Tembine et~al.(2014)Tembine, Zhu and Ba{\c{s}}ar}]{tembine2014risk}
\text{Tembine, H.}, \text{Zhu, Q.} and \text{Ba{\c{s}}ar, T.} (2014).
\newblock Risk-sensitive mean-field games.
\newblock \textit{IEEE Transactions on Automatic Control}, \textbf{59}
  835--850.

\bibitem[{Tu and Recht(2017)}]{tu2017least}
\text{Tu, S.} and \text{Recht, B.} (2017).
\newblock Least-squares temporal difference learning for the linear quadratic
  regulator.
\newblock \textit{arXiv preprint arXiv:1712.08642}.

\bibitem[{Tu and Recht(2018)}]{tu2018gap}
\text{Tu, S.} and \text{Recht, B.} (2018).
\newblock The gap between model-based and model-free methods on the linear
  quadratic regulator: An asymptotic viewpoint.
\newblock \textit{arXiv preprint arXiv:1812.03565}.

\bibitem[{uz~Zaman et~al.(2019)uz~Zaman, Zhang, Miehling and
  Basar}]{kaiqing2019approximate}
\text{uz~Zaman, M.~A.}, \text{Zhang, K.}, \text{Miehling, E.} and \text{Basar,
  T.} (2019).
\newblock Approximate equilibrium computation for discrete-time
  linear-quadratic mean-field games.
\newblock \textit{Manuscript}.

\bibitem[{Vinyals et~al.(2019)Vinyals, Babuschkin, Chung, Mathieu, Jaderberg,
  Czarnecki, Dudzik, Huang, Georgiev, Powell et~al.}]{vinyals2019alphastar}
\text{Vinyals, O.}, \text{Babuschkin, I.}, \text{Chung, J.}, \text{Mathieu,
  M.}, \text{Jaderberg, M.}, \text{Czarnecki, W.}, \text{Dudzik, A.},
  \text{Huang, A.}, \text{Georgiev, P.}, \text{Powell, R.} \text{et~al.}
  (2019).
\newblock Alphastar: Mastering the real-time strategy game starcraft ii.

\bibitem[{Wai et~al.(2018)Wai, Yang, Wang and Hong}]{wai2018multi}
\text{Wai, H.-T.}, \text{Yang, Z.}, \text{Wang, P.~Z.} and \text{Hong, M.}
  (2018).
\newblock Multi-agent reinforcement learning via double averaging primal-dual
  optimization.
\newblock In \textit{Advances in Neural Information Processing Systems}.

\bibitem[{Wang and Zhang(2012)}]{wang2012mean}
\text{Wang, B.-C.} and \text{Zhang, J.-F.} (2012).
\newblock Mean field games for large-population multiagent systems with
  {M}arkov jump parameters.
\newblock \textit{SIAM Journal on Control and Optimization}, \textbf{50}
  2308--2334.

\bibitem[{Wang et~al.(2017{\natexlab{a}})Wang, Zhang, Yuan
  et~al.}]{wang2017display}
\text{Wang, J.}, \text{Zhang, W.}, \text{Yuan, S.} \text{et~al.}
  (2017{\natexlab{a}}).
\newblock Display advertising with real-time bidding ({RTB}) and behavioural
  targeting.
\newblock \textit{Foundations and Trends{\textregistered} in Information
  Retrieval}, \textbf{11} 297--435.

\bibitem[{Wang et~al.(2017{\natexlab{b}})Wang, Chen, Liu, Ma and
  Liu}]{wang2017finite}
\text{Wang, Y.}, \text{Chen, W.}, \text{Liu, Y.}, \text{Ma, Z.-M.} and
  \text{Liu, T.-Y.} (2017{\natexlab{b}}).
\newblock Finite sample analysis of the {GTD} policy evaluation algorithms in
  {M}arkov setting.
\newblock In \textit{Advances in Neural Information Processing Systems}.

\bibitem[{Watkins and Dayan(1992)}]{watkins1992q}
\text{Watkins, C.~J.} and \text{Dayan, P.} (1992).
\newblock Q-learning.
\newblock \textit{Machine learning}, \textbf{8} 279--292.

\bibitem[{Wei et~al.(2017)Wei, Hong and Lu}]{wei2017online}
\text{Wei, C.-Y.}, \text{Hong, Y.-T.} and \text{Lu, C.-J.} (2017).
\newblock Online reinforcement learning in stochastic games.
\newblock In \textit{Advances in Neural Information Processing Systems}.

\bibitem[{Yang and Gu(2004)}]{yang2004multiagent}
\text{Yang, E.} and \text{Gu, D.} (2004).
\newblock Multiagent reinforcement learning for multi-robot systems: A survey.
\newblock \textit{Manuscript}.

\bibitem[{Yang et~al.(2018{\natexlab{a}})Yang, Ye, Trivedi, Xu and
  Zha}]{yang2018deep}
\text{Yang, J.}, \text{Ye, X.}, \text{Trivedi, R.}, \text{Xu, H.} and
  \text{Zha, H.} (2018{\natexlab{a}}).
\newblock Deep mean field games for learning optimal behavior policy of large
  populations.
\newblock In \textit{International Conference on Learning Representations}.

\bibitem[{Yang et~al.(2018{\natexlab{b}})Yang, Luo, Li, Zhou, Zhang and
  Wang}]{yang2018mean}
\text{Yang, Y.}, \text{Luo, R.}, \text{Li, M.}, \text{Zhou, M.}, \text{Zhang,
  W.} and \text{Wang, J.} (2018{\natexlab{b}}).
\newblock Mean field multi-agent reinforcement learning.
\newblock \textit{arXiv preprint arXiv:1802.05438}.

\bibitem[{Yang et~al.(2019)Yang, Chen, Hong and Wang}]{yyy2019aclqr}
\text{Yang, Z.}, \text{Chen, Y.}, \text{Hong, M.} and \text{Wang, Z.} (2019).
\newblock On the global convergence of actor-critic: A case for linear
  quadratic regulator with ergodic cost.
\newblock \textit{arXiv preprint arXiv:1907.06246}.

\bibitem[{Yu(2017)}]{yu2017convergence}
\text{Yu, H.} (2017).
\newblock On convergence of some gradient-based temporal-differences algorithms
  for off-policy learning.
\newblock \textit{arXiv preprint arXiv:1712.09652}.

\bibitem[{Zhang et~al.(2018)Zhang, Yang, Liu, Zhang and
  Ba{\c{s}}ar}]{zhang2018finite}
\text{Zhang, K.}, \text{Yang, Z.}, \text{Liu, H.}, \text{Zhang, T.} and
  \text{Ba{\c{s}}ar, T.} (2018).
\newblock Finite-sample analyses for fully decentralized multi-agent
  reinforcement learning.
\newblock \textit{arXiv preprint arXiv:1812.02783}.

\bibitem[{Zhou and Li(2000)}]{zhou2000continuous}
\text{Zhou, X.~Y.} and \text{Li, D.} (2000).
\newblock Continuous-time mean-variance portfolio selection: A stochastic {LQ}
  framework.
\newblock \textit{Applied Mathematics and Optimization}, \textbf{42} 19--33.

\bibitem[{Ziebart et~al.(2008)Ziebart, Maas, Bagnell and
  Dey}]{ziebart2008maximum}
\text{Ziebart, B.~D.}, \text{Maas, A.~L.}, \text{Bagnell, J.~A.} and \text{Dey,
  A.~K.} (2008).
\newblock Maximum entropy inverse reinforcement learning.
\newblock In \textit{AAAI Conference on Artificial Intelligence}, vol.~3.

\bibitem[{Zou et~al.(2019)Zou, Xu and Liang}]{zou2019finite}
\text{Zou, S.}, \text{Xu, T.} and \text{Liang, Y.} (2019).
\newblock Finite-sample analysis for {SARSA} and {Q}-learning with linear
  function approximation.
\newblock \textit{arXiv preprint arXiv:1902.02234}.

\end{thebibliography}

\newpage
\appendix


\section{Notations in the Appendix}
In the proof, for convenience, for any invertible matrix $M$, we denote by $M^{-\top} = (M^{-1})^\top = (M^\top)^{-1}$ and $\|M\|_\F$ the Frobenius norm. We also denote by $\svec(M)$ the symmetric vectorization of the symmetric matrix $M$, which is the vectorization of the upper triangular matrix of the symmetric matrix $M$, with off-diagonal entries scaled by $\sqrt{2}$.  We denote by $\smat(\cdot)$ the inverse operation.  For any matrices $G$ and $H$, we denote by $G\otimes H$ the Kronecker product, and $G\otimes_s H$ the symmetric Kronecker product,  which is defined as a mapping on a vector $\svec(M)$ such that $(G\otimes_s H)\svec(M) = 1/2\cdot \svec(HMG^\top + GMH^\top)$. 

For notational simplicity, we write $\EE_{\pi}(\cdot)$ to emphasize that the expectation is taken following the policy $\pi$. 

\section{Auxiliary Algorithms and Analysis}

\subsection{Results in D-LQR} \label{sec:res_lqr}

In this section, we provide auxiliary results in analyzing Problem \ref{prob:slqr}.  First, we introduce the value functions of the Markov decision process (MDP) induced by Problem \ref{prob:slqr}. 
We define the state- and action-value functions $V_{K,b}(x)$ and $Q_{K,b}(x,u)$ as follows
\#
& V_{K,b}(x) = \sum_{t = 0}^\infty \Bigl\{ \EE\bigl[ c(x_t, u_t) \given x_0 = x \bigr] - J(K,b) \Bigr\}, \label{eq:val_funcv} \\
& Q_{K,b}(x,u) = c(x,u) - J(K,b) + \EE\bigl[ V_{K,b}(x_1)\given x_0 = x, u_0 = u \bigr], \label{eq:val_funcq}
\#
where $x_t$ follows the state transition, and $u_t$ follows the policy $\pi_{K,b}$ given $x_t$. In other words, we have $u_{t} = -Kx_t + b + \sigma\eta_t$, where $\eta_t\sim \mathcal N(0, I )$.  The following proposition establishes the close forms of these value functions.

\begin{proposition}\label{prop:val_func_form}
The state-value function $V_{K,b}(x)$ takes the form of
\#\label{eq:4a}
V_{K,b}(x) = x^\top P_K x - \tr(P_K \Phi_K)  +  2 f_{K,b}^\top (x -\mu_{K,b}) - \mu_{K,b}^\top P_K\mu_{K,b},
\#
and the action-value function $Q_{K,b}(x,u)$ takes the form of
\#\label{eq:4b}
Q_{K,b}(x,u) &= \begin{pmatrix}
x\\
u
\end{pmatrix}^\top
\Upsilon_K
\begin{pmatrix}
x\\
u
\end{pmatrix} + 2\begin{pmatrix}
p_{K,b}\\
q_{K,b}
\end{pmatrix}^\top \begin{pmatrix}
x\\
u
\end{pmatrix}   - \tr(P_K\Phi_K)  - \sigma^2\cdot \tr(R + P_K B B^\top) \notag\\
& \qquad  - b^\top Rb +2b^\top RK\mu_{K,b}  - \mu_{K,b}^\top (Q+K^\top RK + P_K)\mu_{K,b}    \notag\\
& \qquad + 2f_{K,b}^\top\bigl[ (\overline A\mu+d)-\mu_{K,b}\bigr]  + (\overline A\mu+d)^\top P_K (\overline A\mu+d), 
\# 
where $f_{K,b} = (I-A+BK)^{-\top} [  (A-BK)^\top P_K (Bb+\overline A\mu + d) - K^\top Rb ]$, and $\Upsilon_K$, $p_{K,b}$, and $q_{K,b}$ are defined in \eqref{eq:def_upsilon}. 
\end{proposition}
\begin{proof}
See \S\ref{proof:prop:val_func_form} for a detailed proof. 
\end{proof}

By Proposition \ref{prop:val_func_form}, we know that $V_{K, b}(x)$ is quadratic in $x$, while $Q_{K,b}(x,u)$ is quadratic in $(x^\top, u^\top)^\top$.   Now, we show that \eqref{eq:a1} holds. 

\begin{proposition}\label{prop:cost_form}
The expected total cost $J(K,b)$ defined in Problem \ref{prob:slqr} takes the form of
\$
J(K,b) = J_1(K) + J_2(K,b) + \sigma^2 \cdot \tr(R) + \mu^\top \overline Q \mu,
\$
where
\$
& J_1(K) =\tr\bigl[(Q + K^\top R K)\Phi_K\bigr]  = \tr(P_K \Psi_\epsilon),\notag\\ 
& J_2(K,b) = \begin{pmatrix}
\mu_{K,b}\\
b
\end{pmatrix}^\top
\begin{pmatrix}
Q + K^\top RK & -K^\top R\\
-R K & R
\end{pmatrix}
 \begin{pmatrix}
\mu_{K,b}\\
b
\end{pmatrix}.
\$
Here $\mu_{K,b}$ is defined in \eqref{eq:f2q}, $\Phi_K$ is defined in \eqref{eq:f2p}, and $P_K$ is defined in \eqref{eq:bellman}.
\end{proposition}
\begin{proof}
See \S\ref{proof:prop:cost_form} for a detailed proof. 
\end{proof}

The following proposition establishes the gradients of $J_1(K)$ and $J_2(K,b)$, respectively.

\begin{proposition}\label{prop:pg}
The gradient of $J_1(K)$ and the gradient of $J_2(K,b)$ with respect to $b$ take the forms of
\$
&\nabla_K J_1(K) = 2(\Upsilon_K^{22} K  - \Upsilon_K^{21})\cdot \Phi_K,\qquad \nabla_b J_2(K,b) = 2\bigl[\Upsilon_K^{22}(-K\mu_{K, b} + b) + \Upsilon_{K}^{21}\mu_{K,b} + q_{K,b}\bigr], 
\$
where $\Upsilon_K$ and $q_{K,b}$ are defined in \eqref{eq:def_upsilon}. 
\end{proposition}
\begin{proof}
See \S\ref{proof:prop:pg} for a detailed proof. 
\end{proof}

The following theorem establishes the convergence of Algorithm \ref{algo:ac_lqr}.   

\begin{theorem}[Convergence of Algorithm \ref{algo:ac_lqr}]\label{thm:ac}
Assume that $\rho(A-BK_0) < 1$. Let $\varepsilon > 0$ be a sufficiently small tolerance.  We set
\$
& \gamma\leq \bigl[\|R\|_2 + \|B\|_2^2\cdot J(K_0, b_0)\cdot \sigma_{\min}^{-1}(\Psi_\epsilon)\bigr]^{-1}, \notag\\
& N \geq C\cdot \|\Phi_{K^*}\|_2\cdot\gamma^{-1} \cdot \log\Bigl\{ 4\bigl[J(K_0,b_0) - J(K^*, b^*)\bigr]\cdot\varepsilon^{-1} \Bigr\}, \notag\\
& T_n\geq \poly \bigl( \|K_n\|_\F, \|b_0\|_2, \|\mu\|_2, J(K_0, b_0) \bigr) \cdot \lambda_{K_n}^{-4} \cdot \bigl[1-\rho(A-BK_n)\bigr]^{-9}\cdot \varepsilon^{-5},\notag\\
& \tilde T_n \geq \poly \bigl( \|K_n\|_\F, \|b_0\|_2, \|\mu\|_2, J(K_0, b_0) \bigr) \cdot \lambda_{K_n}^{-2}\cdot \bigl[1-\rho(A-BK_n)\bigr]^{-12}\cdot \varepsilon^{-12},\notag\\
& \gamma_{n,t} = \gamma_0\cdot t^{-1/2}, \notag\\
& \gamma^b\leq\min\Bigl\{1 - \rho(A-BK_N), \bigl[1-\rho(A-BK_N)\bigr]^{-2} \cdot \bigl(  \|B\|_2^2\cdot \|K_N\|_2^2\cdot \|R\|_2 + \|B\|_2^2\cdot \|Q\|_2 \bigr)\Bigr\}, \notag\\
& H \geq C_0\cdot \nu_{K_N}^{-1}\cdot (\gamma^b)^{-1} \cdot \log\Bigl\{ 4\bigl[J(K_N, b_0) - J(K_N, b^{K_N})\bigr]\cdot\varepsilon^{-1} \Bigr\}, \notag\\
& T^b_{h} \geq \poly\bigl( \|K_N\|_\F, \|b_{h}\|_2, \|\mu\|_2, J(K_N, b_{0}) \bigr)\cdot \lambda_{K_N}^{-4}\cdot \nu_{K_N}^{-4}\cdot \bigl[1-\rho(A-BK_N)\bigr]^{-11}\cdot \varepsilon^{-5},\\
& \tilde T^b_{h} \geq \poly\bigl( \|K_N\|_\F, \|b_{h}\|_2, \|\mu\|_2, J(K_N, b_{0}) \bigr)\cdot \lambda_{K_N}^{-4}\cdot \nu_{K_N}^{-2}\cdot \bigl[1-\rho(A-BK_N)\bigr]^{-17}\cdot \varepsilon^{-8},\\
& \gamma^b_{h,t} = \gamma_0\cdot t^{-1/2},
\$
where $C$, $C_0$, and $\gamma_0$ are positive absolute constants,  $\{K_n\}_{n\in[N]}$ and $\{b_h\}_{h\in[H]}$ are the sequences generated by Algorithm \ref{algo:ac_lqr},   $\lambda_{K_n}$ is specified in Proposition \ref{prop:invert_theta},  and $\nu_{K_N}$ is specified in Proposition \ref{prop:convex_J2}. Then it holds with probability at least $1 - \varepsilon^{10}$ that
\$
& J(K_N, b_H) - J(K^*, b^*) < \varepsilon,  \qquad \|b_H - b^*\|_2\leq M_b(\mu) \cdot \varepsilon^{1/2}, \notag\\
& \|K_N - K^*\|_\F\leq \bigl[{\sigma_{\min}^{-1}(\Psi_\epsilon)\cdot \sigma_{\min}^{-1}(R) \cdot \varepsilon}\bigr]^{1/2}, \qquad \| \hat \mu_{K_N,b_H} - \mu_{K^*, b^*} \|_2 \leq \varepsilon,
\$
where $M_b(\mu)$ is defined in \eqref{eq:def-mb}. 
\end{theorem}

\begin{proof}
See \S\ref{proof:thm:ac} for a detailed proof. 
\end{proof}

By Theorem \ref{thm:ac}, given any mean-field state $\mu$, Algorithm \ref{algo:ac_lqr} converges linearly to the optimal policy $\pi^*_\mu$ of Problem \ref{prob:slqr}.

\subsection{Primal-Dual Policy Evaluation Algorithm}\label{sec:pe_algo}

Note that the critic update steps in Algorithm \ref{algo:ac_lqr}  are built upon the estimators of the matrix $\Upsilon_K$ and the vector $q_{K,b}$.  We now derive a policy evaluation algorithm to establish the estimators of $\Upsilon_K$ and $q_{K,b}$, which is based on gradient temporal difference algorithm \citep{sutton2009fast}. 

We define the feature vector as 
\#\label{eq:def_feature}
\psi (x,u) = \begin{pmatrix}
\varphi(x,u)\\
x - \mu_{K,b}\\
u -  (-K\mu_{K,b} + b)\\
\end{pmatrix},
\#
where
\$
\varphi(x,u) = \svec\Biggl[
\begin{pmatrix}
x - \mu_{K,b}\\
u -  (-K\mu_{K,b} + b)
\end{pmatrix}
\begin{pmatrix}
x - \mu_{K,b}\\
u -  (-K\mu_{K,b} + b)
\end{pmatrix}^\top\Biggr]. 
\$
Recall $\svec(M)$ gives the symmetric vectorization of the symmetric matrix $M$. 
We also define
\#\label{eq:q1}
\alpha_{K,b} = \begin{pmatrix}
\svec(\Upsilon_K )\\
\Upsilon_K 
\begin{pmatrix}
\mu_{K,b}\\
-K\mu_{K,b} + b
\end{pmatrix} + \begin{pmatrix}
p_{K,b}\\
q_{K,b}
\end{pmatrix}
\end{pmatrix},
\#
where $\Upsilon_K$, $p_{K,b}$, and $q_{K,b}$ are defined in \eqref{eq:def_upsilon}.  To estimate $\Upsilon_K$ and $q_{K,b}$, it suffices to estimate  $\alpha_{K,b}$. 
Meanwhile, we define 
\#\label{eq:q2}
\Theta_{K,b} = \EE_{\pi_{K,b}}\Bigl\{\psi(x,u)\bigl[ \psi(x,u) - \psi(x',u') \bigr] ^\top   \Bigr\},
\#
where $(x',u')$ is the state-action pair after $(x,u)$ following the policy $\pi_{K,b}$ and the state transition. 
The following proposition characterizes the connection between $\Theta_{K,b}$ and $\alpha_{K,b}$.

\begin{proposition}\label{prop:bellman_compact}
It holds that
\$ 
\begin{pmatrix}
1 & 0\\
\EE_{\pi_{K,b}}\bigl[ \psi(x,u) \bigr] & \Theta_{K,b}
\end{pmatrix}
\begin{pmatrix}
J(K,b)\\
\alpha_{K,b}
\end{pmatrix} = \begin{pmatrix}
J(K,b)\\
\EE_{\pi_{K,b}}\bigl[ c(x,u) \psi(x,u) \bigr]
\end{pmatrix},
\$
where  $\psi(x,u)$ is defined in \eqref{eq:def_feature},  $\alpha_{K,b}$ is defined in \eqref{eq:q1}, and  $\Theta_{K,b}$ is defined in \eqref{eq:q2}. 
\end{proposition} 
\begin{proof}
See \S\ref{proof:prop:bellman_compact} for a detailed proof. 
\end{proof}

By Proposition \ref{prop:bellman_compact}, to obtain $\alpha_{K,b}$, it suffices to solve the following linear system in $\zeta = (\zeta_1, \zeta_2^\top)^\top$,
\#\label{eq:linear_sys}
\tilde \Theta_{K,b} \cdot  \zeta
= \begin{pmatrix}
J(K,b)\\
\EE_{\pi_{K,b}}\bigl[ c(x,u) \psi(x,u) \bigr]
\end{pmatrix},
\#
where for notational convenience, we define
\#\label{eq:def-tilde-theta}
\tilde \Theta_{K,b} = \begin{pmatrix}
1 & 0\\
\EE_{\pi_{K,b}}\bigl[ \psi(x,u) \bigr] & \Theta_{K,b}
\end{pmatrix}. 
\# 
    The following proposition shows that  $\Theta_{K,b}$ is invertible. 

\begin{proposition}\label{prop:invert_theta}
If $\rho(A-BK) < 1$, then the matrix $\Theta_{K,b}$ is invertible, and $\|\Theta_{K,b}\|_2\leq 4 ( 1 + \|K\|_\F^2 )^2\cdot \|\Phi_K\|_2^2$.  Also,  $\sigma_{\min}(\tilde \Theta_{K,b})\geq \lambda_K$, where $\lambda_K$ only depends on $\|K\|_2$ and $\rho(A-BK)$. 
\end{proposition}
\begin{proof}
See \S\ref{proof:prop:invert_theta} for a detailed proof. 
\end{proof}

By Proposition \ref{prop:invert_theta}, $\Theta_{K,b}$ is invertible.  Therefore,  \eqref{eq:linear_sys} admits the unique solution $\zeta_{K,b} = (J(K,b), \alpha_{K,b}^\top)^\top$.

Now, we present the primal-dual gradient temporal difference algorithm.

\vskip5pt
\noindent\textbf{Primal-Dual Gradient Method.}  Instead of solving \eqref{eq:linear_sys} directly, we minimize the following loss function with respect to $\zeta = ((\zeta^1)^\top, (\zeta^2)^\top)$,
\#\label{eq:esam1}
 \bigl[\zeta^1 - J(K,b)\bigr]^2 + \Bigl\| \EE_{\pi_{K,b}}\bigl[\psi(x,u)\bigr]  \zeta^1+ \Theta_{K,b} \zeta^2 -\EE_{\pi_{K,b}}\bigl[ c(x,u) \psi(x,u) \bigr] \Bigr\|_2^2. 
\#
By Fenchel's duality, the minimization of \eqref{eq:esam1} is equivalent to the following primal-dual min-max problem,
\#\label{eq:minmax_pe}
\min_{\zeta\in\cV_\zeta}\max_{\xi\in\cV_\xi} F(\zeta, \xi)& = \Bigl\{  \EE_{\pi_{K,b}}\bigl[\psi(x,u)\bigr] \zeta^1+ \Theta_{K,b} \zeta^2 - \EE_{\pi_{K,b}}\bigl[ c(x,u) \psi(x,u) \bigr]\Bigr\}^\top \xi^2\\
&\qquad  + \bigl[\zeta^1 - J(K,b)\bigr] \cdot \xi^1 -  \|\xi\|_2^2/2,\notag
\#
where we restrict the primal variable $\zeta$ in a compact set $\cV_\zeta$ and the dual variable $\xi$ in a compact set $\cV_\xi$, which are specified in Definition  \ref{assum:proj}.  It holds that
\#\label{eq:grad-forms-pe}
& \nabla_{\zeta^1} F = \xi^1 + \EE_{\pi_{K,b}}\bigl[ \psi(x,u) \bigr]^\top \xi^2, \qquad  \nabla_{\zeta^2} F = \Theta_{K,b}^\top \xi^2, \qquad  \nabla_{\xi^1} F = \zeta^1 - J(K,b)-\xi^1, \quad \notag\\
& \nabla_{\xi^2} F  = \EE_{\pi_{K,b}}\bigl[\psi(x,u)\bigr] \zeta^1+ \Theta_{K,b} \zeta^2  - \EE_{\pi_{K,b}}\bigl[ c(x,u) \psi(x,u) \bigr] - \xi^2.
\#
The primal-dual gradient method updates $\zeta$ and $\xi$ via
\#\label{eq:pe-updates}
& \zeta^1 \gets \zeta^1 - \gamma \cdot \nabla_{\zeta^1} F(\zeta, \xi), \qquad \zeta^2 \gets \zeta^2 - \gamma \cdot \nabla_{\zeta^2} F(\zeta, \xi)\notag \\
& \xi^1 \gets \xi^1 - \gamma \cdot \nabla_{\xi^1} F(\zeta, \xi), \qquad \xi^2 \gets \xi^2 - \gamma \cdot \nabla_{\xi^2} F(\zeta, \xi). 
\#

\vskip5pt
\noindent\textbf{Estimation of Mean-Field State $\mu_{K,b}$.}  To utilize the primal-dual gradient method in \eqref{eq:pe-updates}, it remains to evaluate the feature vector $\psi(x,u)$.  Note that by \eqref{eq:def_feature}, the evaluation of the feature vector $\psi(x,u)$ requires the mean-field state $\mu_{K,b}$. In what follows, we establish the estimator $\hat \mu_{K,b}$ of the mean-field state $\mu_{K,b}$ by simulating the MDP following the policy $\pi_{K,b}$ for $\tilde T$ steps, and calculate the estimated feature vector $\hat \psi(x,u)$ by 
\#\label{eq:est_def_feature}
\hat \psi (x,u) = \begin{pmatrix}
\hat \varphi(x,u)\\
x - \hat \mu_{K,b}\\
u -  (-K\hat \mu_{K,b} + b)\\
\end{pmatrix},
\# 
where $\hat \varphi(x,u)$ takes the form of
\$
\hat \varphi(x,u) = \svec\Biggl[
\begin{pmatrix}
x - \hat \mu_{K,b}\\
u -  (-K\hat \mu_{K,b} + b)
\end{pmatrix}
\begin{pmatrix}
x - \hat \mu_{K,b}\\
u -  (-K\hat \mu_{K,b} + b)
\end{pmatrix}^\top\Biggr]. 
\$

\vskip5pt
We now define the sets $\cV_\zeta$ and $\cV_\xi$ in \eqref{eq:minmax_pe}. 

\begin{definition}\label{assum:proj}
Given $K_0$ and $b_0$ such that $\rho(A-BK_0) < 1$ and $J(K_0, b_0) < \infty$, we define the sets $\cV_\zeta$ and $\cV_\xi$ as
\$
& \cV_\zeta = \Bigl\{ \zeta\colon 0\leq \zeta^1\leq J(K_0, b_0), \|\zeta^2\|_2 \leq M_{\zeta,1} + M_{\zeta,2}\cdot(1+ \|K\|_\F)\cdot \bigl[ 1-\rho(A-BK) \bigr]^{-1}  \Bigr\}, \\
& \cV_\xi = \Bigl\{ \xi\colon |\xi^1|\leq J(K_0, b_0), \|\xi^2\|_2\leq M_\xi\cdot \bigl(1+ \|K\|_\F^2 \bigr)^3 \cdot \bigl[ 1-\rho(A-BK) \bigr]^{-1} \Bigr\}. 
\$
Here $M_{\zeta,1}$, $M_{\zeta,2}$, and $M_\xi$ are constants independent of $K$ and $b$, which take the forms of
\$
&M_{\zeta,1} = \Bigl[ \bigl(\|Q\|_\F + \|R\|_\F\bigr) +  \bigl( \|A\|_\F^2 + \|B\|_\F^2 \bigr) \cdot \sqrt{d}\cdot  J(K_0, b_0)\cdot\sigma_{\min}^{-1}(\Psi_\omega)  \Bigr]\\
&\qquad\qquad + \bigl( \|A\|_2 + \|B\|_2 \bigr) \cdot  J(K_0, b_0)^2\cdot \sigma_{\min}^{-1}(\Psi_\omega)\cdot \sigma_{\min}^{-1}(Q), \notag\\
& \qquad\qquad + \Bigl[ \bigl(\|Q\|_2 + \|R\|_2\bigr) +  \bigl( \|A\|_2 + \|B\|_2 \bigr)^2 \cdot  J(K_0, b_0)\cdot\sigma_{\min}^{-1}(\Psi_\omega)  \Bigr] \notag\\
&\qquad \qquad \cdot J(K_0, b_0)\cdot \bigl[ \sigma_{\min}^{-1}(Q) + \sigma_{\min}^{-1}(R) \bigr] \notag\\
& M_{\zeta,2}  = \bigl( \|A\|_2 + \|B\|_2 \bigr) \cdot  (\kappa_Q + \kappa_R), \qquad M_\xi  = C\cdot (M_{\zeta,1} + M_{\zeta,2}) \cdot J(K_0, b_0)^2\cdot\sigma_{\min}^{-2}(Q),
\$
where $C$ is a positive absolute constant, and $\kappa_Q$ and $\kappa_R$ are condition numbers of  $Q$ and $R$, respectively. 
\end{definition}

We summarize the primal-dual gradient temporal difference algorithm in Algorithm \ref{algo:pg_eval}. Hereafter, for notational convenience, we denote by $\hat\psi_t$ the estimated feature vector $\hat\psi(x_t, u_t)$.

\begin{algorithm}[htpb]
    \caption{Primal-Dual Gradient Temporal Difference Algorithm.}\label{algo:pg_eval}
    \begin{algorithmic}[1]
    \STATE{{\textbf{Input:}} Policy $\pi_{K,b}$, mean-field state $\mu$, numbers of iteration $\tilde T$ and $T$, stepsizes $\{\gamma_t\}_{t\in[T]}$, parameters $K_0$ and $b_0$. }
    \STATE{Define the sets $\cV_\zeta$ and $\cV_\xi$ via Definition \ref{assum:proj} with $K_0$ and $b_0$.  }
    \STATE{Initialize the parameters by $\zeta_0\in \cV_\zeta$ and $\xi_0\in \cV_\xi$. }
    \STATE{Sample $\tilde x_0$ from the the stationary distribution $\mathcal N(\mu_{K,b}, \Phi_K)$.}
    \FOR{ $t =  0, \ldots, \tilde T-1$}
        \STATE{Given the mean-field state $\mu$, take action $\tilde u_{t}$ following $\pi_{K,b}$ and generate the next state $\tilde x_{t+1}$. \label{line:tilde-xu}}
    \ENDFOR
    \STATE{Set $\hat\mu_{K,b} \gets 1/\tilde T\cdot \sum_{t = 1}^{\tilde T} \tilde x_t$ and compute the estimated feature vector $\hat\psi$ via \eqref{eq:est_def_feature}. }
    \STATE{Sample $x_0$ from the the stationary distribution $\mathcal N(\mu_{K,b}, \Phi_K)$.}
    \FOR{ $t =  0, \ldots, T-1$}
        \STATE{Given the mean-field state $\mu$, take action $u_{t}$ following $\pi_{K,b}$, observe the cost $c_{t}$, and generate the next state $x_{t+1}$.}
        \STATE{Set $\delta_{t+1} \gets \zeta_{t}^1  + (\hat \psi_{t} - \hat\psi_{t+1})^\top \zeta_{t}^2   - c_{t}.$}
        \STATE{Update parameters via
        \$
        & \zeta_{t+1}^1 \gets \zeta_{t}^1 - \gamma_{t+1}\cdot (\xi_{t}^1 + \hat \psi_{t}^\top \xi_{t}^2), \qquad &&\zeta_{t+1}^2 \gets \zeta_{t}^2 - \gamma_{t+1}\cdot \hat \psi_{t}(\hat\psi_{t} - \hat\psi_{t+1})^\top \xi_{t}^2, \notag  \\
        & \xi_{t+1}^1 \gets (1-\gamma_{t+1})\cdot \xi_{t}^1 + \gamma_{t+1} \cdot (\zeta_{t}^1 -  c_{t}), \qquad &&\xi_{t+1}^2 \gets (1-\gamma_{t+1})\cdot \xi_{t}^2 + \gamma_{t+1} \cdot\delta_{t+1}\cdot \hat\psi_{t}. 
        \$\vskip-10pt}
        \STATE{Project $\zeta_{t+1}$ and $\xi_{t+1}$ to  $\cV_\zeta$ and $\cV_\xi$, respectively. }
    \ENDFOR
    \STATE{Set $\hat\alpha_{K, b} \gets (\sum_{t=1}^T \gamma_t)^{-1}\cdot(\sum_{t=1}^T \gamma_t\cdot \zeta_t^2)$, and
    \$
    \hat\Upsilon_K \gets \smat(\hat\alpha_{K, b, 1}), \qquad\begin{pmatrix}\hat p_{K, b}\\ \hat q_{K, b}\end{pmatrix} \gets \hat\alpha_{K, b, 2} - \hat\Upsilon_{K}\begin{pmatrix}
    \hat\mu_{K,b}\\
    -K \hat\mu_{K,b} + b
    \end{pmatrix},
    \$\vskip-10pt
    where $\hat\alpha_{K, b, 1} =  (\hat\alpha_{K, b})_{1}^{(k+d+1)(k+d)/2}$ and $\hat\alpha_{K, b, 2} =  (\hat\alpha_{K, b})_{(k+d+1)(k+d)/2+1}^{(k+d+3)(k+d)/2}$. }
    \STATE{\textbf{Output: }  Estimators  $\hat\mu_{K,b}$,  $\hat\Upsilon_K$, and $\hat q_{K, b}$. }
    \end{algorithmic}
\end{algorithm}

We now characterize the rate of convergence of Algorithm \ref{algo:pg_eval}.

\begin{theorem}[Convergence of Algorithm \ref{algo:pg_eval}]\label{thm:pe}
Given $K_0$, $b_0$, $K$, and $b$ such that $\rho(A-BK_0) < 1$ and $J(K,b) \leq J(K_0, b_0)$, we define the sets $\cV_\zeta$ and $\cV_\xi$ through Definition \ref{assum:proj}. Let $\gamma_t = \gamma_0 t^{-1/2}$, where $\gamma_0$ is a positive absolute constant. Let $\rho\in (\rho(A-BK), 1) $. For  $\tilde T \geq \poly_0( \|K\|_\F, \|b\|_2, \|\mu\|_2, J(K_0, b_0) )\cdot (1 - \rho)^{-6}$ and a sufficiently large $T$, it holds with probability at least $1-T^{-4} - \tilde T^{-6}$ that
\$
\|\hat\alpha_{K, b} - \alpha_{K, b}\|_2^2 \leq  \lambda_K^{-2}\cdot \poly_1\bigl( \|K\|_\F, \|b\|_2, \|\mu\|_2, J(K_0, b_0) \bigr)  \cdot  \biggl[  \frac{\log^6 T}{T^{1/2} \cdot (1-\rho)^{4}} +  \frac{\log \tilde T}{\tilde T^{1/4}\cdot (1-\rho)^{2}}\biggr], 
\$
where $\lambda_K$ is defined in Proposition \ref{prop:invert_theta}.  Same bounds for $\|\hat\Upsilon_K - \Upsilon_K\|_\F^2$, $\|\hat p_{K,b} - p_{K,b}\|_2^2$, and $\|\hat q_{K,b} - q_{K,b}\|_2^2$ hold.  Meanwhile,   it  holds with probability at least $1 -  \tilde T^{-6}$ that
\$
\| \hat \mu_{K,b} - \mu_{K,b} \|_2 \leq   \frac{\log \tilde T}{\tilde T^{1/4}}\cdot (1-\rho)^{-2} \cdot \poly_2\bigl( \|\Phi_K\|_2, \|K\|_\F, \|b\|_2, \|\mu\|_2, J(K_0, b_0)  \bigr). 
\$
\end{theorem}
\begin{proof}
See \S\ref{proof:thm:pe} for a detailed proof. 
\end{proof}

\subsection{Temporal Difference Policy Evaluation Algorithm}\label{sec:td_algo}
Besides the primal-dual gradient temporal difference algorithm, we can also evaluate $\alpha_{K,b}$ by TD(0) method \citep{sutton2018reinforcement} in practice, which is presented in Algorithm \ref{algo:td_pe}.

\begin{algorithm}[ht]
    \caption{Temporal Difference Policy Evaluation Algorithm.}\label{algo:td_pe}
    \begin{algorithmic}[1]
    \STATE{{\textbf{Input:}} Policy $\pi_{K,b}$, number of iteration $\tilde T$ and $T$, stepsizes $\{\gamma_t\}_{t\in[T]}$. }
    \STATE{Sample $\tilde x_0$ from the stationary distribution $\mathcal N(\mu_{K,b}, \Phi_K)$.}
    \FOR{ $t =  0, \ldots, \tilde T-1$}
        \STATE{Take action $\tilde u_{t}$ under the policy $\pi_{K,b}$ and generate the next state $\tilde x_{t+1}$.}
    \ENDFOR
    \STATE{Set $\hat\mu_{K,b} \gets 1/\tilde T\cdot \sum_{t = 1}^{\tilde T} \tilde x_t$. }
    \STATE{Sample $x_0$ from the the stationary distribution $\mathcal N(\mu_{K,b}, \Phi_K)$.}
    \FOR{ $t =  0, \ldots, T$}
        \STATE{Given the mean-field state $\mu$, take action $ u_t$ following $\pi_{K,b}$, observe the cost $c_t$, and generate the next state $ x_{t+1}$.}
        \STATE{Set $\delta_{t+1} \gets \zeta_{t}^1  + (\hat \psi_{t} - \hat\psi_{t+1})^\top \zeta_{t}^2   - c_{t}$. }
        \STATE{Update parameters via $\zeta_{t+1}^1 \gets (1-\gamma_{t+1})\cdot \zeta_{t}^1 + \gamma_{t+1}\cdot c_{t}$ and $\zeta_{t+1}^2 \gets \zeta_{t}^2 - \gamma_{t+1}\cdot \delta_{t+1}\cdot \hat\psi_{t}$.}
        \STATE{Project $\zeta_t$ to $\cV'_\zeta$, where $\cV'_\zeta$ is a compact set. }
    \ENDFOR
    \STATE{Set  $\hat\alpha_{K, b} \gets (\sum_{t=1}^T \gamma_t)^{-1}\cdot(\sum_{t=1}^T \gamma_t\cdot \zeta_t^2)$, and 
    \$
    \hat\Upsilon_K \gets \smat(\hat\alpha_{K, b, 1}), \qquad\begin{pmatrix}\hat p_{K, b}\\ \hat q_{K, b}\end{pmatrix} \gets \hat\alpha_{K, b, 2} - \hat\Upsilon_{K}\begin{pmatrix}
    \hat\mu_{K,b}\\
    -K \hat\mu_{K,b} + b
    \end{pmatrix},
    \$\vskip-10pt
    where $\hat\alpha_{K, b, 1} =  (\hat\alpha_{K, b})_{1}^{(k+d+1)(k+d)/2}$ and $\hat\alpha_{K, b, 2} =  (\hat\alpha_{K, b})_{(k+d+1)(k+d)/2+1}^{(k+d+3)(k+d)/2}$. 
  }
    \STATE{\textbf{Output: }  Estimators  $\hat\mu_{K,b}$,  $\hat\Upsilon_K$, and $\hat q_{K, b}$. }

    \end{algorithmic}
\end{algorithm}

Note that in related literature \citep{bhandari2018finite, korda2015td}, non-asymptotic convergence analysis of TD(0) method with linear function approximation is only applied to discounted MDP. As for our ergodic setting, the convergence of TD(0) method is only shown asymptotically \citep{borkar2000ode, kushner2003stochastic} using ordinary differential equation method.  Therefore, in the convergence theorem proposed in \S\ref{sec:algo}, we only focus on the primal-dual gradient temporal difference method (Algorithm \ref{algo:pg_eval}) to establish non-asymptotic convergence result.


\section{General Formulation}\label{sec:gen_form}

In this section, we study a general formulation of LQ-MFG. Compared with Problem \ref{prob:mflqr}, such a general formulation includes an additional term $x_t^\top P \EE x_t^*$ in the cost function.  We define the general formulation as follows. 

\begin{problem}[General LQ-MFG]\label{prob:mflqr2}
We consider the following formulation,
\$
& x_{t+1} = A x_t + B u_t + \overline A  \EE x_t^* + d + \omega_t, \notag\\
& \tilde c(x_t, u_t) = x_t^\top Q x_t + u_t^\top R u_t + (\EE x_t^*)^\top \overline Q(\EE x_t^*) + 2 x_t^\top P (\EE x_t^*), \notag\\
& \tilde J(\pi) = \lim_{T\to\infty} \EE\Biggl[\frac{1}{T}\sum_{t = 0}^T \tilde c(x_t, u_t)\Biggr], 
\$
where $x_t\in\RR^m$ is the state vector, $u_t\in\RR^k$ is the action vector generated by the policy $\pi$, $\{x_t^*\}_{t\geq0}$ is the trajectory generated by a Nash policy $\pi^*$ (assuming it exists), $\omega_t\in\RR^m$ is an independent random noise term following the Gaussian distribution $\mathcal N(0, \Psi_\omega)$, and $d\in\RR^m$ is a drift term. Here the expectation in $\EE x^*_t$ is taken across all the agents.  We aim to find $\pi^*$ such that $\tilde J(\pi^*) = \inf_{\pi\in\Pi}\tilde J(\pi)$. 
\end{problem}

Following similar analysis in \S\ref{sec:algo_lqmfg}, it suffices to study Problem \ref{prob:mflqr2} with $t$ sufficiently large, which motivates us to formulate the following general drifted LQR (general D-LQR) problem. 

\begin{problem}[General D-LQR]\label{prob:slqr2}
Given a mean-field state $\mu\in\RR^m$, we consider the following formulation,
\$
& x_{t+1} = A x_t + B u_t + \overline A \mu + d + \omega_t,\notag\\
& \tilde c_\mu(x_t, u_t) = x_t^\top Q x_t + u_t^\top R u_t + \mu^\top \overline Q\mu + 2 x_t^\top P \mu, \notag\\
& \tilde J_\mu(\pi) = \lim_{T\to\infty} \EE\Biggl[\frac{1}{T}\sum_{t = 0}^T \tilde c_\mu(x_t, u_t)\Biggr],  
\$
where $x_t\in\RR^m$ is the state vector, $u_t\in\RR^k$ is the action vector generated by the policy $\pi$, $\omega_t\in\RR^m$ is an independent random noise term following the Gaussian distribution $\mathcal N(0, \Psi_\omega)$, and $d\in\RR^m$ is a drift term. We aim to find an optimal policy $\pi^*_\mu$ such that $\tilde J_\mu(\pi^*_\mu) = \inf_{\pi\in\Pi}\tilde J_\mu(\pi)$. 
\end{problem}

In Problem \ref{prob:slqr2}, the unique optimal policy $\pi^*_\mu$ admits a linear form $\pi^*_\mu(x_t) = -K_{\pi^*_\mu}x_t + b_{\pi^*_\mu}$ \citep{anderson2007optimal}, where the matrix $K_{\pi^*_\mu}\in\RR^{k\times m}$ and the vector $b_{\pi^*_\mu}\in\RR^{k}$ are the parameters of the policy $\pi$.  
 It then suffices to find the optimal policy in the class $\Pi$ introduced in \eqref{eq:def-Pi}.  
Similar to \S\ref{sec:slqr}, we drop the subscript $\mu$ when we focus on Problem \ref{prob:slqr2} for a fixed $\mu$. 
We write $\pi_{K,b}(x) = -Kx + b + \sigma\eta$ to emphasize the dependence on $K$ and $b$, and $\tilde J(K,b) = \tilde J(\pi_{K,b})$ consequently.
We derive a close form of the expected total cost $\tilde J(K,b)$ in the following proposition. 

\begin{proposition}
The expected total cost $\tilde J(K,b)$ in Problem \ref{prob:slqr2} is decomposed as
\$
\tilde J(K,b) = \tilde J_1(K) + \tilde J_2(K,b) + \sigma^2 \cdot \tr(R) + \mu^\top \overline Q \mu,
\$
where $\tilde J_1(K)$ and $\tilde J_2(K,b)$ take the forms of
\$
&\tilde  J_1(K) =\tr\bigl[(Q + K^\top R K)\Phi_K\bigr]  = \tr(P_K \Psi_\epsilon),\notag\\ 
& \tilde J_2(K,b) = \begin{pmatrix}
\mu_{K,b}\\
b
\end{pmatrix}^\top
\begin{pmatrix}
Q + K^\top RK & -K^\top R\\
-R K & R
\end{pmatrix}
 \begin{pmatrix}
\mu_{K,b}\\
b
\end{pmatrix} + 2\mu^\top P\mu_{K,b}.
\$
Here $\mu_{K,b}$ is defined in \eqref{eq:f2q}, $\Phi_K$ is defined in \eqref{eq:f2p}, and $P_K$ is defined in \eqref{eq:bellman}.
\end{proposition}
\begin{proof}
The proof is similar to the one of Proposition \ref{prop:cost_form}. Thus we omit it here. 
\end{proof}

Compared with the form of $J(K,b)$ in \eqref{eq:a1}, the form of $\tilde J(K,b)$ contains an additional term $2\mu^\top P\mu_{K,b}$ in $\tilde J_2(K,b)$.  Recall that $\mu_{K,b}$ is linear in $b$ by \eqref{eq:f2q}. Therefore, $2\mu^\top P\mu_{K,b}$ is linear in $b$, which shows that $\tilde J_2(K,b)$ is still strongly convex in $b$.  The following proposition formally characterize the strong convexity of $\tilde J_2(K,b)$ in $b$. 

\begin{proposition}
Given any $K$, the function $\tilde J_2(K,b)$ is $\nu_K$-strongly convex in $b$, here $\nu_K = \sigma_{\min}(Y_{1, K}^\top Y_{1, K} + Y_{2, K}^\top Y_{2, K})$, where $Y_{1, K} = {R^{1/2}}K(I-A+BK)^{-1}B - {R^{1/2}}$ and $Y_{2, K} = {Q^{1/2}} (I-A+BK)^{-1} B$. Also, $\tilde J_2(K, b)$ has $\iota_K$-Lipschitz continuous gradient in $b$, where $\iota_K \leq [1-\rho(A-BK)]^{-2} \cdot (  \|B\|_2^2\cdot \|K\|_2^2\cdot \|R\|_2 + \|B\|_2^2\cdot \|Q\|_2 )$. 
\end{proposition}
\begin{proof}
The proof is similar to the one of Proposition \ref{prop:convex_J2}. Thus we omit it here. 
\end{proof}

We derive a similar proposition to Proposition \ref{prop:J2} in the sequel.  

\begin{proposition}
We define $\tilde b^K = \argmin_{b} \tilde J_2(K,b)$. It holds that
\$
\tilde b^K & =  \bigl[ KQ^{-1} (I-A)^\top - R^{-1}B^\top \bigr]  \cdot S \cdot \bigl[ (\overline A\mu + d) + (I-A)Q^{-1}P^\top \mu\bigr] - K Q^{-1} P^\top \mu. 
\$
Moreover, $\tilde J_2(K, b^K)$ takes the form of
\$
\tilde J_2(K, \tilde b^K) = \begin{pmatrix}
\overline A\mu + d\\
P^\top \mu
\end{pmatrix}^\top  \begin{pmatrix}
S  & S (I-A)Q^{-1}\\
Q^{-1}(I-A)^\top S & 3 Q^{-1}(I-A)^\top S(I-A)Q^{-1} - Q^{-1}
\end{pmatrix} \begin{pmatrix}
\overline A\mu + d\\
P^\top \mu
\end{pmatrix},
\$
which is independent of $K$. Here $S = [ (I-A)Q^{-1}(I-A)^\top + BR^{-1}B^\top ]^{-1}$.  
\end{proposition}
\begin{proof}
The proof is similar to the one of Proposition \ref{prop:J2}. Thus we omit it here. 
\end{proof}

Similar to Problem \ref{prob:slqr}, we define the state- and action-value functions as
\$
& \tilde V_{K,b}(x) = \sum_{t = 0}^\infty \Bigl\{ \EE\bigl[ \tilde c(x_t, u_t) \given x_0 = x \bigr] - \tilde J(K,b) \Bigr\},\\
& \tilde Q_{K,b}(x,u) = \tilde c(x,u) - \tilde J(K,b) + \EE \bigl[ \tilde V_{K,b}(x_1)\given x_0 = x, u_0 = u \bigr],
\$
where $x_t$ follows the state transition, and $u_t$ follows the policy $\pi_{K,b}$ given $x_t$. In other words, we have $u_{t} = -Kx_t + b + \sigma\eta_t$, where $\eta_t\sim \mathcal N(0, I )$.   
Similar to Proposition \ref{prop:val_func_form}, the following proposition establishes the close forms of these value functions.

\begin{proposition}
The state-value function $\tilde V_{K,b}(x)$ takes the form of
\$
\tilde V_{K,b}(x) = x^\top P_K x - \tr(P_K \Phi_K)  +  2 \tilde f_{K,b}^\top (x -\mu_{K,b}) - \mu_{K,b}^\top P_K\mu_{K,b},
\$
and the action-value function $\tilde Q_{K,b}(x,u)$ takes the form of
\$
\tilde Q_{K,b}(x,u) = &\begin{pmatrix}
x\\
u
\end{pmatrix}^\top
\Upsilon_K
\begin{pmatrix}
x\\
u
\end{pmatrix} + 2\begin{pmatrix}
\tilde p_{K,b}\\
\tilde q_{K,b}
\end{pmatrix}^\top \begin{pmatrix}
x\\
u
\end{pmatrix}   - \tr(P_K\Phi_K)  - \sigma^2\cdot \tr(R + P_K B B^\top) - b^\top Rb \notag\\
& +2b^\top RK\mu_{K,b}  - \mu_{K,b}^\top (Q+K^\top RK + P_K)\mu_{K,b}    + 2\tilde f_{K,b}^\top\bigl[ (\overline A\mu+d)-\mu_{K,b}\bigr]  \notag\\
& + (\overline A\mu+d)^\top P_K (\overline A\mu+d) - 2\mu^\top P\mu_{K,b}.
\$
Here $\Upsilon_K$ is defined in \eqref{eq:def_upsilon}, and $\tilde p_{K,b}$, $\tilde q_{K,b}$ are defined as
\#\label{eq:def_upsilon2}
\tilde p_{K,b} = A^\top\bigl[P_K\cdot (\overline A \mu+d) +\tilde f_{K,b} \bigr] + P\mu, \qquad \tilde q_{K,b} = B^\top\bigl[P_K\cdot (\overline A \mu+d) +\tilde f_{K,b} \bigr],
\#
where $\tilde f_{K,b} = (I-A+BK)^{-\top} [  (A-BK)^\top P_K (Bb+\overline A\mu + d) - K^\top Rb + P\mu]$. 
\end{proposition}
\begin{proof}
The proof is similar to the one of Proposition \ref{prop:val_func_form}. Thus we omit it here. 
\end{proof}

The following proposition establishes the gradients of $\tilde J_1(K)$ and $\tilde J_2(K, b)$, respectively.

\begin{proposition}
The gradient of $\tilde J_1(K)$ and the gradient of $\tilde J_2(K,b)$ with respect to $b$ take the forms of
\$
&\nabla_K\tilde  J_1(K) = 2(\Upsilon_K^{22} K  - \Upsilon_K^{21})\cdot \Phi_K,\qquad \nabla_b \tilde J_2(K,b) = 2\bigl[\Upsilon_K^{22}(-K\mu_{K,b} + b) + \Upsilon_{K}^{21}\mu_{K,b} +\tilde  q_{K,b}\bigr], 
\$
where $\Upsilon_K$ and $\tilde q_{K,b}$ are defined in \eqref{eq:def_upsilon} and \eqref{eq:def_upsilon2}, respectively. 
\end{proposition}
\begin{proof}
The proof is similar to the one of Proposition \ref{prop:pg}. Thus we omit it here. 
\end{proof}

Equipped with above results,  parallel to the analysis in \S\ref{sec:algo},  it is clear that by slight modification of Algorithms \ref{algo:mflqr}, \ref{algo:ac_lqr}, and \ref{algo:pg_eval}, we derive similar actor-critic algorithms to solve both Problem \ref{prob:mflqr2} and Problem \ref{prob:slqr2}, where all the non-asymptotic convergence results hold.  We omit the algorithms and the convergence results here.


\section{Proofs of Theorems}

\subsection{Proof of Theorem \ref{thm:conv_mfg}}\label{proof:thm:conv_mfg}
We define $\mu_{s+1}^* = \Lambda(\mu_s)$, which is the mean-field state generated by the optimal policy $\pi_{K^*(\mu_s), b^*(\mu_s)} = \Lambda_1(\mu_s)$ under the current mean-field state $\mu_s$. By Proposition \ref{prop:J2}, the optimal $K^*(\mu)$ is independent of the mean-field state $\mu$. Therefore, we write $K^* = K^*(\mu)$ hereafter for notational convenience. By \eqref{eq:f2q}, we know that
\$
\mu^*_{s + 1} = (I-A+BK^*)^{-1}\cdot \bigl[B b^*(\mu_s) + \overline A \mu_s + d\bigr].
\$
We define
\$
\tilde \mu_{s+1} = (I-A + BK_s)^{-1} (Bb_s + \overline A\mu_s + d),
\$
which is the mean-field state generated by the policy $\pi_s$ under the current mean-field state $\mu_s$, where $K_s$ and $b_s$ are the parameters of the policy $\pi_s$.
By triangle inequality, we have
\#\label{eq:3terms}
\|\mu_{s+1} - \mu^*\|_2& \leq \underbrace{\| \mu_{s + 1} - \tilde \mu_{s+1}\|_2}_{E_1}+ \underbrace{\| \tilde \mu_{s + 1} -  \mu^*_{s+1}\|_2}_{E_2} + \underbrace{\| \mu_{s+1}^* - \mu^*  \|_2}_{E_3},
\#
where $\mu_{s+1}$ is generated by Algorithm \ref{algo:mflqr}. 
We upper bound  $E_1$, $E_2$, and $E_3$ in the sequel. 

\vskip5pt
\noindent\textbf{Upper Bound of $E_1$.}  By  Theorem \ref{thm:ac}, it holds with probability at least $1 - \varepsilon^{10}$ that
\#\label{eq:term1bound}
E_1 = \|\mu_{s+1} - \tilde \mu_{s+1}\|_2  < \varepsilon_s  \leq \varepsilon / 8 \cdot 2^{-s},
\#
where $\varepsilon_s$ is given in \eqref{eq:def_eps_s}.  

\vskip5pt
\noindent\textbf{Upper Bound of $E_2$.}  By the triangle inequality, we have
\#\label{eq:term2bound11}
E_2 & =  \Bigl\|  (I-A + BK_s)^{-1} (Bb_s + \overline A\mu_s + d) -  (I-A+BK^*)^{-1}\cdot \bigl[B b^*(\mu_s) + \overline A \mu_s + d\bigr]  \Bigr\|_2\notag\\
& \leq \bigl\|B b^*(\mu_s)  + \overline A\mu_s  + d \bigr \|_2\cdot \Bigl\|  \bigl[ I-A+BK^* + B(K_s - K^*) \bigr]^{-1} - (I-A+BK^*)^{-1} \Bigr\|_2\notag\\
&\qquad + \bigl\|  (I-A+BK_s)^{-1}  \bigr\|_2\cdot \|B\|_2 \cdot \bigl\|b_s - b^*(\mu_s) \bigr\|_2. 
\#
By Taylor's expansion,  we have
\#\label{eq:term2bound12}
& \Bigl\|  \bigl[ I-A+BK^* + B(K_s - K^*) \bigr]^{-1} - (I-A+BK^*)^{-1} \Bigr\|_2 \notag\\
& \qquad =  \Bigl\| (I-A+BK^*)^{-1} \bigl[ I + (I-A+BK^*)^{-1}B(K_s - K^*) \bigr]^{-1} - (I-A+BK^*)^{-1} \Bigr\|_2 \notag\\
& \qquad \leq 2\bigl \|  (I-A+BK^*)^{-1} B (K_s - K^*) (I-A+BK^*)^{-1} \bigr\|_2. 
\#
Meanwhile, by Taylor's expansion, it holds with probability at least $1 - \varepsilon^{10}$ that
\#\label{eq:term2bound13}
 \bigl\|  (I-A+BK_s)^{-1}  \bigr\|_2 & =  \Bigl\|  \bigl(I-A+BK^* + B(K_s - K^*) \bigr)^{-1}  \Bigr\|_2 \notag \\
& =  \Bigl\| (I-A+BK^*)^{-1} \bigl(I + (I-A+BK^*)^{-1} B(K_s - K^*) \bigr)^{-1}  \Bigr\|_2 \notag \\
& \leq \bigl[1 - \rho(A-BK^*) \bigr]^{-1} \cdot \Bigl(1 + \bigl \|(I-A+BK^*)^{-1} B \bigr\|_2 \cdot \|K^* - K_s\|_2 \Bigr)\notag\\
& \leq 2 \bigl[1 - \rho(A-BK^*) \bigr]^{-2},
\#
where the last inequality comes from Theorem \ref{thm:ac}. By plugging \eqref{eq:term2bound12} and \eqref{eq:term2bound13} in \eqref{eq:term2bound11}, it holds with probability at least $1 - \varepsilon^{10}$ that
\#\label{eq:term2bound1}
E_2& \leq  2 \bigl \|B b^*(\mu_s)  + \overline A\mu_s + d \bigr \|_2\cdot \bigl \|  (I-A+BK^*)^{-1} B (K_s - K^*) (I-A+BK^*)^{-1} \bigr \|_2\notag\\
&\qquad + \bigl \|  (I-A+BK_s)^{-1} \bigr \|_2\cdot \|B\|_2 \cdot \bigl \|b_s - b^*(\mu_s) \bigr \|_2\notag\\
& \leq 2 \bigl \|B b^*(\mu_s)  + \overline A\mu_s + d \bigr \|_2\cdot \bigl[ 1 - \rho(A-B K^*) \bigr]^{-2}\cdot \|B\|_2 \cdot \|K_s - K^*\|_2 \\
&\qquad + 2\bigl[ 1 - \rho(A-BK^*) \bigr]^{-2} \cdot \|B\|_2 \cdot \bigl \|b_s - b^*(\mu_s) \bigr \|_2. \notag
\#
By Proposition \ref{prop:J2}, it holds that 
\#\label{eq:term2bound2}
\bigl \|B b^*(\mu_s)  + \overline A\mu_s + d \bigr \|_2& \leq L_1\cdot\|B\|_2\cdot \|\mu_s\|_2 + \|\overline A\|_2\cdot \|\mu_s\|_2 + \|d\|_2\notag\\
& \leq  \bigl(L_1\cdot \|B\|_2 + \|\overline A\|_2\bigr) \cdot \|\mu_s\|_2+ \|d\|_2,
\#
where the scalar $L_1$ is defined in Assumption \ref{assum:contraction}.  Meanwhile, by Theorem \ref{thm:ac}, 
it holds with probability at least $1-\varepsilon^{10}$ that
\#\label{eq:term2bound3}
\|K_s - K^*\|_\F\leq \bigl[{\sigma_{\min}^{-1}(\Psi_\epsilon)\cdot \sigma_{\min}^{-1}(R) \cdot \varepsilon_s}\bigr]^{1/2}, \qquad \bigl\|b_s - b^*(\mu_s)\bigr\|_2\leq M_b(\mu_s) \cdot \varepsilon_s^{1/2}, 
\#
where $M_b(\mu_s)$ is defined in \eqref{eq:def-mb}.  Combining \eqref{eq:term2bound1}, \eqref{eq:term2bound2}, \eqref{eq:term2bound3}, and the choice of $\varepsilon_s$ in \eqref{eq:def_eps_s}, it holds with probability at least $1-\varepsilon^{10}$ that
\#\label{eq:term2bound}
E_2 \leq \varepsilon/8 \cdot 2^{-s}. 
\# 

\vskip5pt
\noindent\textbf{Upper Bound of $E_3$.}  By Proposition \ref{prop:uniq_eq}, we have
\#\label{eq:term3bound}
E_3 = \| \mu_{s+1}^* - \mu^*  \|_2 = \bigl \| \Lambda(\mu_{s}) - \Lambda(\mu^*) \bigr \|_2\leq L_0\cdot  \| \mu_{s} - \mu^*  \|_2, 
\# 
where $L_0 = L_1 L_3 + L_2$ by Assumption \ref{assum:contraction}. 

\vskip5pt
By plugging \eqref{eq:term1bound}, \eqref{eq:term2bound}, and \eqref{eq:term3bound} in \eqref{eq:3terms}, we know that 
\#\label{eq:iter1}
\|\mu_{s+1}-\mu^*\|_2 \leq   L_0\cdot \|\mu_{s}-\mu^*\|_2 + \varepsilon\cdot 2^{-s-2},
\#
which holds with probability at least $1-\varepsilon^{10}$.  Following from \eqref{eq:iter1} and a union bound argument with $S = \cO(\log (1/\varepsilon))$, it holds with probability at least $1-\varepsilon^{5}$ that
\$
\|\mu_S - \mu^*\|_2 \leq L_0^S\cdot \|\mu_0 - \mu^*\|_2 + \varepsilon/2,
\$
where we use the fact that $L_0 < 1$ by Assumption \ref{assum:contraction}.  By the choice of $S$ in \eqref{eq:choice_S}, it further holds with probability at least $1-\varepsilon^{6}$ that
\#\label{eq:xiewanle}
\|\mu_S - \mu^*\| \leq \varepsilon. 
\#

By Theorem \ref{thm:ac} and the choice of $\varepsilon_s$ in \eqref{eq:def_eps_s}, it holds with probability at least $1 - \varepsilon^5$ that
\#\label{eq:xiewanle2}
\| K_S - K^*  \|_\F = \bigl\| K_S - K^*(\mu_S)  \bigr\|_\F \leq  \bigl[{\sigma_{\min}^{-1}(\Psi_\epsilon)\cdot \sigma_{\min}^{-1}(R) \cdot \varepsilon_S} \bigr]^{1/2} \leq \varepsilon. 
\#
Meanwhile, by the triangle inequality and the choice of $\varepsilon_s$ in \eqref{eq:def_eps_s}, it holds with probability at least $1 - \varepsilon^5$ that
\#\label{eq:xiewanle3}
\|b_S - b^*\|_2 & \leq \bigl\|b_S - b^*(\mu_S) \bigr\|_2 + \bigl\|b^*(\mu_S) - b^* \bigr\|_2 \notag\\
& \leq M_b(\mu_S) \cdot \varepsilon_S^{1/2} + L_1 \cdot \| \mu_S - \mu^*  \|_2\notag\\
& \leq (1 + L_1)\cdot \varepsilon,
\#
where the second inequality comes from Theorem \ref{thm:ac} and Proposition \ref{prop:J2}, and the last inequality comes from \eqref{eq:xiewanle}.  By \eqref{eq:xiewanle}, \eqref{eq:xiewanle2}, and \eqref{eq:xiewanle3}, we conclude the proof of the theorem.

\subsection{Proof of Theorem \ref{thm:ac}}\label{proof:thm:ac}
\begin{proof} 
We first show that $J_1(K_N) - J_1(K^*) < \varepsilon /2$ with a high probability, then show that $J_2(K_N, b_H) - J_2(K^*, b^*) < \varepsilon / 2$ with a high probability.   Then we have
\$
J(K_N, b_N) - J(K^*, b^*)& = J_1(K_N) + J_2(K_N, b_H)  -  J_1(K^*) - J_2(K^*, b^*) < \varepsilon
\$
with a high probability, which proves Theorem \ref{thm:ac}. 

\vskip5pt
\noindent\textbf{Part 1.} We show that $J_1(K_N) - J_1(K^*) < \varepsilon /2$ with a high probability. 

We first bound $J_1(K_1) - J_1(K_2)$ for any $K_1$ and $K_2$.   By Proposition \ref{prop:cost_form}, $J_1(K)$ takes the form of
\#\label{eq:J1_form}
J_1(K) = \tr(P_K \Psi_\epsilon) = \EE_{y\sim \mathcal N(0, \Psi_\epsilon)}(y^\top P_K y).
\#
The following lemma calculates $y^\top P_{K_1} y - y^\top P_{K_2} y$ for any $K_1$ and $K_2$. 

\begin{lemma}\label{lemma:cost_diff}
Assume that $\rho(A-BK_1) < 1$ and $\rho(A-BK_2) < 1$. For any state vector $y$, we denote by $\{y_t\}_{t\geq 0}$ the sequence generated by the state transition $y_{t+1} = (A-BK_2)y_t$ with initial state $y_0 = y$. It holds that
\$
y^\top P_{K_2} y - y^\top P_{K_1} y = \sum_{t\geq 0} D_{K_1, K_2}(y_t),
\$
where 
\$
D_{K_1, K_2}(y) = 2y^\top (K_2 - K_1)(\Upsilon_{K_1}^{22} K_1  - \Upsilon_{K_1}^{21}) y + y^\top (K_2 - K_1)^\top\Upsilon_{K_1}^{22} (K_2 - K_1)y.
\$
 Here  $\Upsilon_K$ is defined in \eqref{eq:def_upsilon}. 
\end{lemma}
\begin{proof}
See \S\ref{proof:lemma:cost_diff} for a detailed proof. 
\end{proof}

The following lemma shows that $J_1(K)$ is gradient dominant. 

\begin{lemma}\label{lemma:grad_dom}
Let $K^*$ be the optimal parameter and $K$ be a parameter such that $J_1(K) < \infty$, then it holds that
\#
& J_1(K) - J_1(K^*) \leq \sigma_{\min}^{-1}(R)\cdot \| \Phi_{K^*} \|_2\cdot \tr\bigl[ (\Upsilon_K^{22} K  - \Upsilon_K^{21})^\top (\Upsilon_K^{22} K  - \Upsilon_K^{21}) \bigr], \label{eq:upper-bdJ}\\
& J_1(K) - J_1(K^*) \geq \sigma_{\min}(\Psi_\omega)\cdot \| \Upsilon_{K}^{22} \|_2^{-1}\cdot \tr\bigl[ (\Upsilon_K^{22} K  - \Upsilon_K^{21})^\top (\Upsilon_K^{22} K  - \Upsilon_K^{21}) \bigr]. \label{eq:lower-bdJ}
\#
\end{lemma}
\begin{proof}
See \S\ref{proof:lemma:grad_dom} for a detailed proof. 
\end{proof}    

Recall that from Algorithm \ref{algo:ac_lqr}, the parameter $K$ is updated via
\#\label{eq:ac_update_Kn}
K_{n + 1} = K_n - \gamma \cdot (\hat \Upsilon_{K_n}^{22}K_n - \hat \Upsilon_{K_n}^{21}),
\#
where $\hat \Upsilon_{K_n}$ is the output of Algorithm \ref{algo:pg_eval}.
We upper bound $|J_1(K_{n+1}) - J_1(K^*)|$ in the sequel.  First, we show that if $J_1(K_n) - J_1(K^*)\geq \varepsilon / 2$ holds for any $n\leq N$, we obtain that 
\#\label{eq:lxnb1}
J_1(K_N)\leq J_1(K_{N-1})\leq  \cdots \leq J_1(K_0),
\#
which holds with probability at least $1 - \varepsilon^{13}$.  We prove \eqref{eq:lxnb1} by mathematical induction.  Suppose that 
\#\label{eq:lxnb2}
J_1(K_n)\leq J_1(K_{n-1})\leq  \cdots \leq J_1(K_0),
\#
which holds for $n=0$.    In what follows, we define $\tilde K_{n+1}$ as
\#\label{eq:tilde_update_Kn}
\tilde K_{n+1} = K_n - \gamma \cdot ( \Upsilon_{K_n}^{22}K_n - \Upsilon_{K_n}^{21}),
\#
where $\Upsilon_{K_n}$ is given in \eqref{eq:def_upsilon}. By \eqref{eq:tilde_update_Kn},  we have
\#\label{eq:diff_tildeK_K1}
& J_1(\tilde K_{n+1}) - J_1(K_n) = \EE_{y\sim\mathcal N(0, \Psi_\epsilon)}\bigl[ y^\top ( P_{\tilde K_{n+1}} - P_{K_n} )y \bigr]\notag\\
&\qquad = -2\gamma\cdot \tr\bigl[ \Phi_{\tilde K_{n+1}} \cdot (\Upsilon_{K_n}^{22}K_n - \Upsilon_{K_n}^{21})^\top (\Upsilon_{K_n}^{22}K_n - \Upsilon_{K_n}^{21}) \bigr] \notag\\
&\qquad\qquad+ \gamma^2\cdot \tr\bigl[ \Phi_{\tilde K_{n+1}} \cdot (\Upsilon_{K_n}^{22}K_n - \Upsilon_{K_n}^{21})^\top \Upsilon_{K_n}^{22}(\Upsilon_{K_n}^{22}K_n - \Upsilon_{K_n}^{21}) \bigr]\notag\\
&\qquad \leq -2\gamma\cdot \tr\bigl[ \Phi_{\tilde K_{n+1}} \cdot (\Upsilon_{K_n}^{22}K_n - \Upsilon_{K_n}^{21})^\top (\Upsilon_{K_n}^{22}K_n - \Upsilon_{K_n}^{21}) \bigr]\\
& \qquad\qquad + \gamma^2\cdot \|\Upsilon_{K_n}^{22}\|_2 \cdot \tr\bigl[ \Phi_{\tilde K_{n+1}} \cdot (\Upsilon_{K_n}^{22}K_n - \Upsilon_{K_n}^{21})^\top (\Upsilon_{K_n}^{22}K_n - \Upsilon_{K_n}^{21}) \bigr],\notag
\#
where the first equality comes from \eqref{eq:J1_form}, the second equality comes from Lemma \ref{lemma:cost_diff}, and the last inequality comes from the trace inequality. By the definition of $\Upsilon_K$ in \eqref{eq:def_upsilon}, we obtain that
\#\label{eq:nmd1}
 \|\Upsilon_{K_n}^{22}\|_2 & \leq \|R\|_2 + \|B\|_2^2\cdot \|P_{K_n}\|_2 \leq \|R\|_2 + \|B\|_2^2\cdot J_1(K_n)\cdot \sigma_{\min}^{-1}(\Psi_\epsilon) \notag \\
& \leq \|R\|_2 + \|B\|_2^2\cdot J_1(K_0)\cdot \sigma_{\min}^{-1}(\Psi_\epsilon),
\#
where the second inequality comes from Proposition \ref{prop:cost_form}.  By plugging \eqref{eq:nmd1} and the choice of stepsize $\gamma\leq [\|R\|_2 + \|B\|_2^2\cdot J_1(K_0)\cdot \sigma_{\min}^{-1}(\Psi_\epsilon)]^{-1}$ into \eqref{eq:diff_tildeK_K1},  we obtain that
\#\label{eq:main_diff_J1}
J_1(\tilde K_{n+1}) - J_1(K_n)& \leq -\gamma\cdot \tr\bigl[ \Phi_{\tilde K_{n+1}} \cdot (\Upsilon_{K_n}^{22}K_n - \Upsilon_{K_n}^{21})^\top (\Upsilon_{K_n}^{22}K_n - \Upsilon_{K_n}^{21}) \bigr]\notag\\
&\leq -\gamma\cdot \sigma_{\min}(\Psi_\epsilon)\cdot  \tr\bigl[ (\Upsilon_{K_n}^{22}K_n - \Upsilon_{K_n}^{21})^\top (\Upsilon_{K_n}^{22}K_n - \Upsilon_{K_n}^{21}) \bigr]\notag\\
& \leq -\gamma\cdot \sigma_{\min}(\Psi_\epsilon)\cdot \sigma_{\min}(R)\cdot \|\Phi_{K^*}\|_2^{-1}\cdot  \bigl[ J_1(K_n) - J_1(K^*) \bigr] < 0,
\#
where the last inequality comes from Lemma \ref{lemma:grad_dom}.

The following lemma upper bounds $|J_1(\tilde K_{n+1}) - J_1( K_{n+1})|$. 

\begin{lemma}\label{lemma:error_bound_tilde_J}
Assume that $J_1(K_n)\leq J_1(K_0)$.  It holds with probability at least $1-\varepsilon^{15}$ that
\$
\bigl|J_1(\tilde  K_{n+1}) - J_1( K_{n+1})\bigr| \leq \gamma\cdot \sigma_{\min}(\Psi_\epsilon)\cdot \sigma_{\min}(R)\cdot \|\Phi_{K^*}\|_2^{-1}\cdot  \varepsilon / 4, 
\$
where $K_{n+1}$ and $\tilde K_{n+1}$ are defined in \eqref{eq:ac_update_Kn}  and  \eqref{eq:tilde_update_Kn}, respectively. 
\end{lemma}
\begin{proof}
See \S\ref{proof:lemma:error_bound_tilde_J} for a detailed proof. 
\end{proof}

Combining \eqref{eq:main_diff_J1} and Lemma \ref{lemma:error_bound_tilde_J}, if $J_1(K_n) - J_1(K^*) \geq \varepsilon/2$, it holds with probability at least $1-\varepsilon^{15}$ that
\#\label{eq:main_diff_J3}
J_1(K_{n+1}) - J_1(K_n)&  \leq J_1(\tilde K_{n+1}) - J_1(K_n) + \bigl|J_1(\tilde  K_{n+1}) - J_1( K_{n+1})\bigr|\notag\\
&  \leq  - \gamma\cdot \sigma_{\min}(\Psi_\epsilon)\cdot \sigma_{\min}(R)\cdot \|\Phi_{K^*}\|_2^{-1}\cdot  \varepsilon / 4 < 0. 
\#
Combining \eqref{eq:lxnb2} and \eqref{eq:main_diff_J3}, it holds  with probability at least $1 - \varepsilon^{15}$ that 
\$
J_1(K_{n+1})\leq J_1(K_{n})\leq  \cdots \leq J_1(K_0). 
\$  
Finally, following from a union bound argument and the choice of $N$ in Theorem \ref{thm:ac}, if $J_1(K_n) - J_1(K^*)\geq \varepsilon/2$ holds for any $n \leq N$, we have 
\$
J_1(K_{N})\leq J_1(K_{N-1})\leq  \cdots \leq J_1(K_0),
\$ 
which  holds with probability at least $1 - \varepsilon^{13}$.  Thus, we complete the proof of \eqref{eq:lxnb1}.

Combining \eqref{eq:main_diff_J1} and \eqref{eq:main_diff_J3},  for $J_1(K_n) - J_1(K^*)\geq \varepsilon/2$, we have 
\$
J_1(K_{n+1}) - J_1(K^*) \leq \bigl[1-\gamma\cdot \sigma_{\min}(\Psi_\epsilon)\cdot \sigma_{\min}(R)\cdot \|\Phi_{K^*}\|_2^{-1}\bigr]\cdot  \bigl[ J_1(K_n) - J_1(K^*) \bigr],
\$ 
which holds with probability at least $1-\varepsilon^{13}$. Meanwhile, following from a union bound argument and the choice of $N$ in Theorem \ref{thm:ac},   it holds with probability at least $1 - \varepsilon^{11}$ that 
\#\label{eq:lxnb3}
J_1(K_{N}) - J_1(K^*) \leq \varepsilon / 2.
\#  

The following lemma upper bounds $\|K_{N} - K^*\|_\F$. 

\begin{lemma}\label{lemma:local_sc}
For any  $K$, we have
\$
\|K - K^*\|_\F^2 \leq \sigma_{\min}^{-1}(\Psi_\epsilon)\cdot \sigma_{\min}^{-1}(R) \cdot \bigl[ J_1(K) - J_1(K^*) \bigr]. 
\$
\end{lemma}
\begin{proof}
See \S\ref{proof:lemma:local_sc} for a detailed proof. 
\end{proof}

Combining \eqref{eq:lxnb3} and Lemma \ref{lemma:local_sc}, we have
\#\label{eq:k-bound}
\| K_N - K^* \|_\F \leq \bigl[\sigma_{\min}^{-1}(\Psi_\epsilon)\cdot \sigma_{\min}^{-1}(R) \cdot \varepsilon / 2\bigr]^{1/2},
\# 
which holds with probability $1 - \varepsilon^{11}$.

\vskip5pt
\noindent\textbf{Part 2.}  We show that $J_2(K_N, b_H) - J_2(K^*, b^*) < \varepsilon / 2$ with high probability. 
Following from Proposition \ref{prop:J2}, it holds that $J_2(K^*, b^*) = J_2(K_N, b^{K_N})$. Therefore, it suffices to show that $J_2(K_N, b_H) - J_2(K_N, b^{K_N}) < \varepsilon / 2$.   

First, we show that if $J_2(K_N, b_h) - J_2(K_N, b^{K_N})\geq \varepsilon/2$ for any $h\leq H$, we obtain that
\#\label{eq:lxnb55}
J_2(K_N, b_{H}) \leq J_2(K_N, b_{H-1}) \leq \cdots\leq J_2(K_N, b_{1})\leq J_2(K_N, b_{0}),
\# 
which holds with probability at least $1 - \varepsilon^{13}$. We prove \eqref{eq:lxnb55} by mathematical induction.   Suppose that 
\#\label{eq:lxnb5}
J_2(K_N, b_h)\leq J_2(K_N, b_{h-1})\leq  \cdots \leq J_2(K_N, b_0),
\#
Recall that by Algorithm \ref{algo:ac_lqr}, the parameter $b$ is updated via
\#\label{eq:ac_update_bn}
b_{h+1} = b_h -  \gamma^b\cdot \hat\nabla_b J_2(K_N, b_h).
\#
Here 
\#\label{eq:def-grad-2-hat}
\hat\nabla_b J_2(K_N, b_h) = \hat \Upsilon_{K_N}^{22}(-K_N\hat \mu_{K_N, b_h} + b_h) + \hat \Upsilon_{K_N}^{21}\hat \mu_{K_N,b_h} +\hat  q_{K_N,b_h},
\# 
where  $\hat \Upsilon_{K_N}$ and  $\hat  q_{K_N,b_h}$ are the outputs of Algorithm \ref{algo:pg_eval}.  We define $\tilde b_{h+1}$ as 
\#\label{eq:tilde_update_bn}
\tilde b_{h+1} = b_h -  \gamma^b\cdot \nabla_b J_2(K_N, b_h).
\#
Here 
\#\label{eq:def-grad-2}
\nabla_b J_2(K_N, b_h) =  \Upsilon_{K_N}^{22}(-K_N \mu_{K_N, b_h} + b_h) +  \Upsilon_{K_N}^{21} \mu_{K_N,b_h} +  q_{K_N,b_h},
\#
where  $\Upsilon_{K_N}$ and  $q_{K_N,b_h}$ are defined in \eqref{eq:def_upsilon}.    We upper bound $J_2(K_N, b_{h+1}) - J_2(K_N, b^{K_N})$ in the sequel. 
Following from \eqref{eq:tilde_update_bn} and Proposition \ref{prop:convex_J2},  we have
\#\label{eq:diff_tildeb_b1}
 J_2(K_N, \tilde b_{h+1}) - J_2(K_N, b_h)& \leq -\gamma^b/2\cdot \bigl\| \nabla_b J_2(K_N, b_h)\bigr\|_2^2\notag\\
 &  \leq -\nu_{K_N}\cdot \gamma^b\cdot \bigl[ J_2(K_N, b_h) - J_2(K_N, b^{K_N}) \bigr]\notag\\
& \leq -\nu_{K_N}\cdot \gamma^b\cdot\varepsilon < 0,
\#
where $\nu_{K_N}$ is specified in Proposition \ref{prop:convex_J2}. 
The following lemma upper bounds  $ |J_2(K_N, b_{h+1}) - J_2(K_N, \tilde b_{h+1})|$. 

\begin{lemma}\label{lemma:bound_tilde_J_2}
Assume that $J_2(K_N, b_h) \leq J_2(K_N, b_0)$.  It holds with probability at least $1-\varepsilon^{15}$ that
\$
\bigl|J_2(K_N, b_{h+1}) - J_2(K_N, \tilde b_{h+1})\bigr|  \leq  \nu_{K_N}\cdot\gamma^b\cdot \varepsilon/2,
\$
where $b_{h+1}$ and $\tilde b_{h+1}$ are defined in \eqref{eq:ac_update_bn} and \eqref{eq:tilde_update_bn}, respectively. 
\end{lemma}
\begin{proof}
See \S\ref{proof:lemma:bound_tilde_J_2} for a detailed proof. 
\end{proof}

Combining \eqref{eq:diff_tildeb_b1} and Lemma \ref{lemma:bound_tilde_J_2}, we know that  if $J_2(K_N, b_h) - J_2(K_N, b^{K_N}) \geq \varepsilon$, it holds with probability at least $1-\varepsilon^{15}$ that
\#\label{eq:lxnb6}
J_2(K_N, b_{h+1}) - J_2(K_N, b_h) & \leq J_2(K_N, \tilde b_{h+1}) - J_2(K_N, b_h) + \bigl|J_2(K_N, b_{h+1}) - J_2(K_N, \tilde b_{h+1})\bigr| \notag\\
& \leq  - \nu_{K_N} \cdot \gamma^b\cdot  \varepsilon/2 < 0. 
\#
Combining \eqref{eq:lxnb5} and \eqref{eq:lxnb6}, it holds with probability at least $1-\varepsilon^{15}$ that
\$
J_2(K_N, b_{h+1})\leq J_2(K_N, b_h)\leq  \cdots \leq J_2(K_N, b_0). 
\$ 
Following from a union bound argument and the choice of $H$ in Theorem \ref{thm:ac}, if $J_2(K_N, b_h) - J_2(K_N, b^{K_N})\geq \varepsilon$ holds for any $h\leq H$,  we have
\$
J_2(K_N, b_H)\leq J_2(K_N, b_{H-1})\leq  \cdots \leq J_2(K_N, b_0),
\$ 
which holds with probability at  least  $1-\varepsilon^{13}$.  Thus, we finish the proof of \eqref{eq:lxnb55}. 

Combining \eqref{eq:diff_tildeb_b1} and Lemma \ref{lemma:bound_tilde_J_2}, for $J_2(K_N, b_h) - J_2(K_N, b^{K_N})\geq \varepsilon/2$, we have
\$
J_2(K_N, b_{h+1}) - J_2(K_N, b^{K_N}) \leq (1-\nu_{K_N}\cdot \gamma^b)\cdot  \bigl[ J_2(K_N, b_h) - J_2(K_N, b^{K_N})\bigr],
\$ 
which holds with probability at least $1-\varepsilon^{13}$.   Meanwhile, following from a union bound argument and the choice of $H$ in Theorem \ref{thm:ac},  it holds with probability at least $1 - \varepsilon^{11}$ that
\#\label{eq:bugaile}
J_2(K_N, b_{H}) - J_2(K_N, b^{K_N}) \leq \varepsilon / 2. 
\#

By Proposition \ref{prop:convex_J2} and \eqref{eq:bugaile}, it holds with probability at least $1 - \varepsilon^{11}$ that
\#\label{eq:b-bound1}
\|b_H - b^{K_N}\|_2\leq ({2 \varepsilon/\nu_{K^*} })^{1/2}.
\#
Following from Proposition \ref{prop:J2}, we know that
\#\label{eq:b-bound2}
b^{K_N} - b^{*} & =  (K_N - K^*)Q^{-1} (I-A)^\top   \\
&\qquad \cdot \bigl[ (I-A)Q^{-1}(I-A)^\top + BR^{-1}B^\top \bigr]^{-1}  \cdot  (\overline A \mu + d). \notag
\#
Combining \eqref{eq:k-bound}, \eqref{eq:b-bound1}, and \eqref{eq:b-bound2}, it holds with probability $1 - \varepsilon^{10}$ that
\$
\|b_H - b^{K_N}\|_2\leq M_b \cdot \varepsilon^{1/2},
\$
where 
\$
M_b(\mu) & = 4 \Bigl\| Q^{-1} (I-A)^\top   \cdot \bigl[ (I-A)Q^{-1}(I-A)^\top + BR^{-1}B^\top \bigr]^{-1} \cdot  (\overline A \mu + d) \Bigr\|_2 \notag\\
&\qquad \cdot  \bigl[{\nu_{K^*}^{-1} + \sigma_{\min}^{-1}(\Psi_\epsilon)\cdot \sigma_{\min}^{-1}(R) }\bigr]^{1/2}. 
\$
We finish the proof of the theorem. 
\end{proof}

\subsection{Proof of Theorem \ref{thm:pe}}\label{proof:thm:pe}
\begin{proof}
We follow the proof of Theorem 4.2 in \cite{yyy2019aclqr}, where they only consider LQR without drift terms.  Since our proof requires much more delicate analysis, we present it here. 

\vskip5pt
\noindent\textbf{Part 1.}   We denote by $\hat \zeta$ and $\hat \xi$ the primal and dual variables generated by Algorithm \ref{algo:pg_eval}.  We define the primal-dual gap of \eqref{eq:minmax_pe} as 
\#\label{eq:def_gap}
\gap (\hat \zeta, \hat \xi) = \max_{\xi\in\cV_\xi} F(\hat\zeta, \xi) - \min_{\zeta\in\cV_\zeta} F(\zeta, \hat\xi).
\#
In the sequel, we upper bound $\|\hat \alpha_{K, b} - \alpha_{K, b}\|_2$ using \eqref{eq:def_gap}.

We define $\zeta_{K,b}$ and $\xi(\zeta)$ as
\#\label{eq:def-xi-zeta}
& \zeta_{K,b} = \bigl(J(K,b), \alpha_{K,b}^\top\bigr)^\top, \qquad   \xi(\zeta) = \argmax_\xi F(\zeta, \xi). 
\#
Following from \eqref{eq:grad-forms-pe}, we know that 
\#\label{eq:expl_xi-thm}
& \xi^1(\zeta) = \zeta^1 - J(K, b), \quad \xi^2(\zeta) = \EE_{\pi_{K,b}}\bigl[\psi(x,u)\bigr] \zeta^1+ \Theta_{K,b} \zeta^2  - \EE_{\pi_{K,b}}\bigl[ c(x,u) \psi(x,u) \bigr]. 
\#
The following lemma shows that $\zeta_{K,b}\in\cV_\zeta$ and $\xi(\zeta)\in\cV_\xi$ for any $\zeta\in\cV_\zeta$.

\begin{lemma}\label{lemma:zeta_xi}
Under the assumptions in Theorem \ref{thm:pe}, it holds that $\zeta_{K,b} = (J(K,b), \alpha_{K,b}^\top)^\top\in\cV_\zeta$.  Also, for any $\zeta\in\cV_\zeta$, the vector $\xi(\zeta)$ defined in \eqref{eq:def-xi-zeta} satisfies that $\xi(\zeta)\in\cV_\xi$. 
\end{lemma}
\begin{proof}
See \S\ref{proof:lemma:zeta_xi} for a detailed proof. 
\end{proof}
By \eqref{eq:grad-forms-pe}, we know that $\nabla_\zeta F(\zeta_{K,b}, 0) = 0$ and $\nabla_\xi F(\zeta_{K,b}, 0) = 0$. Combining Lemma \ref{lemma:zeta_xi}, it holds that $(\zeta_{K,b}, 0)$ is a saddle point of the function $F(\zeta, \xi)$ defined in \eqref{eq:minmax_pe}. 

Following from \eqref{eq:def_gap}, it holds that
\#\label{eq:gap1}
& \Bigl \|  \EE_{\pi_{K,b}}\bigl[\psi(x,u)\bigr] \hat\zeta^1+ \Theta_{K,b} \hat\zeta^2 - \EE_{\pi_{K,b}}\bigl[ c(x,u) \psi(x,u) \bigr] \Bigr\|^2_2   + \bigl |\hat\zeta^1 - J(K,b)\bigr |^2\notag\\
&\qquad = F\bigl(\hat\zeta, \xi(\hat\zeta) \bigr) = \max_{\xi\in\cV_\xi} F(\hat\zeta, \xi) = \gap (\hat\zeta, \hat\xi) + \min_{\zeta\in\cV_\zeta} F(\zeta, \hat\xi),
\#
where the first equality comes from \eqref{eq:expl_xi-thm}, and the second equality comes from the fact that $\xi(\hat \zeta) = \argmax_{\xi\in\cV_\xi} F(\hat \zeta, \xi)$ by \eqref{eq:def-xi-zeta} and Lemma \ref{lemma:zeta_xi}.   We upper bound the RHS of \eqref{eq:gap1} and lower bound the LHS of \eqref{eq:gap1} in the sequel. 

As for the RHS of \eqref{eq:gap1}, it holds for any $\xi\in\cV_\xi$ that
\#\label{eq:gap2}
&\min_{\zeta\in\cV_\zeta} F(\zeta, \xi) \leq \min_{\zeta\in\cV_\zeta} \max_{\xi\in\cV_\xi} F(\zeta, \xi) = \min_{\zeta\in\cV_\zeta} F\bigl( \zeta, \xi(\zeta) \bigr)\notag\\
& \qquad = \frac{1}{2} \min_{\zeta\in\cV_\zeta}  \biggl\{   \Bigl \|  \EE_{\pi_{K,b}}\bigl[\psi(x,u)\bigr] \zeta^1+ \Theta_{K,b} \zeta^2 - \EE_{\pi_{K,b}}\bigl[ c(x,u) \psi(x,u) \bigr] \Bigr\|^2_2   + \bigl |\zeta^1 - J(K,b)\bigr |^2\biggr\}\notag\\
&\qquad = 0, 
\#
where the first equality comes from the fact that $\xi(\zeta) = \argmax_{\xi\in\cV_\xi} F(\zeta, \xi)$ by \eqref{eq:def-xi-zeta} and Lemma \ref{lemma:zeta_xi},  the second equality comes from \eqref{eq:expl_xi-thm}, and the last equality holds by taking $\zeta = \zeta_{K, b}\in\cV_\zeta$.  
Meanwhile, we lower bound the LHS of \eqref{eq:gap1} as
\#\label{eq:gap3}
& \Bigl \|  \EE_{\pi_{K,b}}\bigl[\psi(x,u)\bigr] \hat\zeta^1+ \Theta_{K,b} \hat\zeta^2 - \EE_{\pi_{K,b}}\bigl[ c(x,u) \psi(x,u) \bigr] \Bigr\|^2_2   + \bigl |\hat\zeta^1 - J(K,b)\bigr |^2\notag\\
&\qquad =\bigl \| \tilde \Theta_{K, b} (\hat\zeta - \zeta_{K, b})\bigr \|_2^2 \geq \lambda_K^2 \cdot \| \hat\zeta - \zeta_{K, b} \|_2^2 \geq  \lambda_K^2 \cdot \| \hat\alpha_{K, b} - \alpha_{K, b} \|_2^2,
\#
where the first equality comes from the definition of $\tilde \Theta_{K,b}$ in \eqref{eq:def-tilde-theta}, and  the first inequality comes from Proposition \ref{prop:invert_theta}. Here $\lambda_K$ is defined in Proposition \ref{prop:invert_theta}.  Combining \eqref{eq:gap1}, \eqref{eq:gap2}, and \eqref{eq:gap3}, it holds that
\#\label{eq:gap_error}
\|\hat \alpha_{K, b} - \alpha_{K, b}\|_2^2 \leq \lambda_K^{-2} \cdot \gap(\hat\zeta, \hat\xi),
\#
which finishes the proof of this part.

\vskip5pt
\noindent\textbf{Part 2.} We now upper bound $\gap(\hat\zeta, \hat\xi)$. 
We denote by $\tilde z_t = (\tilde x_t^\top, \tilde u_t^\top)^\top$ for $t\in[\tilde T]$, where $\tilde x_t$ and $\tilde u_t$ are generated in Line \ref{line:tilde-xu} of Algorithm \ref{algo:pg_eval}.  Following from the state transition in Problem \ref{prob:mflqr} and the form of the linear policy, $\{\tilde z_t\}_{t\in[\tilde T]}$ follows the following transition, 
\#\label{eq:tilde-trans}
\tilde z_{t+1} = L \tilde z_t + \nu + \delta_t,
\#
where
\$
\nu = \begin{pmatrix}
\overline A \mu + d\\
-K(\overline A\mu+d) + b
\end{pmatrix}, \qquad
\delta_t = \begin{pmatrix}
\omega_t\\
-K\omega_t + \sigma \eta
\end{pmatrix},\qquad
L = \begin{pmatrix}
A & B\\
-KA & -KB
\end{pmatrix}.
\$ 
Note that we have
\$
L = \begin{pmatrix}
A & B\\
-KA & -KB
\end{pmatrix} = \begin{pmatrix}
I\\
-K
\end{pmatrix}\begin{pmatrix}
A  & B
\end{pmatrix}.
\$
Then by the property of spectral radius, it holds that 
\$
\rho(L) = \rho \Biggl(  \begin{pmatrix}
A & B
\end{pmatrix}
\begin{pmatrix}
I\\
-K
\end{pmatrix}  \Biggr) = 
 \rho(A-BK) < 1. 
\$
Thus, the Markov chain generated by \eqref{eq:tilde-trans} admits a unique stationary distribution $\mathcal N(\mu_z, \Sigma_z)$, where 
\#\label{eq:def-mean-var}
\mu_z = (I-L)^{-1}\nu, \qquad \Sigma_z = L\Sigma_z L^\top  + \begin{pmatrix}
\Psi_\omega & -\Psi_\omega K^\top\\
-K\Psi_\omega & K\Psi_\omega K^\top + \sigma^2 I
\end{pmatrix}. 
\#
The following lemma characterizes the average 
\#\label{eq:def-mean-est}
\hat \mu_z = 1/\tilde T\cdot \sum_{t = 1}^{\tilde T} \tilde z_t.
\#

\begin{lemma}\label{lemma:dist_hat_muz}
It holds that
\$
\hat \mu_z\sim\mathcal N\biggl(\mu_z + \frac{1}{\tilde T}\mu_{\tilde T}, ~\frac{1}{\tilde T} \tilde \Sigma_{\tilde T} \biggr),
\$
where $\|\mu_{\tilde T}\|_2\leq M_\mu\cdot (1-\rho)^{-2}\cdot \|\mu_z\|_2$ and $\|\tilde \Sigma_{\tilde T}\|_\F\leq M_\Sigma\cdot (1-\rho)^{-1}\cdot \|\Sigma_z\|_\F$. Here $M_\mu$ and $M_\Sigma$ are positive absolute constants.   Moreover, it holds with probability at least $1 -  \tilde T^{-6}$ that
\$
\| \hat \mu_z - \mu_z \|_2 \leq   \frac{\log \tilde T}{\tilde T^{1/4}}\cdot (1-\rho)^{-2} \cdot \poly \bigl( \|\Phi_K\|_2, \|K\|_\F, \|b\|_2, \|\mu\|_2  \bigr). 
\$
\end{lemma}
\begin{proof}
See \S\ref{proof:lemma:dist_hat_muz} for a detailed proof. 
\end{proof}

Lemma \ref{lemma:dist_hat_muz} gives that
\$
\| \hat \mu_{K,b} - \mu_{K,b}\|_2 \leq \frac{\log \tilde T}{\tilde T^{1/4}}\cdot (1-\rho)^{-2} \cdot \poly \bigl( \|\Phi_K\|_2, \|K\|_\F, \|b\|_2, \|\mu\|_2  \bigr),
\$
which holds with probability at least $1 -  \tilde T^{-6}$. 

We now apply a truncation argument to show that $\gap(\hat\zeta, \hat\xi)$ is upper bounded.   We define the event $\cE$ in the sequel.   Following from Lemma \ref{lemma:dist_hat_muz},  it holds for any $z\sim\mathcal N(\mu_z, \Sigma_z)$ that
\$
z-\hat\mu_z + 1/\tilde T\cdot \mu_{\tilde T} \sim \mathcal N(0,\Sigma_z + 1/\tilde T\cdot \tilde \Sigma_{\tilde T}). 
\$
By Lemma \ref{lemma:hwineq}, there exists a positive absolute constant $C_0$ such that 
\#\label{eq:xgnb1}
& \PP\Bigl[   \bigl |  \| z- \hat \mu_z + 1/\tilde T\cdot \mu_{\tilde T}  \|_2^2 - \tr(\tilde \Sigma_z)    \bigr|   >   \tau    \Bigr]   \leq  2  \exp\Bigl[ -C_0\cdot \min\bigl(  \tau^2 \|\tilde \Sigma_z\|_\F^{-2}, ~\tau\|\tilde \Sigma_z\|_2^{-1}  \bigr) \Bigr],
\#
where we write  $\tilde \Sigma_z = \Sigma_z + 1/\tilde T\cdot \tilde \Sigma_{\tilde T}$ for notational convenience.  By taking $\tau = C_1\cdot \log T\cdot \|\tilde \Sigma_z\|_\F$ in \eqref{eq:xgnb1} for a sufficiently large positive absolute constant $C_1$, it holds that
\#\label{eq:prob1}
\PP\Bigl[   \bigl |  \| z-\hat \mu_z + 1/\tilde T\cdot \mu_{\tilde T}  \|_2^2 - \tr(\tilde \Sigma_z)    \bigr|   >   C_1\cdot \log T\cdot \|\tilde \Sigma_z\|_\F    \Bigr]  \leq T^{-6}.
\#
We define the event $\cE_{t, 1}$ for any $t\in[T]$ as
\$
\cE_{t,1} = \Bigl\{   \bigl |  \| z_t - \hat \mu_z + 1/\tilde T\cdot \mu_{\tilde T}  \|_2^2 - \tr(\tilde \Sigma_z)    \bigr|   \leq    C_1\cdot \log T\cdot \|\tilde \Sigma_z\|_\F   \Bigr\}.
\$
Then by \eqref{eq:prob1}, it holds for any $t\in[T]$ that
\#\label{eq:pp-bd1}
\PP(\cE_{t,1}) \geq 1- T^{-6}. 
\# 
Also, we define 
\#\label{eq:def-ce1}
\cE_1 = \bigcap_{t\in[T]}\cE_{t,1}.
\#
Following from a union bound argument and \eqref{eq:pp-bd1}, it holds that 
\#\label{eq:pp-bd2}
\PP(\cE_1)\geq 1 - T^{-5}.
\#
Also, conditioning on $\cE_1$, it holds for sufficiently large $\tilde T$ that
\#\label{eq:bound_z}
&\max_{t\in[T]} \|z_t - \hat \mu_z\|_2^2\notag\\
 &\qquad  \leq     C_1\cdot \log T\cdot \|\tilde \Sigma_z\|_\F + \tr(\tilde \Sigma_z) + \| 1/\tilde T\cdot \mu_{\tilde T}\|_2^2\notag \\
&\qquad\leq   2 \tilde C_1\cdot \bigl[1 + M_\Sigma(1-\rho)^{-1}/\tilde T^2 \bigr] \cdot \log T\cdot \| \Sigma_z\|_2  + M_\mu(1-\rho)^{-2} / \tilde T^2\cdot \|\mu_z\|_2^2\notag \\
&\qquad \leq C_2\cdot \log T\cdot \bigl(1+\|K\|_\F^2\bigr)\cdot \|\Phi_K\|_2\cdot(1-\rho)^{-1} +   C_3\cdot  \bigl(\|b\|_2^2 + \|\mu\|_2^2 \bigr) \cdot (1-\rho)^{-4} \cdot  \tilde T^{-2}\notag\\
& \qquad\leq  2C_2\cdot \log T\cdot \bigl(1+\|K\|_\F^2\bigr)\cdot \|\Phi_K\|_2\cdot (1-\rho)^{-1},
\#
where $\tilde C_1$, $C_2$, and $C_3$ are positive absolute constants. Here, the first inequality comes from the definition of $\cE_1$ in \eqref{eq:def-ce1}, the second inequality comes from Lemma \ref{lemma:dist_hat_muz}, and the third inequality comes from \eqref{eq:def-mean-var}.   Also, we define the following event
\#\label{eq:def-ce2nn}
\cE_2 = \bigl\{    \| \hat\mu_z -\mu_z + 1/\tilde T\cdot \mu_{\tilde T}  \|_2     \leq    C_1\bigr\}. 
\#
Then by Lemma \ref{lemma:dist_hat_muz}, we know that 
\#\label{eq:pp-bd3}
\PP(\cE_2) \geq 1 - \tilde T^{-6}
\# 
for $\tilde T$ sufficiently large.  We define the event $\cE$ as
\$ 
\cE = \cE_1 \bigcap \cE_2.
\$ 
Then following from \eqref{eq:pp-bd2}, \eqref{eq:pp-bd3}, and a union bound argument, we know that 
\$
\PP(\cE)\geq 1- T^{-5} - \tilde T^{-6}.
\$

Now, we define the truncated feature vector $\tilde \psi(x,u)$  as $\tilde \psi(x,u) = \hat \psi(x,u)\ind_{\cE}$, the truncated cost function $\tilde c(x,u)$ as $\tilde c(x,u) = c(x,u)\ind_{\cE}$, and also the truncated objective function $\tilde F(\zeta, \xi)$ as 
\#\label{eq:trun_obj}
\tilde F(\zeta, \xi) =  \Bigl\{  \EE ( \tilde \psi ) \zeta^1+ \EE \bigl[ (\tilde \psi - \tilde \psi') \tilde \psi^\top \bigr] \zeta^2 - \EE ( \tilde c \tilde \psi )\Bigr\}^\top \xi^2  + \bigl[\zeta^1 - \EE (\tilde c)\bigr] \cdot \xi^1 -  \|\xi\|_2^2 / 2,
\#
where we write $\tilde \psi = \tilde \psi(x,u)$ and $\tilde c = \tilde c(x,u)$ for notational convenience. Here the expectation is taken following the policy $\pi_{K, b}$ and the state transition. The following lemma establishes the upper bound of $|F(\zeta, \xi) - \tilde  F(\zeta, \xi)|$, where $F(\zeta, \xi)$ and $\tilde F(\zeta, \xi)$ are defined in \eqref{eq:minmax_pe} and \eqref{eq:trun_obj}, respectively. 

\begin{lemma}\label{lemma:bound_tilde_F}
It holds with probability at least $1-\tilde T^{-6}$ that
\$
\bigl |F(\zeta, \xi) - \tilde F(\zeta, \xi) \bigr|  \leq  \biggl(\frac{1}{2T} + \frac{\log \tilde T}{\tilde T^{1/4}}\biggr)\cdot (1-\rho)^{-2}\cdot \poly\bigl( \|K\|_\F, \|b\|_2, \|\mu\|_2, J(K_0, b_0) \bigr).
\$ 
\end{lemma}
\begin{proof}
See \S\ref{proof:lemma:bound_tilde_F} for a detailed proof. 
\end{proof}

Following from \eqref{eq:def_gap} and Lemma \ref{lemma:bound_tilde_F}, it holds with probability at least $1-\tilde T^{-6}$ that
\#\label{eq:gap_tilde_bound}
&\bigl| \gap(\hat \zeta, \hat \xi) - \tilde \gap(\hat \zeta, \hat \xi) \bigr|\notag\\
&\qquad \leq \biggl(\frac{1}{2T} + \frac{\log \tilde T}{\tilde T^{1/4}}\biggr)\cdot  (1-\rho)^{-2}\cdot \poly\bigl( \|K\|_\F, \|b\|_2, \|\mu\|_2, J(K_0, b_0) \bigr).
\#
where we define $\tilde \gap(\hat \zeta, \hat \xi)$ as 
\$
\tilde \gap(\hat \zeta, \hat \xi) =  \max_{\xi\in\cV_\xi} \tilde F(\hat\zeta, \xi) - \min_{\zeta\in\cV_\zeta} \tilde F(\zeta, \hat\xi). 
\$
Therefore, to upper bound of $\gap(\zeta, \xi)$, we only need to upper bound $\tilde \gap(\zeta, \xi)$.  

\vskip5pt
\noindent \textbf{Part 3.}   We upper bound $\tilde \gap(\zeta, \xi)$ in the sequel.  We first show that the trajectory generated by the policy $\pi_{K,b}$ and the state transition in Problem \ref{prob:slqr} is $\beta$-mixing. 

\begin{lemma}\label{lemma:mixing}
Consider a linear system $y_{t+1} = D y_t + \vartheta + \upsilon_t$, where $\{y_t\}_{t\geq 0}\subset \RR^m$, the matrix $D\in\RR^{m\times m}$ satisfying $\rho(D) < 1$, the vector $\vartheta\in\RR^m$, and $\upsilon_t\sim\mathcal N(0, \Sigma)$ is the Gaussians. We denote by $\varpi_t$ the marginal distribution of $y_t$ for any $t\geq0$. Meanwhile, assume that the stationary distribution of $\{y_t\}_{t\geq 0}$ is a Gaussian distribution $\mathcal N((I-D)^{-1}\vartheta, \Sigma_\infty)$, where $\Sigma_\infty$ is the covariance matrix. We define the $\beta$-mixing coefficients for any $n\geq 1$ as follows
\$
\beta(n) = \sup_{t\geq 0} \EE_{y\sim \varpi_t} \Bigl[   \bigl\| \PP_{y_n}(\cdot\given y_0 = y) - \PP_{\mathcal N((I-D)^{-1}\vartheta, \Sigma_\infty)} (\cdot )  \bigr\|_\TV   \Bigr].
\$
Then, for any $\rho\in(\rho(D), 1)$, the $\beta$-mixing coefficients satisfy that 
\$
\beta(n) \leq C_{\rho, D, \vartheta}  \cdot \bigl[ \tr(\Sigma_\infty) + m\cdot (1-\rho)^{-2} \bigr]^{1/2}\cdot \rho^n,
\$
where $C_{\rho, D, \vartheta}$ is a constant, which only depends on $\rho$, $D$, and $\vartheta$.   We say that the sequence $\{y_t\}_{t\geq 0}$ is $\beta$-mixing with parameter $\rho$. 
\end{lemma}
\begin{proof}
See Proposition 3.1 in \cite{tu2017least} for details.
\end{proof}

Recall that by \eqref{eq:f1}, the sequence $\{x_t\}_{t\geq 0}$ follows 
\$
x_{t+1} = (A-BK)x_t + (Bb + \overline A\mu + d) + \epsilon_t, \qquad \epsilon_t\sim\mathcal N(0,\Psi_\epsilon),
\$
where the matrix $A-BK$ satisfies that $\rho(A-BK) < 1$. Therefore, by Lemma \ref{lemma:mixing}, the sequence $\{z_t\}_{t\geq 0}$ is $\beta$-mixing with parameter $\rho\in(\rho(A-BK),1)$, where $z_t = (x_t^\top, u_t^\top)^\top$.   The following lemma upper bounds the primal-dual gap for a convex-concave problem.

\begin{lemma}\label{lemma:primal_dual_gap}
Let $\cX$ and $\cY$ be two compact and convex sets such that $\|x-x'\|_2 \leq M$ and $\|y-y'\|_2 \leq M$ for any $x,x'\in\cX$ and $y,y'\in\cY$. We consider solving the following minimax problem
\$
\min_{x\in\cX}\max_{y\in\cY} H(x,y) = \EE_{\epsilon\sim \varpi_\epsilon} \bigl[ G(x,y;\epsilon) \bigr],
\$ 
where the objective function $H(x,y)$ is convex in $x$ and concave in $y$. In addition, we assume that the distribution $\varpi_\epsilon$ is $\beta$-mixing with $\beta(n) \leq C_\epsilon\cdot \rho^n$, where $C_\epsilon$ is a constant. Meanwhile, we assume that it holds almost surely that $G(x,y;\epsilon)$ is $\tilde L_0$-Lipschitz in both $x$ and $y$, the gradient $\nabla_x G(x,y;\epsilon)$ is $\tilde L_1$-Lipschitz in $x$ for any $y\in\cY$, the gradient $\nabla_y G(x,y;\epsilon)$ is $\tilde L_1$-Lipschitz in $y$ for any $x\in\cX$, where $C_\epsilon, \tilde L_0, \tilde L_1 > 1$. Each step of our gradient-based method takes the following forms,
\$
x_{t+1} = \Gamma_{\cX} \bigl[ x_{t} - \gamma_{t+1} \cdot \nabla_x G(x_{t}, y_{t}; \epsilon_t)\bigr], \qquad y_{t+1} = \Gamma_{\cY} \bigl[ y_{t} - \gamma_{t+1} \cdot \nabla_y G(x_{t}, y_{t}; \epsilon_t)\bigr],
\$
where the operators $\Gamma_\cX$ and $\Gamma_\cY$ projects the variables back to $\cX$ and $\cY$, respectively, and the stepsizes take the form $\gamma_t = \gamma_0\cdot{t^{-1/2}}$ for a constant $\gamma_0 > 0$. Moreover, let $\hat x = (\sum_{t = 1}^T \gamma_t)^{-1} (\sum_{t = 1}^T \gamma_t x_t)$ and $\hat y = (\sum_{t = 1}^T \gamma_t)^{-1} (\sum_{t = 1}^T \gamma_t y_t)$ be the final output of the gradient method after $T$ iterations, then there exists a positive absolute constant $C$, such that for any $\delta\in(0,1)$, the primal-dual gap to the minimax problem is upper bounded as
\$
\max_{x\in\cX} H(\hat x, y) - \min_{y\in\cY} H(x, \hat y) \leq \frac{C\cdot (M^2 + \tilde L_0^2 + \tilde L_0\tilde L_1 M)}{\log(1/\rho)} \cdot \frac{\log^2 T + \log(1/\delta)}{\sqrt{T}} + \frac{C\cdot C_\epsilon \tilde L_0 M}{T},
\$
which holds with probability at least $1-\delta$. 
\end{lemma}
\begin{proof}
See Theorem 5.4 in \cite{yyy2019aclqr} for details. 
\end{proof}

To use Lemma \ref{lemma:primal_dual_gap}, we define the function $G(\zeta, \xi; \tilde \psi, \tilde \psi')$ as
\$
G(\zeta, \xi; \tilde \psi, \tilde \psi') = \bigl[  \tilde \psi  \zeta^1+  (\tilde \psi - \tilde \psi') \tilde \psi^\top  \zeta^2 -  \tilde c \tilde \psi \bigr]^\top \xi^2  + (\zeta^1 -  \tilde c ) \cdot \xi^1 - 1/2\cdot \|\xi\|_2^2,
\$
where $\tilde \psi = \tilde \psi(x,u)$ and $\tilde \psi' = \tilde \psi(x',u')$. Note that the gradients of $G(\zeta, \xi; \tilde \psi, \tilde \psi')$ take the form
\$
\nabla_\zeta G(\zeta, \xi; \tilde \psi, \tilde \psi') = \begin{pmatrix} 
\tilde \psi^\top \xi^2 + \xi^1\\
\tilde \psi (\tilde \psi - \tilde \psi')^\top \xi^2
\end{pmatrix}, \quad
\nabla_\xi G(\zeta, \xi; \tilde \psi, \tilde \psi') = \begin{pmatrix} 
\zeta^1 - \tilde c - \xi^1\\
\tilde \psi  \zeta^1+  (\tilde \psi - \tilde \psi') \tilde \psi^\top  \zeta^2 -  \tilde c \tilde \psi - \xi^2
\end{pmatrix}.
\$
By Definition \ref{assum:proj} and Lemma \ref{lemma:zeta_xi}, we know that
\#\label{eq:grad_bound}
& \bigl\|\nabla_\zeta G(\zeta, \xi; \tilde \psi, \tilde \psi')  \bigr\|_2 \leq  \poly \bigl(  \|K\|_\F, J(K_0, b_0)  \bigr)\cdot \log^2 T\cdot (1-\rho)^{-2}, \notag\\
& \bigl\|\nabla_\xi G(\zeta, \xi; \tilde \psi, \tilde \psi')  \bigr\|_2 \leq  \poly \bigl(  \|K\|_\F, \|\mu\|_2, J(K_0, b_0)  \bigr)\cdot \log^2 T \cdot (1-\rho)^{-2}.
\#
This gives the Lipschitz constant $\tilde L_0$ in Lemma \ref{lemma:primal_dual_gap} for $G(\zeta, \xi; \tilde \psi, \tilde \psi')$. 
Also, the Hessians of $G(\zeta, \xi; \tilde \psi, \tilde \psi')$ take the forms of
\$
\nabla^2_{\zeta\zeta} G(\zeta, \xi; \tilde \psi, \tilde \psi') = 0, \qquad \nabla^2_{\xi\xi} G(\zeta, \xi; \tilde \psi, \tilde \psi') = -I,
\$
which follows that
\#\label{eq:hess_bound}
\bigl\|\nabla^2_{\zeta\zeta} G(\zeta, \xi; \tilde \psi, \tilde \psi')\bigr\|_2 = 0, \qquad \bigl\|\nabla^2_{\xi\xi} G(\zeta, \xi; \tilde \psi, \tilde \psi')\bigr\|_2 = 1. 
\#
This gives the Lipschitz constant $\tilde L_1$ in Lemma \ref{lemma:primal_dual_gap} for $\nabla_\zeta G(\zeta, \xi; \tilde \psi, \tilde \psi')$ and $\nabla_\xi G(\zeta, \xi; \tilde \psi, \tilde \psi')$. 
Moreover, note that \eqref{eq:bound_z} provides an upper bound of $M$, combining \eqref{eq:grad_bound}, \eqref{eq:hess_bound} and Lemma \ref{lemma:primal_dual_gap}, it holds with probability at least $1 - T^{-5}$ that 
\#\label{eq:xgnb22}
\tilde \gap(\hat \zeta, \hat \xi) \leq \frac{\poly \bigl( \|K\|_\F, \|\mu\|_2, J(K_0, b_0)  \bigr)\cdot \log^6 T}{(1-\rho)^{4}\cdot \sqrt{T}}.
\#
Combining \eqref{eq:gap_error},  \eqref{eq:gap_tilde_bound}, and \eqref{eq:xgnb22}, we know that
\$
\|\hat\alpha_{K, b} - \alpha_{K, b}\|_2^2 \leq  \lambda_K^{-2}\cdot \poly_1\bigl( \|K\|_\F, \|b\|_2, \|\mu\|_2, J(K_0, b_0) \bigr)  \cdot  \biggl[  \frac{\log^6 T}{T^{1/2} \cdot (1-\rho)^{4}} +  \frac{\log \tilde T}{\tilde T^{1/4}\cdot (1-\rho)^{2}}\biggr].
\$
Same bounds for $\|\hat\Upsilon_K - \Upsilon_K\|_\F^2$, $\|\hat p_{K,b} - p_{K,b}\|_2^2$, and $\|\hat q_{K,b} - q_{K,b}\|_2^2$ hold.  We finish the proof of the theorem. 
\end{proof}


\section{Proofs of Propositions}

\subsection{Proof of Proposition \ref{prop:uniq_eq}}\label{proof:prop:uniq_eq}
\begin{proof}
We follow a similar proof as in the one of Theorem 1.1 in \cite{sznitman1991topics} and Theorem 3.2 in \cite{bensoussan2016linear}. 
Note that for any policy $\pi_{K,b}\in\Pi$, the parameters $K$ and $b$ uniquely determine the policy.  We define the following metric on $\Pi$. 
\begin{definition}\label{def:metric_pi}
For any $\pi_{K_1, b_1}, \pi_{K_2, b_2}\in\Pi$, we define the following metric,
\$
\|\pi_{K_1, b_1} - \pi_{K_2, b_2}\|_2 = c_1\cdot \|K_1 - K_2\|_2 + c_2\cdot \|b_1 - b_2\|_2,
\$
where $c_1$ and $c_2$ are positive constants. 
\end{definition}
One can verify that Definition \ref{def:metric_pi} satisfies the requirement of being a metric.  We first evaluate the forms of the operators $\Lambda_1(\cdot)$ and $\Lambda_2(\cdot, \cdot)$. 

\vskip5pt
\noindent \textbf{Forms of the operators $\Lambda_1(\cdot)$ and $\Lambda_2(\cdot, \cdot)$. }   By the definition of  $\Lambda_1(\mu)$, which gives the optimal policy under the mean-field state $\mu$,  it holds that
\$
\Lambda_1(\mu) = \pi_\mu^*,
\$
where $\pi_\mu^*$ solves Problem \ref{prob:slqr}.  This gives the form of $\Lambda_1(\cdot)$.  We now turn to $\Lambda_2(\mu, \pi)$, which gives the mean-field state $\mu_{\rm new}$ generated by the policy $\pi$ under the current mean-field state $\mu$.  In Problem \ref{prob:slqr}, the sequence of states $\{x_t\}_{t\geq 0}$ constitutes a Markov chain, which admits a unique stationary distribution. Thus, by the state transition in Problem \ref{prob:slqr} and the form of the linear-Gaussian policy, we have 
\#\label{eq:xgnb5}
\mu_{\text{new}} = (A-BK_{\pi})\mu_{\text{new}} + (Bb_{\pi} + \overline A\mu + d),
\#
where $K_\pi$ and $b_\pi$ are parameters of the policy $\pi$. 
By solving \eqref{eq:xgnb5} for $\mu_{\rm new}$, it holds that 
\$
\Lambda_2(\mu, \pi) = \mu_{\text{new}} = (I-A+BK_{\pi})^{-1}(Bb_{\pi} + \overline A\mu + d). 
\$
This gives the form of $\Lambda_2(\cdot, \cdot)$.  

\vskip5pt
Next, we compute the Lipschitz constants for $\Lambda_1(\cdot)$ and $\Lambda_2(\cdot, \cdot)$.

\vskip5pt
\noindent \textbf{Lipschitz constant for $\Lambda_1(\cdot)$.}   By Proposition \ref{prop:J2}, for any $\mu_1, \mu_2\in\RR^m$, the optimal  $K^*$ is fixed for Problem \ref{prob:slqr}. Therefore, by the form of the optimal $b^K$ given in Proposition \ref{prop:J2}, it holds that
\#\label{eq:lip1}
\bigl\| \Lambda_1(\mu_1) - \Lambda_1(\mu_2)  \bigr\|_2 & \leq  c_2\cdot   \Bigl\| \bigl[ (I-A)Q^{-1}(I-A)^\top + BR^{-1}B^\top \bigr]^{-1} \overline A \Bigr\|_2\notag\\
& \qquad \cdot  \Bigl\| \bigl[ K^*Q^{-1} (I-A)^\top - R^{-1}B^\top \bigr]  \Bigr\|_2 \cdot \|\mu_1 - \mu_2\|_2\notag\\
& = c_2   L_1\cdot \|\mu_1 - \mu_2\|_2,
\#
where $L_1$ is defined in Assumption \ref{assum:contraction}. 

\vskip5pt
\noindent \textbf{Lipschitz constants for $\Lambda_2(\cdot, \cdot)$.}   By Proposition \ref{prop:J2}, for any $\mu_1, \mu_2\in\RR^m$, the optimal $K^*$ is fixed for Problem \ref{prob:slqr}. Thus,  for any $\pi\in\Pi$ such that $\pi$ is an optimal policy under some $\mu\in\RR^m$, it holds that
\#\label{eq:lip2}
\bigl \| \Lambda_2(\mu_1, \pi) - \Lambda_2(\mu_2, \pi) \bigr\|_2& =  \bigl\|(I-A+BK_\pi)^{-1}\cdot \overline A\cdot (\mu_1-\mu_2) \bigr\|_2\notag\\
& \leq \bigl[1-\rho(A-BK^*)\bigr]^{-1}\cdot \|\overline A\|_2\cdot \|\mu_1 - \mu_2\|_2\notag\\
& =  L_2\cdot \|\mu_1 - \mu_2\|_2,
\#
where $L_2$ is defined in Assumption \ref{assum:contraction}, and $K_\pi = K^*$ is the parameter of the policy $\pi$.   Meanwhile, for any mean-field state $\mu\in\RR^m$, and any poicies $\pi_1, \pi_2\in\Pi$ that are optimal under some mean-field states $\mu_1$, $\mu_2$, respectively, we have
\#\label{eq:lip3}
\bigl \| \Lambda_2(\mu, \pi_1) - \Lambda_2(\mu, \pi_2) \bigr\|_2& =  \bigl\|(I-A+BK^*)^{-1} B\cdot (b_{\pi_1}-b_{\pi_2}) \bigr\|_2\notag\\
& \leq \bigl[1-\rho(A-BK^*)\bigr]^{-1}\cdot \|B\|_2\cdot \|b_{\pi_1}-b_{\pi_2}\|_2\notag\\
& = c_2^{-1} L_3\cdot \| \pi_1 - \pi_2 \|_2,
\#
where in the last equality, we use the fact that $K_{\pi_1} = K_{\pi_2} = K^*$ by Proposition \ref{prop:J2}. Here $L_3$ is defined in Assumption \ref{assum:contraction}, and $b_{\pi_1}$ and $b_{\pi_2}$ are the parameters of the policies $\pi_1$ and $\pi_2$. 

\vskip5pt
Now we show that the operator $\Lambda(\cdot)$ is a contraction.  For any $\mu_1, \mu_2\in\RR^m$, it holds that
\$
&\bigr\|\Lambda(\mu_1) - \Lambda(\mu_2)\bigr\|_2  = \Bigr\|\Lambda_2\bigl(\mu_1, \Lambda_1(\mu_1)\bigr) - \Lambda_2\bigl(\mu_2, \Lambda_1(\mu_2)\bigr)\Bigr\|_2\\
&\qquad \leq \Bigr\|\Lambda_2\bigl(\mu_1, \Lambda_1(\mu_1)\bigr) - \Lambda_2\bigl(\mu_1, \Lambda_1(\mu_2)\bigr)\Bigr\|_2 + \Bigr\|\Lambda_2\bigl(\mu_1, \Lambda_1(\mu_2)\bigr) - \Lambda_2\bigl(\mu_2, \Lambda_1(\mu_2)\bigr)\Bigr\|_2\\
& \qquad \leq c_2^{-1} L_3\cdot  \bigr\|\Lambda_1(\mu_1) -  \Lambda_1(\mu_2) \bigr\|_2 + L_2 \cdot \|\mu_1 - \mu_2\|_2 \\
& \qquad \leq c_2^{-1}  L_3\cdot c_2 L_1\cdot \|\mu_1 - \mu_2\|_2 + L_2 \cdot \|\mu_1 - \mu_2\|_2 = (L_1L_3 + L_2)\cdot \|\mu_1 - \mu_2\|_2,
\$
where the first inequality comes from triangle inequality, the second inequality comes from \eqref{eq:lip2} and \eqref{eq:lip3}, and the last inequality comes from \eqref{eq:lip1}. By Assumption \ref{assum:contraction}, we know that $L_0 = L_1L_3 + L_2 < 1$, which shows that the operator $\Lambda(\cdot)$ is a contraction.  Moreover, by Banach fixed-point theorem, we obtain that $\Lambda(\cdot)$ has a unique fixed point, which gives the unique equilibrium pair of Problem \ref{prob:mflqr}.  We finish the proof of the proposition. 
\end{proof}

\subsection{Proof of Proposition \ref{prop:J2}}\label{proof:prop:J2}
\begin{proof}
By the definition of $J_2(K,b)$ in \eqref{eq:a2} and the definition of $\mu_{K,b}$ in \eqref{eq:f2q}, the problem 
\$
\min_b J_2(K,b)
\$ 
is equivalent to the following constrained problem, 
\#\label{eq:a4}
& \min_{\mu, b}~~ \begin{pmatrix}
\mu\\
b
\end{pmatrix}^\top
\begin{pmatrix}
Q+K^\top RK & -K^\top R\\
-RK & R
\end{pmatrix}
\begin{pmatrix}
\mu\\
b
\end{pmatrix}\notag\\
& \text{s.t.} ~~ (I-A+BK)\mu - (Bb+\overline{A}\mu + d) = 0. 
\#
Following from the KKT conditions of \eqref{eq:a4}, it holds that
\#\label{eq:a5}
 2M_K
\begin{pmatrix}
\mu\\
b
\end{pmatrix}
+
N_K \lambda = 0, \qquad
 N_K^\top 
\begin{pmatrix}
\mu\\
b
\end{pmatrix} + \overline A \mu + d = 0,
\#
where 
\$
M_K = \begin{pmatrix}
Q+K^\top RK & -K^\top R\\
-RK & R
\end{pmatrix}, \qquad N_K = \begin{pmatrix}
-(I-A+BK)^\top\\
B^\top
\end{pmatrix}. 
\$
By solving \eqref{eq:a5}, the minimizer of \eqref{eq:a4} takes the form of
\#\label{eq:a7}
\begin{pmatrix}
\mu_{K,b^K}\\
b^K
\end{pmatrix}
 = -M_K^{-1}N_K (N_K^\top M_K^{-1}N_K)^{-1}(\overline A\mu+d). 
\#
By substituting \eqref{eq:a7} into the definition of $J_2(K,b)$ in \eqref{eq:a2}, we have
\#\label{eq:a8}
J_2(K,b^K) = (\overline A \mu + d)^\top (N_K^\top M_K^{-1} N_K )^{-1}(\overline A \mu + d). 
\#
Meanwhile, by calculation, we have
\$
M_K^{-1} = \begin{pmatrix}
Q^{-1} & Q^{-1}K^\top\\
K Q^{-1} & KQ^{-1}K^\top +R^{-1}
\end{pmatrix}. 
\$
Therefore, the term $N_K^\top M_K^{-1}N_K$ in \eqref{eq:a8} takes the form of
\#\label{eq:wlx1}
N_K^\top M_K^{-1}N_K = (I-A)Q^{-1}(I-A^\top) + BR^{-1}B^\top. 
\#
By plugging \eqref{eq:wlx1} into \eqref{eq:a8}, we have
\$
J_2(K, b^K) = (\overline A \mu + d)^\top \bigl[ (I-A)Q^{-1}(I-A^\top) + BR^{-1}B^\top \bigr]^{-1} (\overline A \mu + d).
\$
Also, by plugging \eqref{eq:wlx1} into \eqref{eq:a7}, we have
\$
\begin{pmatrix}
\mu_{K,b^K}\\
b^K
\end{pmatrix}
 = \begin{pmatrix}
 Q^{-1} (I-A)^\top\\
 KQ^{-1}(I-A)^\top - R^{-1}B^\top
 \end{pmatrix} \bigl[ (I-A)Q^{-1}(I-A)^\top + BR^{-1}B^\top \bigr]^{-1}  (\overline A \mu + d). 
\$
We finish the proof of the proposition. 
\end{proof}

\subsection{Proof of Proposition \ref{prop:cost_form}}\label{proof:prop:cost_form}
\begin{proof}
By the definition of the cost function $c(x,u)$ in Problem \ref{prob:slqr} (recall that we drop the subscript $\mu$ when we focus on Problem \ref{prob:slqr}), we have
\#\label{eq:f4}
\EE c_t& = \EE(x_t^\top Q x_t  + u_t^\top R u_t + \mu^\top \overline Q \mu)\notag\\
& = \EE(x_t^\top Q x_t  + x_t^\top K^\top R K x_t - 2b^\top RK x_t + b^\top R b + \sigma^2 \eta_t^\top R\eta_t + \mu^\top \overline Q \mu)\notag\\
& = \EE\bigl[ x_t^\top(Q + K^\top R K) x_t - 2b^\top RK x_t \bigr] + b^\top R b + \sigma^2\cdot \tr(R) + \mu^\top \overline Q \mu,
\#
where we write $c_t = c(x_t, u_t)$ for notational convenience. Here in the second line we use $u_t = \pi_{K,b}(x_t) =  -K x_t + b + \sigma \eta_t$. Therefore, combining \eqref{eq:f4} and the definition of $J(K,b)$ in Problem \ref{prob:slqr}, we have
\#\label{eq:ffff1}
J(K,b)& = \lim_{T\to\infty} \frac{1}{T}\sum_{t = 0}^T \Bigl\{ \EE\bigl[ x_t^\top(Q + K^\top R K) x_t  - 2b^\top RK x_t\bigr] + b^\top R b + \sigma^2\cdot \tr(R) +  \mu^\top \overline Q \mu \Bigr\}\notag\\
& = \EE_{x\sim\mathcal N(\mu_{K,b}, \Phi_K)}\bigl[ x^\top(Q + K^\top R K) x  - 2b^\top R K x \bigr] + b^\top R b + \sigma^2\cdot \tr(R) + \mu^\top \overline Q \mu \notag \\
& = \tr\bigl[(Q + K^\top R K)\Phi_K\bigr] + \mu_{K,b}^\top (Q+K^\top RK)\mu_{K,b} - 2b^\top RK\mu_{K,b}  \\
& \qquad + b^\top Rb + \sigma^2\cdot \tr(R) + \mu^\top \overline Q\mu. \notag
\#
Now, by iteratively applying \eqref{eq:f2p} and \eqref{eq:bellman}, we have
\#\label{eq:ffff2}
\tr\bigl[(Q + K^\top R K)\Phi_K\bigr] = \tr(P_K\Psi_\epsilon),
\#
where $P_K$ is given in \eqref{eq:bellman}. Combining \eqref{eq:ffff1} and \eqref{eq:ffff2}, we know that
\$
J(K,b) = J_1(K) + J_2(K,b) + \sigma^2 \cdot \tr(R) + \mu^\top \overline Q \mu,
\$
where
\$
& J_1(K) =\tr\bigl[(Q + K^\top R K)\Phi_K\bigr]  = \tr(P_K \Psi_\epsilon),\notag\\ 
& J_2(K,b) = \begin{pmatrix}
\mu_{K,b}\\
b
\end{pmatrix}^\top
\begin{pmatrix}
Q + K^\top RK & -K^\top R\\
-R K & R
\end{pmatrix}
 \begin{pmatrix}
\mu_{K,b}\\
b
\end{pmatrix}.
\$
We finish the proof of the proposition.  
\end{proof}

\subsection{Proof of Proposition \ref{prop:convex_J2}}\label{proof:prop:convex_J2}
\begin{proof}
By calculating the Hessian matrix of $J_2(K, b)$, we have
\$
 \nabla^2_{bb} J_2(K, b) = & B^\top (I-A+BK)^{-\top} (Q + K^\top RK)(I-A+BK)^{-1} B \notag\\
&\qquad - \bigl[RK(I-A+BK)^{-1}B + B^\top (I-A+BK)^{-\top}  K^\top R\bigr] + R\\
 = & \bigl[{R}^{1/2}K(I-A+BK)^{-1}B - {R}^{1/2}\bigr]^\top\bigl[{R}^{1/2}K(I-A+BK)^{-1}B - {R}^{1/2}\bigr] \notag\\
&\qquad + B^\top (I-A+BK)^{-\top} Q (I-A+BK)^{-1} B,
\$
which is a positive definite matrix independent of $b$. We denote by its minimum singular value as $\nu_K$. 
Also, note that  $\|\nabla^2_{bb} J_2(K, b)\|_2$ is upper bounded as
\$
\bigl\|\nabla^2_{bb} J_2(K, b)\bigr\|_2 \leq \bigl[1-\rho(A-BK)\bigr]^{-2} \cdot \bigl(  \|B\|_2^2\cdot \|K\|_2^2\cdot \|R\|_2 + \|B\|_2^2\cdot \|Q\|_2 \bigr). 
\$
Therefore, it holds that 
\$
\iota_K \leq \bigl[1-\rho(A-BK)\bigr]^{-2} \cdot \bigl(  \|B\|_2^2\cdot \|K\|_2^2\cdot \|R\|_2 + \|B\|_2^2\cdot \|Q\|_2 \bigr),
\$
where $\iota_K$ is the maximum singular value of $\nabla^2_{bb} J_2(K, b)$. We finish the proof of the proposition. 
\end{proof}

\subsection{Proof of Proposition \ref{prop:pg}}\label{proof:prop:pg}
\begin{proof}
Following from Proposition \ref{prop:cost_form}, it holds that 
\#\label{eq:j1-form}
J_1(K) = \tr(P_K \Psi_\epsilon) = \EE_{y\sim \mathcal N(0, \Psi_\epsilon)}  ( y^\top P_K y ) = \EE_{y\sim \mathcal N(0, \Psi_\epsilon)} \bigl[ f_K ( y ) \bigr], 
\#
where $f_K(y) = y^\top P_K y$.  By the definition of $P_K$ in \eqref{eq:bellman}, we obtain that
\#\label{eq:af1}
\nabla_K f_K(y) & = \nabla_K \Bigl\{  y^\top (Q + K^\top RK) y + \bigl[ (A-BK) y \bigr]^\top P_K \bigl[ (A-BK) y \bigr]^\top  \Bigr\}\notag\\
& = 2 RK y y^\top + \nabla_K \Bigl[ f_K\bigl( (A-BK)y \bigr) \Bigr]. 
\#
Also, we have
\#\label{eq:af2}
\nabla_K \Bigl[ f_K\bigl( (A-BK)y \bigr) \Bigr] & = \nabla_K  f_K\bigl( (A-BK)y \bigr) - 2B^\top P_K (A-BK) y y^\top. 
\#
By plugging \eqref{eq:af2} into \eqref{eq:af1}, we have
\#\label{eq:af3}
\nabla_K f_K(y) = 2\bigl[ (R+B^\top P_K B)K - B^\top P_K A \bigr] y y^\top + \nabla_K  f_K\bigl( (A-BK)y \bigr). 
\#
By iteratively applying \eqref{eq:af3}, it holds that 
\#\label{eq:af4}
\nabla_K f_K(y) = 2\bigl[ (R+B^\top P_K B)K - B^\top P_K A \bigr] \cdot \sum_{t = 0}^\infty y_t y_t^\top,
\#
where $y_{t+1} = (A-BK)y_t$ with $y_0 = y$.  Now, combining \eqref{eq:j1-form} and \eqref{eq:af4}, it holds that
\$
\nabla_K J_1(K) = 2\bigl[ (R+B^\top P_K B)K - B^\top P_K A \bigr]\Phi_K = 2(\Upsilon_K^{22} K  - \Upsilon_K^{21})\cdot \Phi_K,
\$
where $\Upsilon_K$ is defined in \eqref{eq:def_upsilon}.  Meanwhile, combining the form of $\mu_{K,b}$ in \eqref{eq:f2q}, it holds by calculation that
\$
\nabla_b J_2(K,b) = 2\bigl[\Upsilon_K^{22}(-K\mu_{K, b} + b) + \Upsilon_{K}^{21}\mu_{K,b} + q_{K,b}\bigr],
\$
where $q_{K,b}$ is defined in \eqref{eq:def_upsilon}.  We finish the proof of the proposition. 
\end{proof}

\subsection{Proof of Proposition \ref{prop:val_func_form}}\label{proof:prop:val_func_form}
\begin{proof}
From the definition of $V_{K,b}(x)$ in \eqref{eq:val_funcv} and the definition of the cost function $c(x,u)$ in Problem \ref{prob:slqr}, it holds  that 
\$
V_{K,b}(x) = \sum_{t = 0}^\infty \Bigl\{ \EE\bigl[ & x_t^\top (Q+K^\top RK)x_t - 2b^\top RKx_t \\
& + b^\top Rb + \sigma^2 \eta_t^\top R\eta_t + \mu^\top \overline Q \mu \given x_0 = x\bigr] -J(K,b) \Bigr\}.
\$
Combining  \eqref{eq:f1}, we know that $V_{K,b}(x)$ is a quadratic function taking the form of $V_{K,b}(x) = x^\top G x + r^\top x + h$, where $G$, $r$, and $h$ are functions of $K$ and $b$. Note that $V_{K,b}(x)$ satisfies that 
\#\label{eq:quad-form1}
V_{K,b}(x) = c(x,-Kx+b) - J(K,b) + \EE\bigl[ V_{K,b}(x')\given x \bigr],
\#
by substituting the form of $c(x,-Kx+b)$ in Problem \ref{prob:slqr} and $J(K,b)$ in \eqref{eq:a1} into \eqref{eq:quad-form1}, we obtain that
\#\label{eq:f6}
&x^\top G x + r^\top x + h\notag\\
&\qquad = x^\top (Q + K^\top RK)x - 2b^\top RKx + b^\top Rb + \mu^\top \overline Q \mu\\
&\qquad\qquad  - \bigl[ \tr(P_K \Psi_\epsilon) + \mu_{K,b}^\top (Q+K^\top RK)\mu_{K,b} - 2b^\top RK\mu_{K,b} + \mu^\top \overline Q \mu + b^\top Rb \bigr ]\notag\\
& \qquad\qquad+ \bigl[ (A-BK)x + (Bb+\overline A\mu + d)  \bigr]^\top G \bigl[(A-BK)x + (Bb+\overline A\mu + d)  \bigr]\notag\\
& \qquad\qquad  + \tr(G\Psi_\epsilon) + r^\top \bigl[ (A-BK)x + (Bb+\overline A\mu + d) \bigr] + h -  \sigma^2 \cdot \tr(R). \notag
\#
By comparing the quadratic terms and linear terms on both the LHS and RHS in \eqref{eq:f6}, we obtain that
\$
G = P_K, \qquad r = 2f_{K,b}, 
\$
where $f_{K,b} = (I-A+BK)^{-\top} [  (A-BK)^\top P_K (Bb+\overline A\mu + d) - K^\top Rb ]$.  Also, by the definition of $V_{K,b}(x)$ in \eqref{eq:val_funcv}, we know that $\EE[V_{K,b}(x)] = 0$, where the expectation is taken following the stationary distribution generated by the policy $\pi_{K,b}$ and the state transition.  Therefore, we have
\$
h = -2f_{K,b}\mu_{K,b} - \mu_{K,b}^\top P_K\mu_{K,b} - \tr(P_K\Phi_K),
\$
which shows that
\#\label{eq:woca}
V_{K,b}(x) = x^\top P_K x - \tr(P_K \Phi_K)  +  2 f_{K,b}^\top (x -\mu_{K,b}) - \mu_{K,b}^\top P_K\mu_{K,b}. 
\# 

For the action-value function $Q_{K,b}(x,u)$, by plugging \eqref{eq:woca} into \eqref{eq:val_funcq}, we obtain that
\$
Q_{K,b}(x,u) &= \begin{pmatrix}
x\\
u
\end{pmatrix}^\top
\Upsilon_K
\begin{pmatrix}
x\\
u
\end{pmatrix} + 2\begin{pmatrix}
p_{K,b}\\
q_{K,b}
\end{pmatrix}^\top \begin{pmatrix}
x\\
u
\end{pmatrix}   - \tr(P_K\Phi_K)  - \sigma^2\cdot \tr(R + P_K B B^\top) \notag\\
& \qquad  - b^\top Rb +2b^\top RK\mu_{K,b}  - \mu_{K,b}^\top (Q+K^\top RK + P_K)\mu_{K,b}    \notag\\
& \qquad + 2f_{K,b}^\top\bigl[ (\overline A\mu+d)-\mu_{K,b}\bigr]  + (\overline A\mu+d)^\top P_K (\overline A\mu+d). 
\$
We finish the proof of the proposition. 
\end{proof}

\subsection{Proof of Proposition \ref{prop:bellman_compact}}\label{proof:prop:bellman_compact}
\begin{proof}
By Proposition \ref{prop:val_func_form}, it holds that  $Q_{K,b}$ takes the following linear form
\#\label{eq:q3}
Q_{K,b}(x,u) =  \psi(x,u)^\top \alpha_{K,b} + \beta_{K,b},
\#
where $\beta_{K,b}$ is a scalar independent of $x$ and $u$. 
Note that $Q_{K,b}(x,u)$ satisfies that
\#\label{eq:q4}
Q_{K,b}(x,u) = c(x,u) - J(K,b) + \EE_{\pi_{K,b}}\bigl[ Q_{K,b}(x',u')\given x,u \bigr],
\#
where $(x',u')$ is the state-action pair after $(x,u)$ following the policy $\pi_{K,b}$ and the state transition. Combining \eqref{eq:q3} and \eqref{eq:q4}, we obtain that
\#\label{eq:q5}
 \psi(x,u)^\top \alpha_{K,b}  =  c(x,u) - J(K,b) +  \EE_{\pi_{K,b}}\bigl[\psi(x',u')\given x,u\bigr]^\top \alpha_{K,b}.
\#
By left multiplying $\psi(x,u)$ to both sides of \eqref{eq:q5}, and taking the expectation, we have
\$
\EE_{\pi_{K,b}}\Bigl\{ \psi(x,u) \bigl[\psi(x,u) -\psi(x',u')\bigr]^\top \Bigr \} \cdot \alpha_{K,b} + \EE_{\pi_{K,b}}\bigl[ \psi(x,u)\bigr] \cdot J(K,b)  =  \EE_{\pi_{K,b}}\bigl[c(x,u) \psi(x,u)\bigr]. 
\$
Combining the definition of the matrix $\Theta_{K,b}$ in \eqref{eq:q2}, we have
\$ 
\begin{pmatrix}
1 & 0\\
\EE_{\pi_{K,b}}\bigl[ \psi(x,u) \bigr] & \Theta_{K,b}
\end{pmatrix}
\begin{pmatrix}
J(K,b)\\
\alpha_{K,b}
\end{pmatrix} = \begin{pmatrix}
J(K,b)\\
\EE_{\pi_{K,b}}\bigl[ c(x,u) \psi(x,u) \bigr]
\end{pmatrix},
\$
which concludes the proof of the proposition. 
\end{proof}

\subsection{Proof of Proposition \ref{prop:invert_theta}}\label{proof:prop:invert_theta}
\begin{proof}
\noindent\textbf{Invertibility and Upper Bound.}  We denote by $z_t = (x_t^\top, u_t^\top)^\top$ for any $t\geq 0$.  Then following from the state transition and the policy $\pi_{K,b}$, the transition of $\{z_t\}_{t\geq0}$ takes the form of
\#\label{eq:system_of_z}
z_{t+1} = L z_t + \nu + \delta_t,
\#
where $L$, $\nu$ and $\delta$ are defined as
\$
L = \begin{pmatrix}
A & B\\
-KA & -KB
\end{pmatrix}, \qquad
\nu = \begin{pmatrix}
\overline A \mu + d\\
-K(\overline A\mu+d) + b
\end{pmatrix}, \qquad
\delta_t = \begin{pmatrix}
\omega_t\\
-K\omega_t + \sigma \eta_t
\end{pmatrix}.  
\$
Note that $L$ also takes the form of
\$
L = \begin{pmatrix}
I\\
-K
\end{pmatrix}
\begin{pmatrix}
A & B
\end{pmatrix}. 
\$
Combining the fact that $\rho(UV) = \rho(VU)$ for any matrices $U$ and $V$, we know that $\rho(L) = \rho(A-BK) < 1$, which verifies the stability of  \eqref{eq:system_of_z}. 
Following from the stability of \eqref{eq:system_of_z}, we know that the Markov chain generated by \eqref{eq:system_of_z} admits a unique stationary distribution $\mathcal N(\mu_z, \Sigma_z)$, where $\mu_z$ and $\Sigma_z$ satisfy that
\$
\mu_z = L\mu_z + \nu, \qquad \Sigma_z = L\Sigma_z L^\top  + \Psi_\delta.
\$
where
\$
\Psi_\delta = \begin{pmatrix}
\Psi_\omega & -\Psi_\omega K^\top\\
-K \Psi_\omega & K\Psi_\omega K^\top +\sigma^2 I 
\end{pmatrix}. 
\$
Also, we know that  $\Sigma_z$ takes the form of
\#\label{eq:sigma_z_factor}
\Sigma_z = {\rm Cov} \Biggl[ \begin{pmatrix}
x\\
u
\end{pmatrix}\Biggr]
 = \begin{pmatrix}
\Phi_K & -\Phi_K K^\top\\
-K\Phi_K & K\Phi_K K^\top + \sigma^2  I
\end{pmatrix} = \begin{pmatrix}
0 & 0\\
0 & \sigma^2 I
\end{pmatrix} + \begin{pmatrix}
I\\
-K
\end{pmatrix}\Phi_K\begin{pmatrix}
I\\
-K
\end{pmatrix}^\top,
\#
where $\Phi_K$ is defined in \eqref{eq:f2p}.

The following lemma establishes the form of $\Theta_{K,b}$. 

\begin{lemma}\label{lemma:theta_form}
The matrix $\Theta_{K,b}$ in \eqref{eq:q2} takes the form of
\$
\Theta_{K,b} = \begin{pmatrix}
2 (\Sigma_z\otimes_s \Sigma_z) (I - L \otimes_s L)^\top & 0 \\
0 & \Sigma_z (I-L)^\top
\end{pmatrix}. 
\$
\end{lemma}
\begin{proof}
See \S\ref{proof:lemma:theta_form} for a detailed proof. 
\end{proof}

Note that since $\rho(L) < 1$, both $I-L \otimes_s L$ and $I-L$ are positive definite. Therefore, by  Lemma \ref{lemma:theta_form}, the matrix $\Theta_{K,b}$ is invertible. This finishes the proof of the invertibility of $\Theta_{K,b}$.  Moreover, by \eqref{eq:sigma_z_factor} and Lemma \ref{lemma:theta_form}, we upper bound the spectral norm of $\Theta_{K,b}$ as 
\$
\|\Theta_{K,b}\|_2 & \leq 2 \max\Bigl\{ \|\Sigma_z\|_2^2\cdot \bigl(1 + \|L\|_2^2\bigr),~ \|\Sigma_z\|_2\cdot \bigl(1 + \|L\|_2\bigr) \Bigr \} \leq 4 \bigl( 1 + \|K\|_\F^2 \bigr)^2\cdot \|\Phi_K\|_2^2,
\$
which proves the upper bound of $\|\Theta_{K,b}\|_2$. 

\vskip5pt
\noindent\textbf{Minimum singular value.}  To lower bound $\sigma_{\min}(\tilde \Theta_{K,b})$, we only need to upper bound $\sigma_{\max}(\tilde \Theta_{K,b}^{-1}) = \|\tilde \Theta_{K,b}^{-1}\|_2$.  We first calculate $\tilde \Theta_{K,b}^{-1}$. 
Recall that the matrix $\tilde \Theta_{K,b}$ in \eqref{eq:linear_sys} takes the form of 
\$
\tilde \Theta_{K,b} = \begin{pmatrix}
1 & 0\\
\EE_{\pi_{K,b}}\bigl[ \psi(x,u)  \bigr] & \Theta_{K,b}
\end{pmatrix}. 
\$
By the definition of the feature vector $\psi(x,u)$ in \eqref{eq:def_feature}, the vector $\tilde \sigma_z = \EE_{\pi_{K,b}}[ \psi(x,u) ]$ takes the form of
\$
\tilde \sigma_z = \EE_{\pi_{K,b}}\bigl[ \psi(x,u) \bigr] = \begin{pmatrix}
\svec(\Sigma_z)\\
\mathbf 0_{k+m}
\end{pmatrix}, 
\$
where $\mathbf 0_{k+m}$ denotes the all-zero column vector with dimension $k+m$. Also, since  $\Theta_{K,b}$ is invertible, the matrix $\tilde \Theta_{K,b}$ is also invertible, whose inverse takes the form of
\$
\tilde \Theta_{K,b}^{-1} = \begin{pmatrix}
1 & 0\\
-\Theta_{K,b}^{-1}\cdot  \tilde \sigma_z & \Theta_{K,b}^{-1}
\end{pmatrix}. 
\$

The following lemma upper bounds the spectral norm of $\tilde \Theta_{K,b}^{-1}$. 

\begin{lemma}\label{lemma:upper_bound_tilde_theta}
The spectral norm of the matrix $\tilde \Theta_{K,b}^{-1}$ is upper bounded by a positive constant $\tilde \lambda_K$, where $\tilde \lambda_K$ only depends on $\|K\|_2$ and $\rho(A-BK)$. 
\end{lemma}
\begin{proof}
See \S\ref{proof:lemma:upper_bound_tilde_theta} for a detailed proof. 
\end{proof}

By Lemma \ref{lemma:upper_bound_tilde_theta}, we know that $\sigma_{\min}(\tilde\Theta_{K,b})$ is lower bounded by a positive constant $\lambda_K = 1/\tilde \lambda_K$, which only depends on $\|K\|_2$ and $\rho(A-BK)$. This concludes the proof of the proposition. 
\end{proof}


\section{Proofs of Lemmas}

\subsection{Proof of Lemma \ref{lemma:cost_diff}}\label{proof:lemma:cost_diff}
\begin{proof}
Following from \eqref{eq:bellman}, it holds that
\#\label{eq:381}
 y^\top P_{K_2} y = \sum_{t\geq 0} y^\top \bigl[ (A-BK_2)^t  \bigr]^\top (Q + K_2^\top RK_2) (A-BK_2)^t y. 
\#
Meanwhile, by the state transition $y_{t + 1} = (A-BK_2) y_t$, we know that 
\#\label{eq:382}
y_t = (A-BK_2)^t y_0 = (A-BK_2)^t y.
\#
By plugging \eqref{eq:382} into \eqref{eq:381}, it holds that
\#\label{eq:ooo1}
 y^\top P_{K_2} y = \sum_{t\geq 0} y_t^\top (Q + K_2^\top RK_2) y_t =  \sum_{t\geq 0} (y_t^\top Q y_t + y_t^\top K_2^\top RK_2 y_t).
\#
Also, it holds that
\#\label{eq:ooo2}
y^\top P_{K_1} y = \sum_{t\geq 0} (y_{t+1}^\top P_{K_1} y_{t + 1}-  y_{t}^\top P_{K_1} y_{t })
\#
Combining \eqref{eq:ooo1} and \eqref{eq:ooo2}, we have 
\#\label{eq:telesum}
 y^\top P_{K_2} y -  y^\top P_{K_1} y =  \sum_{t\geq 0} (y_t^\top Q y_t + y_t^\top K_2^\top RK_2 y_t  +  y_{t+1}^\top P_{K_1} y_{t + 1}-  y_{t}^\top P_{K_1} y_{t } ). 
\#
Also, by the state transition $y_{t + 1} = (A-BK_2)y_t$,  it holds for any $t\geq 0$ that
\#\label{eq:telesum2}
& y_t^\top Q y_t + y_t^\top K_2^\top RK_2 y_t  +  y_{t+1}^\top P_{K_1} y_{t + 1}-  y_{t}^\top P_{K_1} y_{t }\notag\\
& \qquad = y_t^\top \bigl[ Q + (K_2 - K_1 + K_1)^\top R (K_2 - K_1 + K_1) \bigr]y_t\notag\\
&\qquad\qquad +  y_t^\top \bigl[ A-BK_1 - B(K_2 - K_1) \bigr]^\top P_{K_1} \bigl[ A-BK_1 - B(K_2 - K_1) \bigr] y_t - y_t^\top P_{K_1}y_t\notag\\
& \qquad = 2y_t^\top  (K_2 - K_1)^\top \bigl[ (R + B^\top P_{K_1}B)K_1 - B^\top P_{K_1}A \bigr] y_t\notag\\
&\qquad \qquad + y_t^\top (K_2 - K_1)^\top (R + B^\top P_{K_1}B)(K_2 - K_1) y_t\notag\\
& \qquad = 2y_t^\top  (K_2 - K_1)^\top ( \Upsilon_{K_1}^{22} K_1 - \Upsilon_{K_1}^{21} ) y_t + y_t^\top (K_2 - K_1)^\top \Upsilon_{K_1}^{22} (K_2 - K_1) y_t,
\#
where the matrix $\Upsilon_{K_1}$ is defined in \eqref{eq:def_upsilon}.  By plugging \eqref{eq:telesum2} into \eqref{eq:telesum}, we have
\$
& y^\top P_{K_2} y -  y^\top P_{K_1} y \notag\\
&\qquad = \sum_{t\geq 0}  2y_t^\top  (K_2 - K_1)^\top ( \Upsilon_{K_1}^{22} K_1 - \Upsilon_{K_1}^{21} ) y_t + y_t^\top (K_2 - K_1)^\top \Upsilon_{K_1}^{22} (K_2 - K_1) y_t\notag\\
&\qquad = \sum_{t\geq 0} D_{K_1, K_2}(y_t),
\$
where $D_{K_1, K_2}(y) = 2y^\top (K_2 - K_1)(\Upsilon_{K_1}^{22} K_1  - \Upsilon_{K_1}^{21}) y + y^\top (K_2 - K_1)^\top\Upsilon_{K_1}^{22} (K_2 - K_1)y$.  We finish the proof of the lemma. 
\end{proof}

\subsection{Proof of Lemma \ref{lemma:grad_dom}}\label{proof:lemma:grad_dom}
\begin{proof}
We prove  \eqref{eq:upper-bdJ} and  \eqref{eq:lower-bdJ} separately in the sequel. 

\vskip5pt
\noindent\textbf{Proof of \eqref{eq:upper-bdJ}.}    From the definition of $J_1(K)$ in \eqref{eq:a2}, we have
\#\label{eq:JJ_diff1}
J_1(K) - J_1(K^*)&  = \tr( P_K\Psi_\epsilon - P_{K^*}\Psi_\epsilon ) = \EE_{y\sim \mathcal N(0, \Psi_\epsilon)} (y^\top P_K y - y^\top P_{K^*} y )\notag\\
& =  -\EE\biggl[ \sum_{t\geq 0} D_{K, K^*} (y_t) \biggr],
\#
where in the last equality, we apply Lemma \ref{lemma:cost_diff} and the expectation is taken following the transition $y_{t+ 1} = (A- BK^*)y_t$ with initial state $y_0 \sim\mathcal N(0,\Psi_\epsilon)$.  Here we denote by $D_{K,K^*}(y)$ as
\$
D_{K, K^*}(y) = 2y^\top (K^* - K)(\Upsilon_{K}^{22} K  - \Upsilon_{K}^{21}) y + y^\top (K^* - K)^\top\Upsilon_{K}^{22} (K^* - K)y.
\$
Also,  we write $D_{K, K^*}(y)$ as
\#\label{eq:bound_D1}
D_{K, K^*}(y)& = 2y^\top (K^* - K)(\Upsilon_K^{22} K  - \Upsilon_K^{21}) y + y^\top (K^* - K)^\top\Upsilon_K^{22} (K^* - K) y \\
& = y^\top \bigl[ K^* - K + (\Upsilon_K^{22})^{-1}  (\Upsilon_K^{22} K  - \Upsilon_K^{21}) \bigr]^\top \Upsilon_K^{22} \bigl[ K^* - K + (\Upsilon_K^{22})^{-1}  (\Upsilon_K^{22} K  - \Upsilon_K^{21}) \bigr] y  \notag\\
&\qquad\qquad -y^\top (\Upsilon_K^{22} K  - \Upsilon_K^{21})^\top (\Upsilon_{K}^{22})^{-1}(\Upsilon_K^{22} K  - \Upsilon_K^{21})y. \notag
\#
Note that the first term on the RHS of \eqref{eq:bound_D1} is positive, due to the fact that it is a quadratic form of a positive definite matrix, we lower bound $D_{K, K^*}(y)$ as
\#\label{eq:bound_D}
D_{K, K^*}(y) \geq -y^\top (\Upsilon_K^{22} K  - \Upsilon_K^{21})^\top (\Upsilon_{K}^{22})^{-1}(\Upsilon_K^{22} K  - \Upsilon_K^{21})y. 
\#
Combining \eqref{eq:JJ_diff1} and \eqref{eq:bound_D}, it holds that
\$
J_1(K) - J_1(K^*) & \leq   \biggl\|  \EE  \biggl( \sum_{t\geq 0}  y_t y_t^\top\biggr) \biggr \|_2   \cdot  \tr\bigl[(\Upsilon_K^{22} K  - \Upsilon_K^{21})^\top (\Upsilon_{K}^{22})^{-1}(\Upsilon_K^{22} K  - \Upsilon_K^{21})\bigr]\notag\\
& = \| \Phi_{K^*} \|_2\cdot \tr\bigl[(\Upsilon_K^{22} K  - \Upsilon_K^{21})^\top (\Upsilon_{K}^{22})^{-1}(\Upsilon_K^{22} K  - \Upsilon_K^{21})\bigr]\notag\\
& \leq  \bigl\|(\Upsilon_{K}^{22})^{-1}\bigr\|_2 \cdot   \| \Phi_{K^*} \|_2\cdot \tr\bigl[(\Upsilon_K^{22} K  - \Upsilon_K^{21})^\top (\Upsilon_K^{22} K  - \Upsilon_K^{21})\bigr]\notag\\
& \leq \sigma_{\min}^{-1}(R) \cdot   \| \Phi_{K^*} \|_2\cdot \tr\bigl[(\Upsilon_K^{22} K  - \Upsilon_K^{21})^\top (\Upsilon_K^{22} K  - \Upsilon_K^{21})\bigr],
\$
where the last line comes from the fact that $\Upsilon_{K}^{22} = R + B^\top P_K B\succeq R$.  This complete the proof of \eqref{eq:upper-bdJ}. 

\vskip5pt
\noindent\textbf{Proof of \eqref{eq:lower-bdJ}.}  Note that for any $\tilde K$, it holds by the optimality of $K^*$ that
\#\label{eq:JJ_diff3}
J_1(K) - J_1(K^*)\geq J_1(K) - J_1(\tilde K ) = -\EE\biggl[ \sum_{t\geq 0} D_{K, \tilde K} (y_t) \biggr],
\#
where the expectation is taken following the transition $y_{t+ 1} = (A- B\tilde K)y_t$ with initial state $y_0 \sim \mathcal N(0, \Psi_\epsilon)$.  By taking $\tilde K = K - (\Upsilon_K^{22})^{-1}  (\Upsilon_K^{22} K  - \Upsilon_K^{21})$ and following from a similar calculation as in \eqref{eq:bound_D1}, the function $D_{K, \tilde K}(y)$ takes the form of
\#\label{eq:form_D_tilde}
D_{K, \tilde K}(y)& = -y^\top (\Upsilon_K^{22} K  - \Upsilon_K^{21})^\top (\Upsilon_{K}^{22})^{-1}(\Upsilon_K^{22} K  - \Upsilon_K^{21})y. 
\#
Combining \eqref{eq:JJ_diff3} and \eqref{eq:form_D_tilde}, it holds that 
\$
J(K) - J(K^*)& \geq \tr\bigl[ \Phi_{\tilde K}  (\Upsilon_K^{22} K  - \Upsilon_K^{21})^\top (\Upsilon_{K}^{22})^{-1}(\Upsilon_K^{22} K  - \Upsilon_K^{21})  \bigr]\\
& \geq \sigma_{\min}(\Psi_\epsilon) \cdot \| \Upsilon_{K}^{22} \|_2^{-1}\cdot \tr\bigl[ (\Upsilon_K^{22} K  - \Upsilon_K^{21})^\top (\Upsilon_K^{22} K  - \Upsilon_K^{21}) \bigr],
\$
where we use the fact that $\Phi_{\tilde K} = (A-B\tilde K)\Phi_{\tilde K} (A-B\tilde K)^\top + \Psi_\epsilon \succeq \Psi_\epsilon$ in the last line.  This finishes the proof of \eqref{eq:lower-bdJ}.  
\end{proof}

\subsection{Proof of Lemma \ref{lemma:error_bound_tilde_J}}\label{proof:lemma:error_bound_tilde_J}
\begin{proof}
By Proposition \ref{prop:cost_form}, we have
\#\label{eq:312diff1}
\bigl|J_1(\tilde  K_{n+1}) - J_1( K_{n+1})\bigr| =  \tr\bigl[ (P_{\tilde K_{n+1}} - P_{ K_{n+1}} ) \Psi_\epsilon \bigr]\leq \|P_{\tilde K_{n+1}} - P_{ K_{n+1}}\|_2\cdot \|\Psi_\epsilon\|_\F. 
\#
The following lemma upper bounds the term $\|P_{\tilde K_{n+1}} - P_{ K_{n+1}}\|_2$.  

\begin{lemma}\label{lemma:perturb}
Suppose that the parameters $K$ and $\tilde K$ satisfy that
\#\label{eq:perturbK}
\| \tilde K - K \|_2\cdot \bigl( \|A-BK\|_2 +1 \bigr) \cdot \| \Phi_K \|_2  \leq \sigma_{\min}(\Psi_\omega)/4\cdot \|B\|_2^{-1}, 
\#
then it holds that
\#\label{eq:perturbPK}
\| P_{\tilde K} - P_K \|_2 & \leq 6\cdot \sigma_{\min}^{-1} (\Psi_\omega)\cdot \|\Phi_K\|_2 \cdot \|K\|_2\cdot \|R\|_2\cdot \|\tilde   K- K\|_2   \\
&\qquad \cdot \bigl( \|B\|_2\cdot \|K\|_2)\cdot \|A-BK\|_2 + \|B\|_2\cdot\|K\|_2 + 1 \bigr). \notag
\#
\end{lemma}
\begin{proof}
See Lemma 5.7 in \cite{yyy2019aclqr} for a detailed proof. 
\end{proof}

To use  Lemma \ref{lemma:perturb}, it suffices to verify that  $\tilde K_{n+1}$ and $K_{n+1}$ satisfy \eqref{eq:perturbK}. Note that from the definitions of $K_{n+1}$ and $\tilde K_{n+1}$  in  \eqref{eq:ac_update_Kn} and  \eqref{eq:tilde_update_Kn}, respectively, we have
\#\label{eq:veridiff1}
& \| \tilde K_{n+1} - K_{n+1} \|_2 \cdot \bigl( \|A-B\tilde K_{n+1}\|_2 +1 \bigr) \cdot \| \Phi_{\tilde K_{n+1}} \|_2  \notag\\
&\qquad\leq \gamma\cdot \|\hat \Upsilon_{K_n} - \Upsilon_{K_n}\|_\F \cdot \bigl( 1 + \| K_n \|_2 \bigr)\cdot \bigl( \|A-B\tilde K_{n+1}\|_2 +1 \bigr) \cdot \| \Phi_{\tilde K_{n+1}} \|_2. 
\#
Now, we upper bound  the RHS of \eqref{eq:veridiff1}.  For the term $\|A-B\tilde K_{n+1}\|_2$, it holds by the definition of $\tilde K_{n+1}$ in \eqref{eq:tilde_update_Kn} that
\#\label{eq:veribd1}
\|A-B\tilde K_{n+1}\|_2& \leq \|A-BK_n\|_2 + \gamma\cdot \| B \|_2\cdot \|\Upsilon_{K_n}^{22}K_n - \Upsilon_{K_n}^{21}\|_2\notag\\
& \leq \|A-BK_n\|_2 + \gamma\cdot \| B \|_2\cdot \|\Upsilon_{K_n}\|_2\cdot \bigl(1 + \|K_n\|_2 \bigr).
\#
By the definition of $\Upsilon_{K_n}$ in \eqref{eq:def_upsilon}, we upper bound $\|\Upsilon_{K_n}\|_2$ as 
\#\label{eq:veribd2}
\|\Upsilon_{K_n}\|_2 & \leq \|Q\|_2 + \|R\|_2 + \bigl( \|A\|_\F + \|B\|_\F \bigr)^2\cdot \|P_{K_n}\|_2\notag\\
&\leq \|Q\|_2 + \|R\|_2 + \bigl( \|A\|_\F + \|B\|_\F \bigr)^2\cdot J_1(K_0) \cdot \sigma_{\min}^{-1}(\Psi_\epsilon), 
\#
where the last line comes from the fact that 
\$
J_1(K_0) \geq J_1(K_n) = \tr\bigl[(Q + K_n^\top R K_n)\Phi_{K_n}\bigr]  = \tr(P_{K_n} \Psi_\epsilon) \geq \|P_{K_n}\|_2 \cdot \sigma_{\min}(\Psi_\epsilon). 
\$
As for the term $ \| \Phi_{\tilde K_{n+1}} \|_2$ in \eqref{eq:veridiff1}, from the fact that
\$
J_1(K_0) \geq J_1(\tilde K_{n+1}) = \tr\bigl[(Q + \tilde K_{n+1}^\top R \tilde K_{n+1})\Phi_{\tilde K_{n+1}}\bigr] \geq \|\Phi_{\tilde K_{n+1}}\|_2 \cdot \sigma_{\min}(Q),
\$
it holds that
\#\label{eq:veribd3}
 \| \Phi_{\tilde K_{n+1}} \|_2  \leq J_1(K_0) \cdot \sigma_{\min}^{-1}(Q). 
\#
Therefore, combining \eqref{eq:veridiff1}, \eqref{eq:veribd1}, \eqref{eq:veribd2}, and \eqref{eq:veribd3}, we know that 
\#\label{eq:tcup1}
& \| \tilde K_{n+1} - K_{n+1} \|_2 \cdot \bigl( \|A-B\tilde K_{n+1}\|_2 +1 \bigr) \cdot \| \Phi_{\tilde K_{n+1}} \|_2\notag\\
&\qquad  \leq \poly_1\bigl( \|K_n\|_2 \bigr) \cdot  \|\hat \Upsilon_{K_n} - \Upsilon_{K_n}\|_\F. 
\# 
From Theorem \ref{thm:pe}, it holds with probability at least $1-T_n^{-4} - \tilde T_n^{-6}$ that
\#\label{eq:iphone2}
\|\hat\Upsilon_{K_n} - \Upsilon_{K_n}\|_\F &  \leq  \frac{\poly_3 \bigl(  \|K_n\|_\F, \|\mu\|_2  \bigr) }{ \lambda_{K_n}\cdot (1-\rho)^{2} }\cdot\frac{\log^3 T_n}{T_n^{1/4}} \\
&\qquad + \frac{\poly_4 \bigl( \|K_n\|_\F, \|b_{0}\|_2, \|\mu\|_2 \bigr)}{\lambda_{K_n}}\cdot \frac{\log^{1/2} \tilde T_n}{\tilde T_n^{1/8}\cdot (1-\rho)}, \notag
\#
which holds for any $\rho\in( \rho(A-BK_n), 1 )$.  Note that from the choice of $T_n$ and $\tilde T_n$ in the statement of Theorem \ref{thm:ac} that
\$
& T_n\geq \poly_5 \bigl( \|K_n\|_\F, \|b_{0}\|_2, \|\mu\|_2 \bigr) \cdot \lambda_{K_n}^{-4} \cdot \bigl[1-\rho(A-BK_n)\bigr]^{-9}\cdot \varepsilon^{-5},\\
& \tilde T_n \geq \poly_6 \bigl( \|K_n\|_\F, \|b_{0}\|_2, \|\mu\|_2 \bigr) \cdot \lambda_{K_n}^{-2}\cdot \bigl[1-\rho(A-BK_n)\bigr]^{-12}\cdot \varepsilon^{-12}, 
\$
it holds that
\#\label{eq:iphone3}
& \frac{\poly_3 \bigl(  \|K_n\|_\F, \|\mu\|_2  \bigr) }{ \lambda_{K_n}\cdot (1-\rho)^{2} }\cdot\frac{\log^3 T_n}{T_n^{1/4}}  + \frac{\poly_4 \bigl( \|K_n\|_\F, \|b_{0}\|_2, \|\mu\|_2 \bigr)}{\lambda_{K_n}}\cdot \frac{\log^{1/2} \tilde T_n}{\tilde T_n^{1/8}\cdot (1-\rho)}\notag\\
&\qquad \leq \min\biggl\{  \Bigl[\poly_1\bigl( \|K_n\|_2\bigr)\Bigr]^{-1}\cdot \sigma_{\min}(\Psi_\omega)/4\cdot\|B\|_2^{-1},  \\
&\qquad\qquad\qquad  \Bigl[\poly_2\bigl(\|K_n\|_2\bigr)\Bigr]^{-1}\cdot \varepsilon/8 \cdot \gamma\cdot \sigma_{\min}(\Psi_\epsilon)\cdot \sigma_{\min}(R)\cdot \|\Phi_{K^*}\|_2^{-1}\cdot \|\Psi_\epsilon\|_\F^{-1}   \biggr\}.\notag
\#
Combining \eqref{eq:tcup1}, \eqref{eq:iphone2}, and \eqref{eq:iphone3}, we know that  \eqref{eq:perturbK} holds with probability at least $1-\varepsilon^{15}$ for sufficiently small $\varepsilon > 0$.   Meanwhile, by \eqref{eq:veribd1}, \eqref{eq:veribd2}, and \eqref{eq:veribd3}, the RHS of \eqref{eq:perturbPK} is upper bounded as
\#\label{eq:iphone1}
& 6\cdot \sigma_{\min}^{-1} (\Psi_\omega)\cdot \|\Phi_{\tilde K_{n+1}}\|_2 \cdot \|\tilde K_{n+1}\|_2\cdot \|R\|_2\cdot \|\tilde K_{n+1}-  K_{n+1}\|_2  \notag\\
&\qquad\qquad\cdot \bigl( \|B\|_2\cdot \|\tilde K_{n+1}\|_2)\cdot \|A-B\tilde K_{n+1}\|_2 + \|B\|_2\cdot\|\tilde K_{n+1}\|_2 + 1 \bigr)\notag\\
&\qquad\leq \poly_2\bigl(\|K_n\|_2\bigr)\cdot  \|\hat \Upsilon_{K_n} - \Upsilon_{K_n}\|_\F. 
\#
 Now, by Lemma \ref{lemma:perturb}, it holds with probability at least $1 - \varepsilon^{15}$ that
 \#\label{eq:iphonese}
 \| P_{\tilde K_{n+1}} - P_{K_{n+1}} \|_2 & \leq 6\cdot \sigma_{\min}^{-1} (\Psi_\omega)\cdot \|\Phi_{\tilde K_{n+1}}\|_2 \cdot \|\tilde K_{n+1}\|_2\cdot \|R\|_2\cdot \|\tilde K_{n+1}-  K_{n+1}\|_2  \notag\\
&\qquad\cdot \bigl( \|B\|_2\cdot \|\tilde K_{n+1}\|_2)\cdot \|A-B\tilde K_{n+1}\|_2 + \|B\|_2\cdot\|\tilde K_{n+1}\|_2 + 1 \bigr)\notag\\
&\leq \poly_2\bigl(\|K_n\|_2\bigr)\cdot  \|\hat \Upsilon_{K_n} - \Upsilon_{K_n}\|_\F \notag\\
& \leq \varepsilon/8 \cdot \gamma\cdot \sigma_{\min}(\Psi_\epsilon)\cdot \sigma_{\min}(R)\cdot \|\Phi_{K^*}\|_2^{-1}\cdot \|\Psi_\epsilon\|_\F^{-1},
 \#
 where the second inequality comes from \eqref{eq:iphone1}, and the last inequality comes from \eqref{eq:iphone2} and \eqref{eq:iphone3}.  Combining \eqref{eq:312diff1} and \eqref{eq:iphonese}, it holds with probability at least $1-\varepsilon^{15}$ that
\$
\bigl|J_1(\tilde  K_{n+1}) - J_1( K_{n+1})\bigr| \leq \gamma\cdot \sigma_{\min}(\Psi_\epsilon)\cdot \sigma_{\min}(R)\cdot \|\Phi_{K^*}\|_2^{-1}\cdot  \varepsilon / 4,
\$
which concludes the proof of the lemma. 
\end{proof}

\subsection{Proof of Lemma \ref{lemma:local_sc}}\label{proof:lemma:local_sc}
\begin{proof}
Note that $\Upsilon_{K^*}^{22}K^* - \Upsilon_{K^*}^{21}$ is the natural gradient of $J_1$ at the minimizer $K^*$, which implies that 
\#\label{eq:grad=0}
\Upsilon_{K^*}^{22}K^* - \Upsilon_{K^*}^{21} = 0. 
\#
By Lemma \ref{lemma:cost_diff}, it holds that
\#\label{eq:grad-bd1}
J_1(K) - J_1(K^*)& = \tr( P_K\Psi_\epsilon - P_{K^*}\Psi_\epsilon ) = \EE_{y\sim \mathcal N(0, \Psi_\epsilon)} (y^\top P_K y - y^\top P_{K^*} y )\notag\\
& = \EE\biggl\{  \sum_{t\geq 0}  \Bigl[  2y_t^\top (K-K^*)(\Upsilon_{K^*}^{22}K^* - \Upsilon_{K^*}^{21})y_t + y_t^\top (K-K^*)^\top \Upsilon_{K^*}^{22}(K-K^*)y_t  \Bigr]  \biggr\}\notag\\
& = \EE\biggl\{  \sum_{t\geq 0}   y_t^\top (K-K^*)^\top \Upsilon_{K^*}^{21}(K-K^*)y_t    \biggr\},
\#
where we use \eqref{eq:grad=0} in the last line. Here the expectations are taken following the transition $y_{t+1} = (A-BK)y_t$ with initial state $y_0\sim\mathcal N(0, \Psi_\epsilon)$. Also, we have
\#\label{eq:grad-bd2}
& \EE\biggl\{  \sum_{t\geq 0}   y_t^\top (K-K^*)^\top \Upsilon_{K^*}^{22}(K-K^*)y_t    \biggr\} \notag\\
&\qquad  = \tr\bigl[ \Phi_{K} (K-K^*)^\top\Upsilon_{K^*}^{22}(K-K^*) \bigr]\notag\\
&\qquad \geq \|\Phi_K\|_2\cdot \| \Upsilon_{K^*}^{22} \|_2 \cdot \tr\bigl[ (K-K^*)^\top(K-K^*) \bigr]\notag\\
&\qquad \geq \sigma_{\min}(\Psi_\epsilon)\cdot \sigma_{\min}(R) \cdot \|K-K^*\|_\F^2,
\#
where we use the fact that $\Phi_K = (A-BK)\Phi_K (A-BK) + \Psi_\epsilon \succeq \Psi_\epsilon$ and $\Upsilon_{K^*}^{22} = R +  B^\top P_{K^*} B\succeq R$ in the last line.  Combining \eqref{eq:grad-bd1} and \eqref{eq:grad-bd2}, we have
\$
J_1(K) - J_1(K^*) \geq \sigma_{\min}(\Psi_\epsilon)\cdot \sigma_{\min}(R) \cdot \|K-K^*\|_\F^2.  
\$
We conclude the proof of the lemma.  
\end{proof}

\subsection{Proof of Lemma \ref{lemma:bound_tilde_J_2}}\label{proof:lemma:bound_tilde_J_2}
\begin{proof}
Following from Proposition \ref{prop:convex_J2}, we have
\#\label{eq:brute_force_bound0}
& J_2(K_N, b_{h+1}) - J_2(K_N, \tilde b_{h+1})\notag\\
 &\qquad \leq \gamma^b\cdot  \nabla_b J_2(K_N, \tilde b_{h+1})^\top \bigl[  \nabla_b J_2(K_N, b_h) - \hat\nabla_b J_2(K_N, b_h)  \bigr ]\notag\\
&\qquad\qquad + (\gamma^b)^2\cdot {\nu_{K_N}}/{2}\cdot \bigl \|\nabla_b J_2(K_N, b_h) - \hat\nabla_b J_2(K_N, b_h)\bigr\|_2^2,\notag\\
&  J_2(K_N, \tilde b_{h+1}) - J_2( K_N, b_{h+1}) \notag\\
&\qquad \leq - \gamma^b\cdot  \nabla_b J_2(K_N, \tilde b_{h+1})^\top \bigl[  \nabla_b J_2(K_N, b_h) - \hat\nabla_b J_2(K_N, b_h)  \bigr ]\notag\\
 &\qquad\qquad - (\gamma^b)^2\cdot{\iota_{K_N}}/{2}\cdot \bigl\|\nabla_b J_2(K_N, b_h) - \hat\nabla_b J_2(K_N, b_h)\bigr\|_2^2,
\#
where $\nu_{K_N}$ and $\iota_{K_N}$ are defined in Proposition \ref{prop:convex_J2}. Also, following from Proposition \ref{prop:pg}, it holds that 
\#\label{eq:brute_force_bound}
\bigl\| \nabla_b J_2(K_N, \tilde b_{h+1})\bigr\|_2 \leq \poly_1\bigl( \|K_N\|_\F, \|b_h\|_2,  \|\mu\|_2, J(K_N, b_0) \bigr)\cdot \bigl[1 -  \rho(A-B K_N) \bigr]^{-1}. 
\#
Combining \eqref{eq:brute_force_bound0}, \eqref{eq:brute_force_bound}, and the fact that $\nu_{K_N}\leq \iota_{K_N}\leq [1-\rho(A-B K_N)]^{-2}\cdot \poly_2(\|K_N\|_2)$, we know that
\#\label{eq:bynb1}
&\bigl|J_2(K_N, b_{h+1}) - J_2(K_N, \tilde b_{h+1})\bigr|\\
 &\qquad  \leq  (\gamma^b)^2 \cdot  \poly_2\bigl(\|K_N\|_2\bigr)  \cdot \bigl \|\nabla_b J_2(K_N, b_h) - \hat\nabla_b J_2(K_N, b_h)\bigr\|_2^2\cdot \bigl[1 -  \rho(A-B K_N) \bigr]^{-2} \notag\\
&\qquad\qquad  + \gamma^b  \cdot \poly_1\bigl( \|K_N\|_\F, \|b_h\|_2,  \|\mu\|_2, J(K_N, b_0) \bigr) \cdot \bigl\|  \nabla_b J_2(K_N, b_h) - \hat\nabla_b J_2(K_N, b_h)  \bigr \|_2 \notag\\
& \qquad\qquad \cdot   \bigl[1 -  \rho(A-B K_N) \bigr]^{-1}\notag. 
\#
Note that from the definition of $\hat \nabla_b J_2(K_N, b_h)$ and $\nabla_b J_2(K_N, b_h)$ in  \eqref{eq:def-grad-2-hat} and \eqref{eq:def-grad-2}, respectively, it holds by triangle inequality that
\$
& \bigl \|  \nabla_b J_2(K_N, b_h)  - \hat \nabla_b J_2(K_N, b_h)  \bigr\|_2\notag\\
  &\qquad \leq \| \hat \Upsilon_{K_N}^{22} - \Upsilon_{K_N}^{22} \|_2\cdot \|K_N\|_2\cdot \|\hat\mu_{K_N, b_h}\|_2 + \|\Upsilon_{K_N}^{22} \|_2\cdot \|K_N\|_2\cdot \|\hat\mu_{K_N, b_h} - \mu_{K_N, b_h}\|_2\notag\\
  &\qquad\qquad + \|\hat\Upsilon_{K_N}^{22} - \Upsilon_{K_N}^{22}\|_2\cdot \|b_h\|_2  + \|\hat\Upsilon_{K_N}^{21} - \Upsilon_{K_N}^{21}\|_2\cdot \|\hat\mu_{{K_N}, b_h} \|_2 + \|\Upsilon_{K_N}^{21}\|_2\cdot \|\hat\mu_{{K_N}, b_h} - \mu_{{K_N}, b_h}  \|_2 \notag\\
  & \qquad\qquad + \|\hat q_{{K_N}, b_h} - q_{{K_N}, b_h}\|_2.
\$
By Theorem \ref{thm:pe}, combining the fact that $J_2({K_N}, b_h) \leq J_2({K_N}, b_0)$ and the fact that $\|\mu_{{K_N},b}\|_2\leq J({K_N}, b_0)/\sigma_{\min}(Q)$, we know that with probability at least $1 - (T^b_n)^{-4} - (\tilde T^b_n)^{-6}$, it holds for any $\rho \in ( \rho(A-BK_N), 1  ) $   that
\#\label{eq:bynb2}
& \bigl \|  \nabla_b J_2({K_N}, b_h)  - \hat \nabla_b J_2({K_N}, b_h)  \bigr\|_2\\
&\qquad \leq  \lambda_{K_N}^{-1}  \cdot  \poly_3\bigl( \|{K_N}\|_\F, \|b_h\|_2, \|\mu\|_2, J_2({K_N}, b_0) \bigr)\cdot \biggl[ \frac{\log^3 T_n^b}{(T_n^b)^{1/4}(1-\rho)^{2}} +  \frac{\log^{1/2} \tilde T_n^b}{(\tilde T_n^b)^{1/8}\cdot (1-\rho)}  \biggr]. \notag
\#
Following from the choices of $\gamma^b$, $T_n^b$, and $\tilde T_n^b$ in the statement of Theorem  \ref{thm:ac}, it holds that
\$
& \gamma^b\cdot \poly_1\bigl( \|{K_N}\|_\F, \|b_h\|_2,  \|\mu\|_2, J({K_N}, b_0) \bigr)\cdot \lambda_{K_N}^{-1}\cdot \poly_3\bigl( \|{K_N}\|_\F, \|b_h\|_2, \|\mu\|_2, J_2({K_N}, b_0) \bigr)\\
&\qquad   \cdot  \biggl[ \frac{\log^3 T_n^b}{(T_n^b)^{1/4}(1-\rho)^{2}} +  \frac{\log^{1/2} \tilde T_n^b}{(\tilde T_n^b)^{1/8}\cdot (1-\rho)}  \biggr]\cdot \bigl[1 -  \rho(A-B{K_N}) \bigr]^{-1} + \bigl[1 -  \rho(A-B {K_N}) \bigr]^{-2} \\
&\qquad\cdot \poly_3\bigl( \|{K_N}\|_\F, \|b_h\|_2, \|\mu\|_2, J_2({K_N}, b_0) \bigr) \cdot  \biggl[ \frac{\log^6 T_n^b}{(T_n^b)^{1/2}(1-\rho)^{4}} +  \frac{\log \tilde T_n^b}{(\tilde T_n^b)^{1/4}\cdot (1-\rho)^{2}}  \biggr]\notag\\
&\qquad\cdot (\gamma^b)^2\cdot \poly_2\bigl( \|{K_N}\|_2 \bigr)  \cdot \lambda_{K_N}^{-1}\\
& \quad \leq \nu_{K_N}\cdot \gamma^b\cdot \varepsilon/2.
\$
Further combining \eqref{eq:bynb1} and \eqref{eq:bynb2}, it holds with probability at least $1-\varepsilon^{15}$ that 
\$
\bigl|J_2(K_N, b_{h+1}) - J_2(K_N, \tilde b_{h+1})\bigr|  \leq  \nu_{K_N}\cdot\gamma^b\cdot \varepsilon/2. 
\$
We then finish the proof of the lemma. 
\end{proof}

\subsection{Proof of Lemma \ref{lemma:zeta_xi}}\label{proof:lemma:zeta_xi}
\begin{proof}
We show that $\zeta_{K,b}\in\cV_\zeta$ and $\xi(\zeta)\in\cV_\xi$ for any $\zeta\in \cV_\zeta$ separately. 

\vskip5pt
\noindent\textbf{Part 1.} First we show that $\zeta_{K,b}\in\cV_\zeta$. 
Note that from Definition \ref{assum:proj}, we know that $\zeta_{K,b}^1 = J(K,b)$ satisfies that $0\leq \zeta_{K,b}^1 \leq J(K_0, b_0)$.  It remains to show that $\zeta_{K,b}^2 = \alpha_{K,b}$ satisfies that $\|\zeta_{K,b}^2\|_2\leq M_\zeta$. By the definition of $\alpha_{K,b}$ in \eqref{eq:q1}, we know that
\#\label{eq:alpha_bound}
\|\alpha_{K,b}\|_2^2 &\leq \|\Upsilon_K\|_\F^2 + \|\Upsilon_K\|_2^2 \cdot \bigl( \|\mu_{K,b}\|_2^2 + \|\mu_{K,b}^u\|_2^2 \bigr) \notag\\
&\qquad + \bigl(\|A\|_2 + \|B\|_2\bigr)^2\cdot \bigl(\|P_K\|_2 \cdot \|\overline A \mu + d\|_2 + \|f_{K, b}\|_2\bigr)^2
\#
where  $f_{K,b} = (I-A+BK)^{-\top} [  (A-BK)^\top P_K (Bb+\overline A\mu + d) - K^\top Rb ]$ and  for notational simplicity, we denote by $\mu_{K,b}^u = -K\mu_{K,b} + b$. We only need to bound $\Upsilon_K$, $\mu_{K,b}$, $\mu_{K,b}^u$, $P_K$, and $f_{K,b}$.  Note that by Proposition \ref{prop:cost_form}, the expected total cost $J(K, b)$ takes the form of
\$
J(K, b) = \tr(P_K\Psi_\epsilon) + \mu_{K,b}^\top Q\mu_{K,b} + (\mu_{K,b}^u)^\top R\mu_{K,b}^u + \sigma^2\cdot \tr(R) + \mu^\top \overline Q\mu. 
\$
Thus, we have
\$
& J(K_0, b_0)\geq J(K, b)\geq \sigma_{\min}(\Psi_\omega)\cdot \tr(P_K) \geq \sigma_{\min}(\Psi_\omega)\cdot \|P_K\|_2, \\
& J(K_0, b_0)\geq J(K, b)\geq \mu_{K,b}^\top Q\mu_{K,b}\geq \sigma_{\min}(Q)\cdot \|\mu_{K,b}\|_2, \\
& J(K_0, b_0)\geq J(K, b)\geq (\mu_{K,b}^u)^\top R\mu_{K,b}^u\geq \sigma_{\min}(R)\cdot \|\mu_{K,b}^u\|_2, 
\$
which imply that 
\#\label{eq:bound_mu_P}
& \|P_K\|_2\leq J(K_0, b_0)/\sigma_{\min}(\Psi_\omega), \notag\\
& \|\mu_{K,b}\|_2\leq J(K_0, b_0)/\sigma_{\min}(Q),\notag\\
& \|\mu_{K,b}^u\|_2\leq J(K_0, b_0)/\sigma_{\min}(R). 
\#
For $\Upsilon_K$, it holds that 
\$
\Upsilon_K = \begin{pmatrix}
Q & \mathbf{0} \\
\mathbf{0} & R
\end{pmatrix}  +  \begin{pmatrix}
A^\top\\
B^\top
\end{pmatrix}P_K    \begin{pmatrix}
A & B
\end{pmatrix},
\$
which gives
\$
& \|\Upsilon_K\|_\F \leq (\|Q\|_\F + \|R\|_\F) + \bigl( \|A\|_\F^2 + \|B\|_\F^2 \bigr)\cdot \| P_K \|_\F, \notag\\
& \|\Upsilon_K\|_2 \leq (\|Q\|_2 + \|R\|_2) + \bigl( \|A\|_2 + \|B\|_2 \bigr)^2 \cdot \| P_K \|_2. 
\$
Combining \eqref{eq:bound_mu_P} and the fact that $\|P_K\|_\F\leq \sqrt{m}\cdot \|P_K\|_2$, we know that 
\#\label{eq:bound_upsilon}
& \|\Upsilon_K\|_\F \leq \bigl(\|Q\|_\F + \|R\|_\F \bigr) + \bigl( \|A\|_\F^2 + \|B\|_\F^2 \bigr)\cdot \sqrt{m}\cdot J(K_0, b_0)/\sigma_{\min}(\Psi_\omega), \notag\\
& \|\Upsilon_K\|_2 \leq \bigl(\|Q\|_2 + \|R\|_2 \bigr) + \bigl( \|A\|_2 + \|B\|_2 \bigr)^2 \cdot J(K_0, b_0)/\sigma_{\min}(\Psi_\omega). 
\#
Now, we upper bound the vector $f_{K,b}$. 
Note that by algebra, the vector $f_{K,b}$ takes the form of 
\$
f_{K,b} = -P_K\mu_{K,b} + (I-A+BK)^{-T} ( Q\mu_{K,b} - K^\top R\mu_{K,b}^u ).
\$ 
Therefore, we upper bound $f_{K,b}$ as
\#\label{eq:bound_fKb}
\|f_{K,b}\|_2\leq  J(K_0, b_0)^2\cdot \sigma_{\min}^{-1}(\Psi_\omega)\cdot \sigma_{\min}^{-1}(Q) + \bigl[ 1-\rho(A-BK) \bigr]^{-1}\cdot(\kappa_Q + \kappa_R\cdot \|K\|_\F)
\#
Combining \eqref{eq:alpha_bound}, \eqref{eq:bound_mu_P}, \eqref{eq:bound_upsilon}, and \eqref{eq:bound_fKb}, it holds that 
\$
\|\zeta_{K,b}^2\|_2 = \|\alpha_{K,b}\|_2 \leq M_{\zeta,1} + M_{\zeta,2}\cdot(1+ \|K\|_\F)\cdot [ 1-\rho(A-BK) ]^{-1}.
\$
Therefore, it holds that $\zeta_{K,b}\in\cV_\zeta$. 

\vskip5pt
\noindent\textbf{Part 2.} Now we show that for any $\zeta\in\cV_\zeta$, we have $\xi(\zeta)\in\cV_\xi$. Recall that from \eqref{eq:expl_xi-thm}, it holds that
\#\label{eq:expl_xi}
& \xi^1(\zeta) = \zeta^1 - J(K, b), \quad \xi^2(\zeta) = \EE_{\pi_{K,b}}\bigl[\psi(x,u)\bigr] \zeta^1+ \Theta_{K,b} \zeta^2  - \EE_{\pi_{K,b}}\bigl[ c(x,u) \psi(x,u) \bigr]. 
\#
Then we have 
\#\label{eq:xi1}
\bigl|\xi^1(\zeta)\bigr| = \bigl|\zeta^1 - J(K, b)\bigr|\leq J(K_0, b_0),
\#
where we use the fact that since $\zeta\in \cV_\zeta$, we have $0 \leq \zeta^1 \leq J(K_0, b_0)$ by Definition \ref{assum:proj}. Also, by \eqref{eq:expl_xi}, we have 
\#\label{eq:xi2_1}
\bigl\|\xi^2(\zeta)  \bigr\|_2\leq \underbrace{\Bigl\|  \EE_{\pi_{K,b}}\bigl[\psi(x,u)\bigr] \zeta^1 \Bigr\|_2}_{B_1} + \underbrace{\|\Theta_{K,b}\|_2\cdot \|\zeta^2\|_2}_{B_2} + \underbrace{\Bigl\|\EE_{\pi_{K,b}}\bigl[ c(x,u) \psi(x,u) \bigr]\Bigr\|_2}_{B_3}. 
\#
Note that we upper bound $B_1$ as
\#\label{eq:bound_B11}
B_1 \leq J(K_0, b_0)\cdot  \Bigl\|  \EE_{\pi_{K,b}}\bigl[\psi(x,u)\bigr]  \Bigr\|_2.  
\#
Following from the definition of $\psi(x,u)$ in \eqref{eq:def_feature}, we know that 
\#\label{eq:laji1}
\Bigl\|  \EE_{\pi_{K,b}}\bigl[\psi(x,u)\bigr]  \Bigr\|_2 \leq \| \Sigma_z \|_\F,
\#
where $\Sigma_z$ is defined as 
\$
\Sigma_z = {\rm Cov} \Biggl[ \begin{pmatrix}
x\\
u
\end{pmatrix}\Biggr]
 = \begin{pmatrix}
\Phi_K & -\Phi_K K^\top\\
-K\Phi_K & K\Phi_K K^\top + \sigma^2  I
\end{pmatrix} = \begin{pmatrix}
0 & 0\\
0 & \sigma^2 I
\end{pmatrix} + \begin{pmatrix}
I\\
-K
\end{pmatrix}\Phi_K\begin{pmatrix}
I\\
-K
\end{pmatrix}^\top. 
\$
Combining \eqref{eq:bound_B11} and \eqref{eq:laji1}, we have
\#\label{eq:bound_B1}
B_1  \leq J(K_0, b_0)\cdot \|\Sigma_z\|_\F. 
\#
By Proposition \ref{prop:invert_theta}, we upper bound $B_2$ as
\#\label{eq:bound_B2}
B_2 \leq 4 ( 1 + \|K\|_\F^2 )^3\cdot \|\Phi_K\|_2^2 \cdot (M_{\zeta,1} + M_{\zeta,2})\cdot \bigl[ 1-\rho(A-BK) \bigr]^{-1},
\#
where we use the fact that $\zeta\in\cV_\zeta$ and Definition \ref{assum:proj}.  As for the term $B_3$ in \eqref{eq:xi2_1},  we utilize the following lemma to provide an upper bound. 

\begin{lemma}\label{lemma:form_c_psi}
The vector $\EE_{\pi_{K, b}}[c(x,u)\psi(x,u)]$ takes the following form,
\$
\EE_{\pi_{K, b}} \bigl[ c(x,u)\psi(x,u)\bigr ]& = \begin{pmatrix}
2\svec\bigl[ \Sigma_z \diag(Q,R)\Sigma_z + \la \Sigma_z, \diag(Q,R)\ra \Sigma_z \bigr]\\
\Sigma_z\begin{pmatrix}
2Q\mu_{K,b}\\
2R\mu_{K,b}^u
\end{pmatrix}
\end{pmatrix}\notag\\
&\qquad + \bigl[\mu_{K,b}^\top Q\mu_{K,b} + (\mu_{K,b}^u)^\top R\mu_{K,b}^u + \mu^\top \overline Q\mu\bigr]\begin{pmatrix}
\svec(\Sigma_z)\\
\mathbf{0}_{m}\\
\mathbf{0}_{k}
\end{pmatrix}.
\$
Here the matrix $\Sigma_z$ takes the form of
\$
\Sigma_z = \begin{pmatrix}
\Phi_K & -\Phi_K K^\top\\
-K\Phi_K & K\Phi_K K^\top + \sigma^2 \cdot I
\end{pmatrix}. 
\$
\end{lemma}
\begin{proof}
See \S\ref{proof:lemma:form_c_psi} for a detailed proof. 
\end{proof}

From Lemma \ref{lemma:form_c_psi} and \eqref{eq:bound_mu_P}, it holds that
\#\label{eq:bound_B3}
B_3& \leq 3 \bigl[ \|Q\|_\F + \|R\|_\F  + J(K_0, b_0)\cdot \|Q\|_2 / \sigma_{\min}(Q) \\
& \qquad + J(K_0, b_0)\cdot \|R\|_2 / \sigma_{\min}(R)   \bigr] \cdot \|\Sigma_z\|_2^2. \notag
\#
Moreover, by the definition of $\Sigma_z$ in \eqref{eq:sigma_z_factor}, combining the triangle inequality, we have the following bounds for $\|\Sigma_z\|_\F$ and $\|\Sigma_z\|_2$,
\#\label{eq:sigmaz_bound}
\|\Sigma_z\|_\F \leq 2(d + \|K\|_\F^2)\cdot \|\Phi_K\|_2, \qquad \|\Sigma_z\|_2 \leq 2(1 + \|K\|_\F^2)\cdot \|\Phi_K\|_2. 
\#
Also, we have
\$
J(K_0, b_0) \geq J(K,b)\geq \tr\bigl[ (Q + K^\top RK) \Phi_K \bigr] \geq \|\Phi_K\|_2 \cdot \sigma_{\min}(Q),
\$
which gives the upper bound for $\Phi_K$ as follows,
\#\label{eq:bound_phix}
\|\Phi_K\|_2 \leq J(K_0, b_0) / \sigma_{\min}(Q). 
\#
Therefore, combining \eqref{eq:xi2_1}, \eqref{eq:bound_B1}, \eqref{eq:bound_B2}, \eqref{eq:bound_B3}, \eqref{eq:sigmaz_bound}, and \eqref{eq:bound_phix}, we know that 
\#\label{eq:xi2}
\bigl\|\xi^2(\zeta)  \bigr\|_2 & \leq C\cdot (M_{\zeta,1} + M_{\zeta,2}) \cdot J(K_0, b_0)^2/\sigma^2_{\min}(Q) \\
&\qquad  \cdot \bigl( 1 + \|K\|_\F^2 \bigr)^3\cdot \bigl[ 1-\rho(A-BK) \bigr]^{-1}. \notag
\#
By \eqref{eq:xi1} and \eqref{eq:xi2}, we know that $\xi(\zeta)\in\cV_\xi$ for any $\zeta\in\cV_\zeta$. We conclude the proof of the lemma.   
\end{proof}

\subsection{Proof of Lemma \ref{lemma:dist_hat_muz}}\label{proof:lemma:dist_hat_muz}
\begin{proof}
Assume that $\tilde z_0\sim\mathcal N(\mu_\dagger, \Sigma_\dagger)$. Following from the fact that
\$
\tilde z_{t+1} = L \tilde z_t + \nu + \delta_t,
\$
it holds that 
\#\label{eq:dist_ztmu}
\tilde z_t \sim \mathcal   N\Biggl(   L^t \mu_\dagger +  \sum_{i = 0}^{t-1} L^i\cdot \nu,   ~ (L^\top)^t \Sigma_\dagger L^t + \sum_{i = 0}^{t-1} (L^\top)^{i} \Psi_\delta L^i         \Biggr),
\#
where
\$
\Psi_\delta = \begin{pmatrix}
\Psi_\omega & K \Psi_\omega\\
K\Psi_\omega & K\Psi_\omega K^\top + \sigma^2 I
\end{pmatrix}.   
\$
From \eqref{eq:def-mean-var}, we know that $\mu_z$ takes the form of 
\#\label{eq:muz-sum}
\mu_z = (I-L)^{-1}\nu = \sum_{j = 0}^\infty L^j\nu. 
\#
Therefore, combining \eqref{eq:dist_ztmu} and \eqref{eq:muz-sum}, we have
\#\label{eq:mean_muz1}
\EE(\hat\mu_z) = \mu_z + \frac{1}{\tilde T}\sum_{t = 1}^{\tilde T}L^t \mu_\dagger - \frac{1}{\tilde T}\sum_{t = 1}^{\tilde T}\sum_{i = t}^\infty L^i\nu. 
\#
We denote by
\$
\mu_{\tilde T} = \sum_{t = 1}^{\tilde T}L^t \mu_\dagger - \sum_{t = 1}^{\tilde T}\sum_{i = t}^\infty L^i\nu. 
\$
Meanwhile, it holds that 
\#\label{eq:mean_muz2}
&\biggl\| \sum_{t = 1}^{\tilde T}L^t \mu_\dagger -  \sum_{t = 1}^{\tilde T}\sum_{i = t}^\infty L^i\nu\biggr\|_2\notag\\
&\qquad \leq  \sum_{t = 1}^{\tilde T}\rho(L)^t \cdot \|\mu_\dagger\|_2 +  \sum_{t = 1}^{\tilde T}\sum_{i = t}^\infty \rho(L)^i\cdot \|\nu\|_2\notag\\
& \qquad \leq {\bigl[ 1 - \rho(L) \bigr]^{-1}}\cdot \|\mu_\dagger\|_2 + {\bigl[ 1 - \rho(L) \bigr]^{-2}} \cdot \|\nu\|_2 \notag\\
&\qquad \leq M_\mu\cdot (1-\rho)^{-2}\cdot \|\mu_z\|_2,
\#
where $M_\mu$ is a positive absolute constant.  

For the covariance, note that for any random variables $X\sim\mathcal N(\mu_1, \Sigma_1)$ and $Y\sim\mathcal N(\mu_2, \Sigma_2)$, we know that $Z = X+Y\sim\mathcal N(\mu_1+\mu_2, \Sigma)$, where $\|\Sigma\|_\F\leq 2\|\Sigma_1\|_\F + 2\|\Sigma_2\|_\F$.  Combining \eqref{eq:dist_ztmu}, we know that $\hat\mu_z\sim\mathcal N(\EE\hat\mu_z , \tilde \Sigma_{\tilde T}/\tilde T)$, where $\tilde \Sigma_{\tilde T}$ satisfies that
\$
\tilde T/2\cdot \|\tilde \Sigma_{\tilde T}\|_\F & \leq  \sum_{t = 1}^{\tilde T} \rho(L)^{2t}\cdot \|\Sigma_\dagger\|_\F + \sum_{t = 1}^{\tilde T}\sum_{i = 0}^{t-1}\rho(L)^{2i}\cdot \|\Psi_\delta\|_\F\notag\\
& \leq \bigl[1 - \rho(L)^2\bigr]^{-1}\cdot \|\Sigma_\dagger\|_\F + \tilde T\cdot \bigl[1 - \rho(L)^2\bigr]^{-1}\cdot \|\Psi_\delta\|_\F,
\$
which implies that 
\#\label{eq:var_muz2}
\|\tilde \Sigma_{\tilde T}\|_\F\leq M_\Sigma\cdot (1-\rho)^{-1}\cdot \|\Sigma_z\|_\F,
\#
where $M_\Sigma$ is a positive absolute constant. Combining \eqref{eq:mean_muz1}, \eqref{eq:mean_muz2}, and \eqref{eq:var_muz2}, we obtain that 
\$
\hat \mu_z\sim\mathcal N\biggl(\mu_z + \frac{1}{\tilde T}\mu_{\tilde T}, ~\frac{1}{\tilde T} \tilde \Sigma_{\tilde T} \biggr),
\$
where $\|\mu_{\tilde T}\|_2\leq M_\mu\cdot (1-\rho)^{-2}\cdot \|\mu_z\|_2$ and $\|\tilde \Sigma_{\tilde T}\|_\F\leq M_\Sigma\cdot (1-\rho)^{-1}\cdot \|\Sigma_z\|_\F$.    Moreover, by the Gaussian tail inequality, it holds with probability at least $1 -  \tilde T^{-6}$ that
\$
\| \hat \mu_z - \mu_z \|_2 \leq   \frac{\log \tilde T}{\tilde T^{1/4}}\cdot (1-\rho)^{-2} \cdot \poly \bigl( \|\Phi_K\|_2, \|K\|_\F, \|b\|_2, \|\mu\|_2  \bigr). 
\$
Then we finish the proof of the lemma.  
\end{proof}

\subsection{Proof of Lemma \ref{lemma:bound_tilde_F}}\label{proof:lemma:bound_tilde_F}
\begin{proof}
We continue using the notations given in \S\ref{proof:thm:pe}.   We define
\$
\hat  F(\zeta, \xi) =  \Bigl\{  \EE ( \hat  \psi ) \zeta^1+ \EE \bigl[ (\hat  \psi - \hat  \psi') \hat  \psi^\top \bigr] \zeta^2 - \EE (   c \hat  \psi )\Bigr\}^\top \xi^2  + \bigl[\zeta^1 - \EE (  c)\bigr] \cdot \xi^1 - 1/2\cdot \|\xi\|_2^2,
\$
where $\hat \psi = \hat \psi(x,u)$ is the estimated feature vector. Here the expectation is only taken over the trajectory generated by the state transition and the policy $\pi_{K,b}$, conditioning on the randomness induced when calculating the estimated feature vectors.   Thus, the function $\hat F(\zeta, \xi)$ is still random, where the randomness comes from the estimated feature vectors. 
Note that $|F(\zeta, \xi) - \tilde F(\zeta, \xi)|  \leq |F(\zeta, \xi) - \hat F(\zeta, \xi)|  + |\hat F(\zeta, \xi) - \tilde F(\zeta, \xi)| $.  Thus, we only need to upper bound $|F(\zeta, \xi) - \hat F(\zeta, \xi)|$ and $|\hat F(\zeta, \xi) - \tilde F(\zeta, \xi)|$.
 
\vskip5pt
\noindent\textbf{Part 1.} First we upper bound $|F(\zeta, \xi) - \hat F(\zeta, \xi)|$.  Note that by algebra, we have
\#\label{eq:diff_hat_F}
&  \bigl|F(\zeta, \xi) - \hat F(\zeta, \xi)\bigr| \notag\\
&\qquad = \biggl|  \Bigl\{ \EE(\psi - \hat\psi)\zeta^1 + \EE\bigl[ (\psi - \psi')\psi^\top - (\hat \psi - \hat\psi')\hat\psi^\top \bigr] \zeta^2 - \EE\bigl[ c(\psi - \hat\psi) \bigr]  \Bigr\}^\top \xi^2   \biggr|\notag\\
&\qquad \leq \EE\bigl(\|\psi - \hat\psi\|_2\bigr) \cdot \Big[  |\zeta^1|  +   \EE\bigl(   \| \psi - \psi' \|_2 + 2\|\hat\psi\|_2    \bigr)  \cdot  \| \zeta^2 \|_2   + \EE( c)     \Bigr]\cdot \|\xi^2\|_2,
\#
where the expectation is only taken over the trajectory generated by the state transition and the policy $\pi_{K, b}$.   From Lemma \ref{lemma:dist_hat_muz}, it holds that 
\#\label{eq:pen1}
\PP\bigl(\| \hat\mu_z -\mu_z + 1/\tilde T\cdot \mu_{\tilde T}  \|_2  \leq  C_1\bigr) \geq 1 - \tilde T^{-6}. 
\# 
Therefore,  combining \eqref{eq:pen1}, it holds  with probability at least $1-\tilde T^{-6}$ that 
\#\label{eq:pen2}
\EE \bigl(   \| \psi - \psi' \|_2 + 2\|\hat\psi\|_2 \bigr) \leq \poly\bigl( \|\Phi_K\|_2, \|K\|_\F, \|b\|_2, \|\mu\|_2, J(K_0, b_0) \bigr),
\#
where the expectation is conditioned on the randomness induced when calculating the estimated feature vectors.   Also, we know that
\#\label{eq:pen3}
\EE(c) \leq \poly\bigl( \|\Phi_K\|_2, \|K\|_\F, \|b\|_2, \|\mu\|_2, J(K_0, b_0) \bigr). 
\# 
Therefore, combining \eqref{eq:diff_hat_F},  \eqref{eq:pen2}, \eqref{eq:pen3}, and Definition \ref{assum:proj}, it holds with probability at least $1-\tilde T^{-6}$ that
\#\label{eq:diff_hat_F2}
\bigl|F(\zeta, \xi) - \hat F(\zeta, \xi)\bigr| \leq \EE\bigl(\|\psi - \hat\psi\|_2\bigr)\cdot \poly\bigl( \|\Phi_K\|_2, \|K\|_\F, \|b\|_2, \|\mu\|_2, J(K_0, b_0) \bigr). 
\#
Following from the definitions of $\psi(x,u)$ in \eqref{eq:def_feature} and $\hat\psi(x,u)$ in \eqref{eq:est_def_feature}, we upper bound $\|\psi(x,u) - \hat\psi(x,u)\|_2$ for any $x$ and $u$ as
\#\label{eq:error_psi}
\|\psi(x,u) - \hat\psi(x,u)\|_2^2&  = \|\hat \mu_z - \mu_z\|_2^2 + \bigl\|   z(\hat\mu_z - \mu_z)^\top + (\hat\mu_z - \mu_z)z^\top  \bigr\|_\F^2 + \| \mu_z\mu_z^\top - \hat\mu_z \hat\mu_z^\top \|_\F^2\notag\\
& \leq \poly\bigl( \|\Phi_K\|_2, \|K\|_\F, \|b\|_2, \|\mu\|_2, J(K_0, b_0) \bigr) \cdot \| \hat\mu_z - \mu_z \|_2^2,
\#
where $\mu_z$ is defined in \eqref{eq:def-mean-var}, $\hat\mu_z$ is defined in \eqref{eq:def-mean-est}, and $z = (x^\top, u^\top)^\top$.   Also, by Lemma \ref{lemma:dist_hat_muz}, we know that 
\#\label{eq:error_muz}
\| \hat\mu_z - \mu_z \|_2 \leq \frac{\log \tilde T}{\tilde T^{1/4}}\cdot (1-\rho)^{-2}\cdot \poly\bigl( \|\Phi_K\|_2, \|K\|_\F, \|b\|_2, \|\mu\|_2, J(K_0, b_0) \bigr),
\#
which holds with probability at least $1 - \tilde T^{-6}$.  Combining \eqref{eq:diff_hat_F2}, \eqref{eq:error_psi}, and \eqref{eq:error_muz}, it holds with probability at least $1 - \tilde T^{-6}$ that
\#\label{eq:bound_F_hatF}
\bigl|F(\zeta, \xi) - \hat F(\zeta, \xi)\bigr| \leq   \frac{\log \tilde T}{\tilde T^{1/4}}\cdot (1-\rho)^{-2}\cdot \poly\bigl(  \|K\|_\F, \|b\|_2, \|\mu\|_2, J(K_0, b_0) \bigr).
\# 

\vskip5pt
\noindent\textbf{Part 2.} We now upper bound $|\hat F(\zeta, \xi) - \tilde F(\zeta, \xi)|$ in the sequel.  By definitions, we have
\#\label{eq:diff_tilde_hat_F}
& \bigl|\tilde F(\zeta, \xi) - \hat F(\zeta, \xi)\bigr|\notag\\
&\qquad  =   \biggl|  \Bigl\{ \EE(\tilde \psi - \hat\psi)\zeta^1 + \EE\bigl[ (\tilde \psi - \tilde \psi')\tilde \psi^\top - (\hat \psi -  \hat\psi')\hat\psi^\top \bigr] \zeta^2 - \EE( \tilde c\tilde \psi - \hat c \hat \psi )  \Bigr\}^\top \xi^2  + \EE(\hat c - \tilde c)\xi^1  \biggr|\notag\\
& \qquad \leq  \biggl|  \Bigl\{ \EE(\hat\psi) \zeta^1 + \EE ( \hat \psi \hat\psi^\top )  \zeta^2 - \EE( \hat c \hat \psi )    \Bigr\}^\top \xi^2  + \EE(\hat c) \xi^1  \biggr|\cdot \ind_{\cE^c} \\
&\qquad\qquad+ \Bigl|  \bigl[\EE(\hat \psi'\hat \psi^\top)\zeta^2\bigr]^\top \xi^2  \Bigr|\cdot \ind_{(\cE'\cap\cE)^c},\notag
\#
where we define the event $\cE'$ as 
\$
\cE' = \biggl(\bigcap_{t \in[T]}\Bigl\{   \bigl |  \| z_t' -\mu_z + 1/\tilde T\cdot \mu_{\tilde T}  \|_2^2 - \tr(\tilde \Sigma_z)    \bigr|   \leq    C_1\cdot \log T\cdot \|\tilde \Sigma_z\|_2   \Bigr\} \biggr) \bigcap  \cE_2,
\$
where $\cE_2$ is defined in \eqref{eq:def-ce2nn}. 
Combining the fact that $\PP(\cE_2) \geq 1 - \tilde T^{-6}$ and Lemma \ref{lemma:hwineq}, it holds that $\PP(\cE')\geq 1 - T^{-5} - \tilde T^{-6}$.  Following a similar argument as in \textbf{Part 1}, it holds from \eqref{eq:diff_tilde_hat_F} that
\#\label{eq:bound_tilde_hat_F}
\bigl|\tilde F(\zeta, \xi) - \hat F(\zeta, \xi)\bigr| \leq \biggl(\frac{1}{T} + \frac{1}{\tilde T^{1/4}}\biggr) \cdot \poly\bigl(  \|K\|_\F, \|b\|_2, \|\mu\|_2, J(K_0, b_0) \bigr)
\#
for sufficiently large $T$ and $\tilde T$. 

\vskip5pt
Now, combining \eqref{eq:bound_F_hatF} and \eqref{eq:bound_tilde_hat_F}, by triangle inequality, it holds with probability at least $1-\tilde T^{-6}$ that
\$
\bigl |F(\zeta, \xi) - \tilde F(\zeta, \xi) \bigr|  \leq  \biggl(\frac{1}{2T} + \frac{\log \tilde T}{\tilde T^{1/4}}\biggr)\cdot (1-\rho)^{-2}\cdot \poly\bigl( \|K\|_\F, \|b\|_2, \|\mu\|_2, J(K_0, b_0) \bigr).
\$
We finish the proof of the lemma. 
\end{proof}

\subsection{Proof of Lemma \ref{lemma:theta_form}}\label{proof:lemma:theta_form}
\begin{proof}
Recall that the feature vector $\psi(x,u)$ takes the following form
\$
\psi(x,u) = \begin{pmatrix}
\svec\bigl[ (z-\mu_z)(z-\mu_z)^\top \bigr]\\
z-\mu_z
\end{pmatrix}.
\$
We then have
\#\label{eq:mzm1}
\psi(x,u) - \psi(x',u') = \begin{pmatrix}
\svec\bigl[ yy^\top - (Ly +\delta)(Ly+\delta)^\top \bigr]\\
y - (Ly+\delta)
\end{pmatrix},
\#
where we denote by $y = z-\mu_z$, and $(x', u')$ is the state-action pair after $(x, u)$ following the state transition and the policy $\pi_{K,b}$. Therefore, for any symmetric matrices $M$, $N$ and any vectors $m$, $n$,  it holds from \eqref{eq:q2} and \eqref{eq:mzm1} that
\#\label{eq:quad_deri}
& \begin{pmatrix}
\svec(M)\\
m
\end{pmatrix}^\top 
\Theta_{K,b}
\begin{pmatrix}
\svec(N)\\
n
\end{pmatrix}\notag \\
& \qquad = \EE_{y, \delta}\Biggl\{  \begin{pmatrix}
\svec(M)\\
m
\end{pmatrix}^\top   
\begin{pmatrix}
\svec(yy^\top)\\
y
\end{pmatrix}
\begin{pmatrix}
\svec\bigl[ yy^\top - (Ly +\delta)(Ly +\delta)^\top \bigr]\\
y - (Ly+\delta)
\end{pmatrix}^\top
\begin{pmatrix}
\svec(N)\\
n
\end{pmatrix} \Biggr \} \notag \\
&\qquad = \EE_{y, \delta}\Bigl\{ \bigl( \la M, yy^\top\ra + m^\top y\bigr)\cdot \bigl[  \la N,  yy^\top - (Ly+\delta)(Ly+\delta)^\top\ra  + n^\top (y-Ly-\delta)     \bigr]  \Bigr\}\notag \\
&\qquad = \underbrace{ \EE_{y}\bigl(  \la yy^\top, M\ra\cdot \la yy^\top -Lyy^\top L^\top - \Psi_\delta, N  \ra  \bigr) }_{A_1}  + \underbrace{ \EE_{y}\bigl( \la yy^\top, M\ra\cdot n^\top (y-Ly)  \bigr) }_{A_2}\\
&\qquad\qquad +  \underbrace{  \EE_{y}  \bigl(   m^\top y\cdot \la yy^\top -Lyy^\top L^\top - \Psi_\delta, N\ra   \bigr)  }_{A_3} + \underbrace{  \EE_{y}  \bigl[  m^\top y\cdot n^\top (y-Ly)  \bigr] }_{A_4},\notag
\#
where the expectations are taken over  $y\sim\mathcal N(0,\Sigma_z)$ and $\delta\sim\mathcal N(0,\Psi_\delta)$.  We evaluate the terms $A_1$, $A_2$, $A_3$, and $A_4$ in the sequel. 

For the terms $A_2$ and $A_3$ in \eqref{eq:quad_deri}, by the fact that $y = z-\mu_z\sim\mathcal N (0,\Sigma_z)$, we know that these two terms vanish.   For $A_4$, it holds that
\#\label{eq:cal_a4}
A_4 = \EE_y\bigl[ m^\top y\cdot (y-Ly)^\top n \bigr] = \EE_y\bigl[ m^\top y y^\top (I-L)^\top n \bigr] = m^\top \Sigma_z (I-L)^\top n. 
\#
For $A_1$, by algebra, we have
\#\label{eq:cal_a1_1}
A_1&  =   \EE_{y}\bigl(  \la yy^\top, M\ra\cdot \la yy^\top -Lyy^\top L^\top - \Psi_\delta, N  \ra  \bigr)\notag\\
& = \EE_{y}\bigl(  \la yy^\top, M\ra\cdot \la yy^\top -Lyy^\top L^\top, N  \ra  \bigr) - \EE_{y}\bigl(  \la yy^\top, M\ra\cdot  \la \Psi_\delta, N  \ra  \bigr)\notag\\
& = \EE_{y}\bigl[  y^\top M y \cdot y^\top (N-L^\top N L) y  \bigr] -  \la \Sigma_z, M\ra\cdot  \la \Psi_\delta, N  \ra \notag\\
& = \EE_{u\sim\mathcal N(0,I)}\bigl[  u^\top \Sigma_z^{1/2} M \Sigma_z^{1/2} u \cdot u^\top \Sigma_z^{1/2} (N-L^\top N L) \Sigma_z^{1/2} u  \bigr] -  \la \Sigma_z, M\ra\cdot  \la \Psi_\delta, N  \ra. 
\#  
Now, by applying Lemma \ref{lemma:magnus} to the first term on the RHS of \eqref{eq:cal_a1_1}, we know that
\$
A_1 & = 2\tr\bigl[ \Sigma_z^{1/2} M \Sigma_z^{1/2}\cdot \Sigma_z^{1/2} (N-L^\top N L) \Sigma_z^{1/2} \bigr]\notag\\
&\qquad + \tr(\Sigma_z^{1/2} M \Sigma_z^{1/2})\cdot \tr\bigl[ \Sigma_z^{1/2} (N-L^\top N L) \Sigma_z^{1/2} \bigr] - \la \Sigma_z, M\ra\cdot  \la \Psi_\delta, N  \ra\notag\\
& = 2\bigl\la M, \Sigma_z (N-L^\top N L )\Sigma_z\ra + \la \Sigma_z, M \bigr\ra\cdot \la \Sigma_z - L\Sigma_z L^\top - \Psi_\delta, N\ra  \notag\\
&= 2\bigl\la M, \Sigma_z (N-L^\top N L )\Sigma_z\bigr\ra,
\$
where we use the fact that $\Sigma_z = L\Sigma_z L^\top  + \Psi_\delta$ in the last equality.   By using the property of the operator $\svec(\cdot)$ and the definition of the symmetric Kronecker product, we obtain that
\#\label{eq:cal_a1}
A_1& = 2\svec(M)^\top \svec\bigl[ \Sigma_z (N-L^\top N L )\Sigma_z \bigr]\notag\\
& = 2\svec(M)^\top \bigl[\Sigma_z\otimes_s \Sigma_z - (\Sigma_z L^\top) \otimes_s (\Sigma_z L^\top) \bigr]\svec(N)\notag\\
& = 2\svec(M)^\top \bigl[(\Sigma_z\otimes_s \Sigma_z)(I-L \otimes_s L)^\top \bigr]\svec(N). 
\#
Combining \eqref{eq:quad_deri}, \eqref{eq:cal_a4}, and \eqref{eq:cal_a1}, we obtain that
\$
& \begin{pmatrix}
\svec(M)\\
m
\end{pmatrix}^\top 
\Theta_{K,b}
\begin{pmatrix}
\svec(N)\\
n
\end{pmatrix} \\
&\qquad =  \svec(M)^\top \bigl[2 (\Sigma_z\otimes_s \Sigma_z)(I-L \otimes_s L)^\top \bigr]\svec(N) +   m^\top \Sigma_z (I-L)^\top n\notag\\
& \qquad = \begin{pmatrix}
\svec(M)\\
m
\end{pmatrix}^\top 
\begin{pmatrix}
2 (\Sigma_z\otimes_s \Sigma_z) (I - L \otimes_s L)^\top & 0 \\
0 & \Sigma_z (I-L)^\top
\end{pmatrix}
\begin{pmatrix}
\svec(N)\\
n
\end{pmatrix}. 
\$
Thus, the matrix $\Theta_{K,b}$ takes the following form, 
\$
\Theta_{K,b} = \begin{pmatrix}
2 (\Sigma_z\otimes_s \Sigma_z) (I - L \otimes_s L)^\top & 0 \\
0 & \Sigma_z (I-L)^\top
\end{pmatrix},
\$
which concludes the proof of the lemma. 
\end{proof}

\subsection{Proof of Lemma \ref{lemma:upper_bound_tilde_theta}}\label{proof:lemma:upper_bound_tilde_theta}
\begin{proof}
From the definition of $\tilde \Theta_{K,b}$ in \eqref{eq:def-tilde-theta}, it holds that
\#\label{eq:norm_tilde}
\|\tilde \Theta_{K,b}^{-1} \|_2^2 \leq 1 + \|\Theta_{K,b}^{-1}\|_2^2 + \|\Theta_{K,b}^{-1}\tilde \sigma_z\|_2^2,
\#
where $\tilde \sigma_z$ is defined as
\$
\tilde \sigma_z = \EE_{\pi_{K,b}}\bigl[ \psi(x,u) \bigr] = \begin{pmatrix}
\svec(\Sigma_z)\\
\mathbf 0_{k+m}
\end{pmatrix}.
\$
We bound the RHS of \eqref{eq:norm_tilde} in the sequel. 
For the term $\Theta_{K,b}^{-1}\tilde \sigma_z$, combining Lemma \ref{lemma:theta_form}, we have
\#\label{eq:what1}
\Theta_{K,b}^{-1}\tilde \sigma_z & = \begin{pmatrix}
1/2\cdot  (I-L \otimes_s L) ^{-\top}  (\Sigma_z\otimes_s \Sigma_z)^{-1}\cdot \svec(\Sigma_z)\notag\\
\mathbf 0_{k+m}
\end{pmatrix}\\
& = \begin{pmatrix}
1/2\cdot  (I-L \otimes_s L) ^{-\top}  (\Sigma_z^{-1}\otimes_s \Sigma_z^{-1} ) \cdot  \svec(\Sigma_z)\notag\\
\mathbf 0_{k+m}
\end{pmatrix}\\
& = \begin{pmatrix}
1/2\cdot  (I-L \otimes_s L) ^{-\top} \cdot  \svec(\Sigma_z^{-1})\\
\mathbf 0_{k+m}
\end{pmatrix},
\#
where we use the property of the symmetric Kronecker product in the second and last line.  By taking the spectral norm on both sides of \eqref{eq:what1}, it holds that
\#\label{eq:norm_theta_sigma}
\|\Theta_{K,b}^{-1}\tilde \sigma_z\|_2&  = 1/2\cdot \bigl \| (I-L \otimes_s L) ^{-\top} \cdot  \svec(\Sigma_z^{-1})\bigr \|_2\notag\\
& \leq 1/2\cdot \bigl \| (I-L \otimes_s L) ^{-\top} \bigr \|_2 \cdot  \bigl \| \svec(\Sigma_z^{-1})\bigr \|_2\notag\\
& \leq 1/2\cdot \bigl[1-\rho^2(L)\bigr]^{-1}\cdot \|\Sigma_z^{-1}\|_\F\notag\\
& \leq 1/2\cdot\sqrt{k+m}\cdot \bigl[1-\rho^2(L)\bigr]^{-1}\cdot \|\Sigma_z^{-1}\|_2\notag\\
& = 1/2\cdot\sqrt{k+m}\cdot \bigl[1-\rho^2(L)\bigr]^{-1} \cdot \sigma_{\min}^{-1} (\Sigma_z),
\# 
where in the third line we use Lemma \ref{lemma:alizadeh1998primal} to the matrix $L\otimes_s L$.  Similarly, we upper bound  $\|\Theta_{K,b}^{-1}\|_2$ in the sequel
\#\label{eq:norm_theta}
\|\Theta_{K,b}^{-1}\|_2 \leq \min\Bigl\{  1/2\cdot \bigl[ 1-\rho^2(L) \bigr]^{-1} \sigma_{\min}^{-2} (\Sigma_z), \bigl[ 1-\rho(L) \bigr]^{-1} \sigma_{\min}^{-1} (\Sigma_z) \Bigr\}. 
\#
Thus, combining \eqref{eq:norm_tilde}, \eqref{eq:norm_theta_sigma}, and \eqref{eq:norm_theta}, we obtain that
\#\label{eq:norm_tilde2}
\|\tilde \Theta_{K,b}^{-1} \|_2^2  &\leq 1 + 1/2\cdot\sqrt{k+m}\cdot \bigl[1-\rho^2(L)\bigr]^{-1} \cdot  \sigma_{\min}^{-1} (\Sigma_z)\notag\\
&  \qquad +  \min\Bigl\{  1/2\cdot \bigl[ 1-\rho^2(L) \bigr]^{-1} \sigma_{\min}^{-2} (\Sigma_z), \bigl[ 1-\rho(L) \bigr]^{-1} \sigma_{\min}^{-1} (\Sigma_z)\Bigr\}. 
\#
Now it remains to characterize $\sigma_{\min}(\Sigma_z)$.  For any vectors $s\in\RR^m$ and $r\in\RR^k$, we have
\#\label{eq:quad_sigmaz}
\begin{pmatrix}
s\\
r
\end{pmatrix}^\top  \Sigma_z \begin{pmatrix}
s\\
r
\end{pmatrix} & = \EE_{x\sim\mathcal N(\mu_{K,b}, \Phi_K), u\sim \pi_{K,b}(\cdot\given x)  }\Bigl\{ \bigl[ s^\top (x-\mu_{K,b}) + r^\top (u+K\mu_{K,b} - b) \bigr]^2 \Bigr\} \notag\\
& = \EE_{x\sim\mathcal N(\mu_{K,b}, \Phi_K), \eta\sim\mathcal N(0,I)}\Bigl\{ \bigl[ (s-K^\top r)^\top (x-\mu_{K,b}) + \sigma r^\top \eta \bigr]^2 \Bigr\}\notag\\
& = \EE_{x\sim\mathcal N(\mu_{K,b}, \Phi_K)}\Bigl\{ \bigl[ (s-K^\top r)^\top (x-\mu_{K,b})  \bigr]^2 \Bigr\} + \EE_{\eta\sim\mathcal N(0,I)} \bigl[ ( \sigma r^\top \eta)^2 \bigr]. 
\#
The first term on the RHS of \eqref{eq:quad_sigmaz} is lower bounded  as 
\#\label{eq:quad_sigmaz1}
& \EE_{x\sim\mathcal N(\mu_{K,b}, \Phi_K)}\Bigl\{ \bigl[ (s-K^\top r)^\top (x-\mu_{K,b})  \bigr]^2 \Bigr\} = (s-K^\top r)^\top \Phi_K (s-K^\top r)\notag\\
&\qquad \geq \|s-K^\top r\|_2^2\cdot \sigma_{\min}(\Phi_K) \geq \|s-K^\top r\|_2^2\cdot \sigma_{\min}(\Psi_\omega),
\#
where the last inequality comes from the fact that $\sigma_{\min}(\Phi_K)\geq \sigma_{\min}(\Psi_\omega)$ by \eqref{eq:f2p}. The second term on the RHS of \eqref{eq:quad_sigmaz} takes the form of 
\#\label{eq:what2}
\EE_{\eta\sim\mathcal N(0,I)} \bigl[ ( \sigma r^\top \eta)^2 \bigr] = \sigma^2 \|r\|_2^2.
\# 
Therefore, combining \eqref{eq:quad_sigmaz}, \eqref{eq:quad_sigmaz1}, and \eqref{eq:what2}, we have
\$
\begin{pmatrix}
s\\
r
\end{pmatrix}^\top  \Sigma_z \begin{pmatrix}
s\\
r
\end{pmatrix} & \geq \|s-K^\top r\|_2^2\cdot \sigma_{\min}(\Psi_\omega) + \sigma^2 \|r\|_2^2 \notag\\
&\geq \sigma_{\min}(\Psi_\omega)\cdot \|s\|_2^2 + \bigl[\sigma^2 -  \|K\|_2^2\cdot \sigma_{\min}(\Psi_\omega) \bigr]\cdot \|r\|_2^2. 
\$
From this, we know that 
\#\label{eq:what3}
\sigma_{\min}(\Sigma_z) \geq \min\bigl\{ \sigma_{\min}(\Psi_\omega), \sigma^2 -  \|K\|_2^2\cdot \sigma_{\min}(\Psi_\omega)  \bigr\}.
\#
Thus, combining \eqref{eq:norm_tilde2} and \eqref{eq:what3}, we know that $\|\tilde \Theta_{K,b}^{-1}\|_2$ is upper bounded by a constant $\tilde \lambda_K$, where $\tilde \lambda_K$ only depends on $\|K\|_2$ and $\rho(L) = \rho(A-BK)$.   This finishes the proof of the lemma. 
\end{proof}

\subsection{Proof of Lemma \ref{lemma:form_c_psi}}\label{proof:lemma:form_c_psi}
\begin{proof}
First, note that the cost function $c(x,u)$ takes the following form,
\$
c(x,u) = \psi(x,u)^\top \begin{pmatrix}
\svec\bigl[ \diag(Q,R) \bigr]\\
2Q\mu_{K,b}\\
2R\mu_{K,b}^u
\end{pmatrix} + \bigl[\mu_{K,b}^\top Q\mu_{K,b} + (\mu_{K,b}^u)^\top R\mu_{K,b}^u + \mu^\top \overline Q\mu\bigr]. 
\$
For any matrix $V$ and vectors $v_x$, $v_u$, it holds that
\#\label{eq:d1pd2}
& \EE_{\pi_{K,b}} \bigl[ c(x,u)\psi(x,u) \bigr]^\top \begin{pmatrix}
\svec(V)\\
v_x\\
v_u
\end{pmatrix}\notag \\
&\qquad = \underbrace{\EE_{\pi_{K, b}}\vast\{  \psi(x,u)^\top\begin{pmatrix}
\svec\bigl[\diag(Q,R) \bigr]\\
2Q\mu_{K,b}\\
2R\mu_{K,b}^u
\end{pmatrix} \psi(x,u)^\top \begin{pmatrix}
\svec(V)\\
v_x\\
v_u
\end{pmatrix}  \vast\}}_{D_1} \\
&\qquad\qquad+ \underbrace{\EE_{\pi_{K, b}}\vast\{\psi(x,u)^\top (\mu_{K,b}^\top Q\mu_{K,b} + (\mu_{K,b}^u)^\top R\mu_{K,b}^u + \mu^\top \overline Q\mu)  \begin{pmatrix}
\svec(V)\\
v_x\\
v_u
\end{pmatrix}\vast\}}_{D_2}. \notag
\#
In the sequel, we calculate $D_1$ and $D_2$ respectively. 

\vskip5pt
\noindent \textbf{Calculation of $D_1$.} Note that by the definition of $\psi(x,u)$ in \eqref{eq:def_feature}, it holds that 
\#\label{eq:D1_1}
D_1&  = \EE_{\pi_{K, b}} \vast\{\Biggl [ (z-\mu_z)^\top \diag(Q,R) (z-\mu_z)  +(z-\mu_z)^\top \begin{pmatrix}
2Q\mu_{K,b}\\
2R\mu_{K,b}^u
\end{pmatrix} \Biggr ]\notag\\
&\qquad\qquad \cdot  \Biggl [ (z-\mu_z)^\top V (z-\mu_z)  +(z-\mu_z)^\top \begin{pmatrix}
v_x\\
v_u
\end{pmatrix} \Biggr ]   \vast\}\notag\\
& = \EE_{\pi_{K, b}} \bigl [ (z-\mu_z)^\top \diag(Q,R) (z-\mu_z) \cdot (z-\mu_z)^\top V (z-\mu_z) \bigr]  \\
& \qquad + \EE_{\pi_{K, b}} \Biggl[ \begin{pmatrix}
2Q\mu_{K,b}\\
2R\mu_{K,b}^u
\end{pmatrix}^\top (z-\mu_z)  (z-\mu_z)^\top \begin{pmatrix}
v_x\\
v_u
\end{pmatrix}  \Biggr]. \notag
\#
Here $z = (x^\top, u^\top)^\top$ and $\mu_z = \EE_{\pi_{K,b}}(z)$.  For the first term on the RHS of \eqref{eq:D1_1}, note that $z-\mu_z\sim \mathcal N(0, \Sigma_z)$. Therefore, by Lemma \ref{lemma:magnus}, we obtain that
\#\label{eq:ipp1}
& \EE_{\pi_{K, b}} \bigl [ (z-\mu_z)^\top \diag(Q,R) (z-\mu_z) \cdot (z-\mu_z)^\top V (z-\mu_z) \bigr]\notag\\
& \qquad = 2\bigl\la \Sigma_z \diag(Q,R) \Sigma_z, V\bigr\ra + \bigl\la \Sigma_z, \diag(Q, R)\bigr \ra \cdot \la \Sigma_z, V\ra\notag\\
& \qquad = \svec\Bigl[ 2 \Sigma_z \diag(Q,R)\Sigma_z +  \bigl\la \Sigma_z, \diag(Q, R)\bigr \ra \cdot  \Sigma_z  \Bigr]^\top \svec(V). 
\#
Meanwhile, the second term on the RHS of \eqref{eq:D1_1} takes the form of
\#\label{eq:ipp2}
 \EE_{\pi_{K, b}} \Biggl[ \begin{pmatrix}
2Q\mu_{K,b}\\
2R\mu_{K,b}^u
\end{pmatrix}^\top (z-\mu_z)  (z-\mu_z)^\top \begin{pmatrix}
v_x\\
v_u
\end{pmatrix}  \Biggr] = \Biggl[\Sigma_z \begin{pmatrix}
2Q\mu_{K,b}\\
2R\mu_{K,b}^u
\end{pmatrix}\Biggr]^\top \begin{pmatrix}
v_x\\
v_u
\end{pmatrix}. 
\#
Combining \eqref{eq:D1_1}, \eqref{eq:ipp1}, and \eqref{eq:ipp2}, we obtain that 
\#\label{eq:D1}
D_1 = \begin{pmatrix}
2\svec\bigl[ \Sigma_z \diag(Q,R)\Sigma_z + \la \Sigma_z, \diag(Q,R)\ra \Sigma_z \bigr]\\
\Sigma_z\begin{pmatrix}
2Q\mu_{K,b}\\
2R\mu_{K,b}^u
\end{pmatrix}
\end{pmatrix}^\top \begin{pmatrix}
\svec(V)\\
v_x\\
v_u
\end{pmatrix}. 
\#

\vskip5pt
\noindent\textbf{Calculation of $D_2$.} By the definition of the feature vector $\psi(x,u)$ in \eqref{eq:def_feature}, we know that 
\#\label{eq:D2}
D_2 =  (\mu_{K,b}^\top Q\mu_{K,b} + (\mu_{K,b}^u)^\top R\mu_{K,b}^u + \mu^\top \overline Q\mu) \begin{pmatrix}
\svec(\Sigma_z)\\
\mathbf{0}_{m}\\
\mathbf{0}_{k}
\end{pmatrix}^\top \begin{pmatrix}
\svec(V)\\
v_x\\
v_u
\end{pmatrix}. 
\#
Now, combining \eqref{eq:d1pd2}, \eqref{eq:D1}, and \eqref{eq:D2}, it holds that
\$
\EE_{\pi_{K, b}} \bigl[ c(x,u)\psi(x,u)\bigr ]& = \begin{pmatrix}
2\svec\bigl[ \Sigma_z \diag(Q,R)\Sigma_z + \la \Sigma_z, \diag(Q,R)\ra \Sigma_z \bigr]\\
\Sigma_z\begin{pmatrix}
2Q\mu_{K,b}\\
2R\mu_{K,b}^u
\end{pmatrix}
\end{pmatrix}\notag\\
&\qquad + \bigl[\mu_{K,b}^\top Q\mu_{K,b} + (\mu_{K,b}^u)^\top R\mu_{K,b}^u + \mu^\top \overline Q\mu\bigr]\begin{pmatrix}
\svec(\Sigma_z)\\
\mathbf{0}_{m}\\
\mathbf{0}_{k}
\end{pmatrix},
\$ 
which concludes the proof of the lemma. 
\end{proof}

\section{Auxiliary Results}

\begin{lemma}\label{lemma:magnus}
Assume that the random variable $w\sim\mathcal N(0,I)$, and let $U$ and $V$ be two symmetric matrices, then it holds that
\$
\EE(w^\top U w \cdot w^\top V w) = 2\tr(UV) + \tr(U)\cdot \tr(V). 
\$
\end{lemma}
\begin{proof}
See \cite{magnus1978moments} and \cite{magnus1979expectation} for a detailed proof. 
\end{proof}

\begin{lemma}\label{lemma:alizadeh1998primal}
Let $M$, $N$ be commuting symmetric matrices, and let $\alpha_1, \ldots, \alpha_n$, $\beta_1, \ldots, \beta_n$ denote their eigenvalues with $v_1, \ldots, v_n$ a common basis of orthogonal eigenvectors. Then the $n(n+1)/2$ eigenvalues of $M\otimes_s N$ are given by $(\alpha_i\beta_j + \alpha_j\beta_i)/2$, where $1\leq i\leq j\leq n$. 
\end{lemma}
\begin{proof}
See Lemma 2 in \cite{alizadeh1998primal} for a detailed proof. 
\end{proof}

\begin{lemma}\label{lemma:hwineq}
For any integer $m>0$, let $A\in\RR^{m\times m}$ and $\eta\sim \mathcal N(0, I_m)$. Then, there exists some absolute constant $C>0$ such that for any $t\geq 0$, we have
\$
\PP\Bigl[ \bigl| \eta^\top A\eta - \EE(\eta^\top A\eta)   \bigr| > t \Bigr]  \leq 2\cdot \exp\Bigl[ -C\cdot \min\bigl(  t^2 \|A\|_\F^{-2}, ~t\|A\|_2^{-1}  \bigr) \Bigr]. 
\$
\end{lemma}
\begin{proof}
See \cite{rudelson2013hanson} for a detailed proof. 
\end{proof}

\end{document}